\documentclass[reqno]{amsart}
\usepackage{amsmath}
\usepackage[dvips]{graphicx}
\usepackage{amsfonts}
\usepackage{amssymb}
\usepackage{latexsym}
\graphicspath{{Img/}}
\newtheorem{theorem}{Theorem}
\theoremstyle{plain}

\newtheorem{corollary}{Corollary}

\newtheorem{definition}{Definition}
\newtheorem{example}{Example}

\newtheorem{lemma}{Lemma}

\newtheorem{problem}{Problem}
\newtheorem{proposition}{Proposition}
\newtheorem{remark}{Remark}

\numberwithin{equation}{section}

\newcommand{\abs}[1]{\left\lvert#1\right\rvert}

\begin{document}

\title[The Weighted Fermat-Steiner-Frechet multitree]{A Lagrangian program detecting the Weighted Fermat-Steiner-Frechet multitree for a Frechet $N-$multisimplex in the $N-$dimensional Euclidean Space}
\author{Anastasios N. Zachos}
\address{University of Patras, Department of Mathematics, GR-26500 Rion, Greece}
\email{azachos@gmail.com}
\begin{abstract}
In this paper, we introduce the Fermat-Steiner-Frechet problem for a given $\frac{N(N+1)}{2}-$tuple of positive real numbers determining the edge lengths of an $N-$simplex in $\mathbb{R}^{N},$ in order to study its solution called the "Fermat-Steiner-Frechet multitree," which consist of a union of Fermat-Steiner trees for all derived pairwise incongruent $N-$simplexes in the sense of Blumenthal, Herzog for $N=3$ and Dekster-Wilker for $N\ge 3.$ We obtain a method to determine the Fermat-Steiner Frechet multitree in $\mathbb{R}^{N}$  based on the theory of Lagrange multipliers, whose equality constraints depend on $N-1$ independent solutions of the inverse weighted Fermat problem for an $N-$simplex in $\mathbb{R}^{N}.$ A fundamental application of the Lagrangian program for the Fermat-Steiner Frechet problem in $\mathbb{R}^{N}$ is the detection of the Fermat-Steiner tree with global minimum length having $N-1$ equally weighted Fermat-Steiner points among $\frac{[\frac{N(N+1)}{2}]!}{(N+1)!}$ incongruent $N-$simplexes determined by an $\frac{N(N+1)}{2}-$tuple of consecutive natural numbers controlled by Dekster-Wilker, Blumenthal-Herzog conditions and enriched with the fundamental evolutionary processes of Nature (Minimum communication networks, minimum mass trasfer, maximum volume of incongruent simplexes). Furthermore, we obtain the unique solution of the inverse weighted Fermat problem, referring to the unique set of $(N+1)$ weights, which correspond to the vertices of an $N-$ boundary simplex in $\mathbb{R}^{N}.$ Additionaly, we give a negative answer for an intermediate weighted Fermat-Steiner-Frechet multitree having one node (weighted Fermat point) for $m$ boundary closed polytopes, ($m\ge N+2$), which is determined by $m$ prescribed rays meeting at a fixed weighted Fermat point, by deriving a linear dependence for the $m$ variable weights (Plasticity of an Intermediate Fermat-Steiner-Frechet multitree for $m$ boundary closed polytopes). By enriching the plasticity of an intermediate Fermat-Steiner-Frechet multitree for $m$ boundary closed polytopes with a two-way mass transport from $k$ vertices to the unique weighted Fermat point and from this point to the $m-k$ vertices and reversely, we derive the equations of "mutation" of intermediate Fermat-Steiner-Frechet multitree for $m$ boundary closed polytopes in $\mathbb{R}^{N}.$
Finally, we apply the unique solution of the inverse weighted Fermat problem for an $N-$simplex in $\mathbb{R}^{N},$ in order to construct an $\epsilon$ approximation for the weights which correspond to the vertices of a Frechet $N-$multisimplex.
\end{abstract}\maketitle

\section{Introduction}

The problem considered in the present paper is a result of synthesis of two problems:
The weighted Fermat-Steiner problem on the shortest networks for boundary $N-$simplexes in $\mathbb{R}^{N}$ and Frechet's problem on seeking the necessary and sufficient conditions for a given $\frac{N(N+1)}{2}$-tuple of positive real numbers determining the edge lengths of an $N$-simplex in $\mathbb{R}^{N}$ for $N\ge 3.$

The weighted Fermat-Steiner problem for a boundary $N-$simplex is a problem on finding a shortest network of total weighted length connecting $N+1$ non-collinear and non-coplanar points possessing each of them a positive real number (weight) in $\mathbb{R}^{N}.$

In \cite{Zachos:21}, we find the solution of the weighted Fermat-Steiner problem for a tetrahedron $A_{1}A_{2}A_{3}A_{4}$ ($N=3$). The solution is a tree $T$ having two weighted points $A_{0}$ and $A_{0}^{\prime}$ (nodes) inside the tetrahedron. Each node is a weighted Fermat-Steiner (or Torricelli) point. We call $T$ a weighted Steiner tree. $T$ is the union of one line segment between the two weighted Fermat-Steiner points and of $4$ line segments joining each weighted Fermat-Steiner point with two fixed neighboring vertices. In \cite{RubinsteinThomasWeng:02}, Rubinstein, Thomas and Weng solved the unweighted Fermat-Steiner problem for tetrahedra in $\mathbb{R}^{3}.$ The weighted Fermat-Steiner problem for four points in the Euclidean plane is a problem of practical importance first considered by Gauss in a letter to Schumacher, which deal with the construction of a railway network interconnecting four German cities (\cite{Gauss:73}). The weight of each vertex may be considered as the population of each city. If one of the four weights equals zero and the other three weights are equal, we obtain an equally weighted Fermat-Steiner problem for three points in $\mathbb{R}^{2},$ we derive the original problem posed by Fermat (\cite{Fermat:29}).
It is worth mentioning that Fermat proved Snell's law of retraction of light, which requires the computation of the minimum of a function of one variable as a sum of two weighted distances (the third weight is zero), such that the first weight is inversely proportional to the velocity of propagation of light in the upper medium and the second weight is inversely proportional to the velocity of propagation of light in the lower one (see in \cite[p.~21]{Tikhomirov:90}). In \cite{JarnikKossler:34}, Jarnik and Kossler introduced a generalization of the Fermat problem on optimal networks of a finite set of points in $\mathbb{R}^{2}.$ In \cite{CourantRobbins:96}, Courant and Robbins restated and named the problem as the Steiner problem. In \cite{IvanovTuzhilin:95}, \cite{IvanovTuzhilin:01b}, Ivanov and Tuzhilin studied properties of weighted minimal binary trees, which is an important generalization of weighted Fermat-Steiner trees in $\mathbb{R}^{2}.$

The Frechet problem in $\mathbb{R}^{3}$ states that (\cite{Frechet:35}):
Given a sextuple of positive real numbers $a_{ij}=a_{ji},$ for $i,j=1,2,3,4.$ What are the necessary and sufficient conditions that they be the lengths of the tetrahedron $A_{1}A_{2}A_{3}A_{4}$ in $\mathbb{R}^{3}$?

In \cite[Theorem~2.1]{Blumenthal:59}, Blumenthal proved that a sextuple of positive real numbers $(a+nd)^{1/2},$ $n=0,1,2,3,4,5,$ $0< 4d\le a$ form a completely tetrahedral sextuple, which yields thirty incongruent tetrahedra in $\mathbb{R}^{3}.$

In \cite[Theorem~3]{Fritz:59}, Hertog proved a theorem characterizing complete tetrahedral sextuples, which states that: Six positive real numbers $\{a,b,c,d,e,f\}$ satisfying the conditions $a>b>c>d>e>f>0$ and $e+f\ge a$ form a complete tetrahedral sextuple if and only if the Caley-Menger determinant $D(a,b;c,f;e)\ge 0,$ (the edges separated by semicolons are pairs of opposite edges of the tetrhaedron) that is if and only if the tetrahedron in which the faces are formed by the triples $(a,c,d),$ $(a,e,f),$ $(b,c,e)$ and $(b,d,f)$ is realizable in $\mathbb{R}^{3}.$

In \cite[Remark~4,~5,~6]{Fritz:59}, Hertog gave the following three important applications of his theorem:

i. Six consecutive positive integers $x+5,x+4,x+3,x+2,x+1,x$ form a complete tetrahedral sextuple
 if and only if $x$ is greater than or equal to the positive root $\rho$ of the equation $D(x+5,x+4;x+3,x;x+2,x+1)=0:$ $6.09<\rho<6.10$ (Hertog sextuple of consecutive positive integers)

ii. The sextuple of positive real numbers $\{1,1-\delta; 1-2\delta, \epsilon;1-3\delta,1-4\delta \}$ for $\delta=\delta(\epsilon)>0,$ $0<\epsilon<1,$ such that $1-4\delta >\epsilon$ forms a complete tetrahedral sextuple

iii. Blumenthal's sextuple $\{a+nd)^{1/2},\}$ $n=0,1,2,3,4,5$ is a complete tetrahedral sextuple if and only if $\frac{a}{d}$ is greater than or equal to a cubic irrationality $r:$ $1.91<r<1.92.$

In \cite{DeksterWilker:87}, Dekster and Wilker derived the condition $\frac{x}{x+5}\ge \lambda(3)\equiv \frac{1}{\sqrt{2}},$ which gives that $x\ge 13,$ which leaves six runs from Hertog result sextuple of consecutive integers that works for $x\ge 7.$

We introduce the Fermat-Steiner-Frechet problem for a given\\ $\frac{N(N+1)}{2}-$tuple of positive real numbers determining $N-$ simplexes in $\mathbb{R}^{N}.$

\begin{problem}[The weighted Fermat-Steiner-Frechet problem in $\mathbb{R}^{N}$]
Given a $2N-$tuple of positive real numbers (weights), such that $N-1$ of them are equal and a $\frac{N(N+1)}{2}$ of positive real numbers determining the edge lengths of $N-$simplexes in $\mathbb{R}^{N},$ find the shortest networks (weighted Fermat-Steiner trees) of total weighted length for all derived incongruent $N-$simplexes in $\mathbb{R}^{N}.$
\end{problem}

We call Fermat-Steiner Frechet multitree the solution of the weighted Fermat-Steiner-Frechet problem in $\mathbb{R}^{N},$ which is a union of the corresponding weighted Fermat-Steiner trees for all derived incongruent $N-$simplexes in $\mathbb{R}^{N}$ and Frechet $N-$multisimplex the class of incongruent $N-$simplexes derived by the same given $\frac{N(N+1)}{2}-$tuple of positive real numbers in $\mathbb{R}^{N}.$

In \cite{Zachos:16}, we solve the weighted Fermat-Frechet problem for $N=3$ in $\mathbb{R}^{3}$ and we find the condition to locate the corresponding weighted Fermat trees, by applying a generalization of the cosine law in $\mathbb{R}^{3},$ which depend only on edge lengths and the Caley-Menger determinant representing tetrahedral volume w.r to edge lengths.

In this paper, we obtain a Lagrangian program to detect a weighted Fermat-Steiner multitree for a boundary $N-$Frechet multisimplex in $\mathbb{R}^{N},$ which
gives:

(i) a characterization of the most natural of natural numbers obtained by a $\frac{N(N+1)}{2}-$tuple of consecutive natural numbers
\\
(ii)The plasticity of weighted Fermat-Steiner multitree having one node (intermediate multitrees) for boundary $m-$ closed polytopes in $\mathbb{R}^{N}$
\\
(iii) )The Bessel (random) plasticity of weighted Fermat-Steiner multitree having one node (intermediate multitrees) for boundary $m-$ closed polytopes in $\mathbb{R}^{N}$
\\
(iv) An $\epsilon$ approximation method for the weights of a Frechet $N-$multisimplex in $\mathbb{R}^{N}.$

Our main results are:

1. A Lagrangian program, which detects the weighted Fermat-Steiner-Frechet multitree for a given $\frac{N(N+1)}{2}$ in $\mathbb{R}^{N}$ for $N\ge 3$

\emph{Theorem~\ref{Lagrangerulemultitreern} Lagrange multiplier rule for the weighted Fermat-Steiner Frechet multitree in $\mathbb{R}^{N}$}
If the admissible point $\tilde{x}_{i}$ yields a weighted minimum multitree for $1\le i \le \frac{\frac{1}{2}N(N+1)!}{(N+1)!},$ which correspond to a Frechet $N-$multisimplex derived by a $\frac{N(N+1)}{2}-$tuple of edge lengths determining upto $\frac{\frac{1}{2}N(N+1)!}{(N+1)!}$ incongruent $N-$simplexes constructed by the Dekster-Wilker domain $DW_{\mathbb{R}^{N}}(\ell,s)$, then there are numbers $\lambda_{0i},\lambda_{1i},\lambda_{2i},\ldots \lambda_{(N^2-1)i},$ (components of the Lagrangian vector) such that:

\[\frac{\partial \mathcal{L}_{i}(\tilde{x}_{i},\tilde{\lambda}_{i})}{\partial x_{ji}}=0,\]
for $j=1,2,\ldots,2N(N-1),$

\[\tilde{x}_{i}= \{a_{(0,1),1},a_{(0,1),2},a_{(0,1),3},\beta_{(0,1),4},\ldots,\beta_{(0,1),N},\ldots,
a_{(0,N-1),1},a_{(0,N-1),2},\]\[a_{(0,N-1),3},\beta_{(0,N-1),4},\ldots,\beta_{(0,N-11),N},\]\[ w_{(0,1),1},\ldots, w_{(0,1),N},w_{(0,2),1},\ldots, w_{(0,2),N},\ldots,\] \[w_{(0,N-1),1},\ldots, w_{(0,N-1),N}\} \]

$\tilde{\lambda}_{i}=\{\lambda_{0},\lambda_{1},\ldots,\lambda_{N^2-1}\}\}$

and \[\mathcal{L}_{i}(\tilde{x}_{i},\tilde{\lambda}_{i})\] is a Lagrangian function.

Theorems~\ref{Lagrangerulemultitree},\ref{Lagrangerulemultitreer4} are particular cases for $N=3$ and $N=4,$ respectively.
We note that we mention these two particular cases, in order to explain how the Schlafli angle formed by the normals of two planes is eliminated and vanishes from the computation of the Lagrangian function in $\mathbb{R}^{3}$ and how the Schlafli angle formed by the normals of a plane and a hyperplane
cannot be eliminated and is embodied in the computation of the Lagrangian function in $\mathbb{R}^{4}.$ Hence, $N-3$ Schlafli angles are embodied in the computation of the Lagrangian function in $\mathbb{R}^{N}.$

2. A characterization of the most natural of natural numbers obtained by a $\frac{N(N+1)}{2}-$tuple of consecutive natural numbers

\emph{Theorem~\ref{mostnaturalconsecutivetentuplesrn}}
The most natural $\frac{N(N+1)}{2}-$tuple of numbers from $\frac{N(N+1)}{2}$ consecutive natural numbers $\{a+\frac{N(N+1)}{2}-1,\ldots,a+1,a,\}$
for $a\ge a(N)$ is a $\frac{N(N+1)}{2}-$tuple of edge lengths having the maximum volume (maximum $\frac{N(N+1)}{2}-$tuple) among the $\frac{\frac{1}{2}N(N+1)!}{(N+1)!},$ incongruent $N-$simplexes, which corresponds to a Fermat-Steiner tree of minimum total weighted length (global minimum solution), such that the upper bound for the weight $B_{ST}$ is determined by the rest Fermat-Steiner minimal trees having larger or equal weighted minimal total length.

The function $a_{N}$ is derived by the Dekster-Wilker function $\ell (N)$ (Section~8).

We note that for $N=3,$ we use Blumenthal-Herzog sextuples of positive real numbers determining the edge lengths of 30 incongruent tetrahedra (Frechet multitetrahedron), in order to detect the most natural of natural numbers obtained by a consecutive sextuple of natural numbers whose least element $a\ge 7$ (Theorem~\ref{mostnaturalconsecutivesextuples}). The corresponding Dekster-Wilker sextuples detects the most natural of natural numbers for $a\ge 13.$

For $N=4,$ we use Dekster-Wilker tentuples of positive real numbers determining the edge lengths of 30.240 incongruent $4$-simplexes (Frechet $4-$multisimplex), in order to detect the most natural of natural numbers obtained by a consecutive sextuple of natural numbers whose least element $a\ge 30$ (Theorem~\ref{mostnaturalconsecutivetentuplesr4}).

3. A theoretical construction of a weighted Fermat-Steiner Frechet multitree for a Frechet multitetrahedron in $\mathbb{R}^{3},$ using two variable dihedral angles, which focus on an auxiliary construction of five points that lie on the same circle.

\emph{Theorem~\ref{mainresmultitreetetr2}}
There are upto $4!\cdot 30$ weighted Simpson lines defined by the points $T_{12}, T_{34},$ such that the two equally weighted Fermat-Steiner points $O_{12}, O_{34}$ in $[T_{12},T_{34}],$ which give the position of the weighted Fermat-Steiner-Frechet multitree for 30 incongruent tetrahedra determined by Blumenthal, Herzog and Dekster-Wilker sextuples of edge lengths in $\mathbb{R}^{3}.$

4. Non-random and random plasticity equations of an intermediate weighted Fermat-Steiner Frechet multitrees.

A. \emph{Theorem\ref{5inverseR4}} [Unique solution of the $4-$INVWF problem in $\mathbb{R}^{4}$]

The weight $B_{i}$ is uniquely determined by:

\[B_{i}=\frac{C}{1+\abs{\frac{\sin{\alpha_{i,k0lm}}}{\sin{\alpha_{j,k0lm}}}}+\abs{\frac{\sin{\alpha_{i,j0lm}}}{\sin{\alpha_{k,j0lm}}}}+\abs{\frac{\sin{\alpha_{i,k0jm}}}{\sin{\alpha_{l,k0jm}}}}
+\abs{\frac{\sin{\alpha_{i,k0jl}}}{\sin{\alpha_{m,k0jl}}}}},\]
for $i,j,k,l,m=1,2,3,4,5$ and $i \neq j\neq k\neq l\neq m,$
where $\alpha_{i,k0lm}$ is the angle formed by the line segment $A_{0}A_{i}$ and $A_{0},P_{i},$ where $P_{i}$ is the trace of the orthogonal projection
of $A_{i}$ to the hyperplane defined by $A_{0}A_{k}A_{l}A_{m}.$

\emph{Theorem~\ref{nplus1inverseRn}}[Solution of the $N-$INVWF problem in $\mathbb{R}^{N}$]
The weight $B_{i}$ is uniquely determined by:

\[B_{i}=\frac{C}{1+\abs{\frac{\sin{\alpha_{i,0k_{1}k_{2}\ldots k_{N-1}}}}{\sin{\alpha_{k_{N},0k_{1}k_{2}\ldots k_{N-1}}}}}+\abs{\frac{\sin{\alpha_{i,0k_{1}k_{2}\ldots k_{N-2}k_{n}}}}{\sin{\alpha_{k_{N-1},0k_{1}k_{2}\ldots k_{N-2}k_{N}}}}}+\ldots+\abs{\frac{\sin{\alpha_{i,0k_{2}\ldots k_{N-1}k_{N}}}}{\sin{\alpha_{k_{1},0k_{2}\ldots k_{N-1}k_{N}}}}}},\]

for $i, k_{1},k_{2},...,k_{N}=1,2,...,N+1$ and $k_{1} \neq k_{2}\neq...\neq k_{N}.$

The non-uniqueness of the $N-$INVWF problem in $\mathbb{R}^{N}$ gives the plasticity equations of non-random plasticity and Bessel(random) plasticity
of intermediate weighted Fermat-Steiner Frechet multitrees having one node (weighted Fermat point) for boundary $m-$ closed polytopes in $\mathbb{R}^{N}$
for $m\ge N+2$ (Theorems~\ref{theorplastmpol} and ~\ref{Besselpasticitypolytopes}, respectively).

\emph{Theorem~\ref{plasticitymultitreern}}
An increase to the weight that corresponds to the $(N+2)th$ ray causes a decrease to the weight that corresponds to the $(N+1)th$ ray and a variation to the weight that corresponds to the $ith$ ray depends on the difference $(B_{i})_{123\ldots N(N+1)}-(B_{i})_{123\ldots N(N+2)},$ such that the geometric structure of the weighted Fermat-Frechet multitree with respect to a boundary $N-$multisimplex, remains the same, for $i=1,2,\ldots N,$ $1\le k \le \frac{\frac{1}{2}N(N+1)!}{(N+1)!}.$

The Bessel plasticity of multitrees is derived by adding a random weight, which follows a Bessel motion on a new growing branch starting from the weighted Fermat point $A_{0},$ i.e $A_{0}A_{N+2}$ (Theorem~\ref{Besselpasticitypolytopes}).

B.

\emph{Theorem~\ref{RatioVolumesSimplex}, weighted volume equalities in $\mathbb{R}^{N}$}
\[\frac{B_{1}}{a_{1}Vol(A_{0}A_{2}A_{3}...A_{N+1})}=\frac{B_{2}}{a_{2}Vol(A_{1}A_{0}...A_{N+1})}=...=\]
\[=\frac{B_{N+1}}{a_{N+1}Vol(A_{1}A_{2}...A_{0})}=\frac{\sum_{i=1}^{N+1}\frac{B_{i}}{a_{i}}}{Vol(A_{1}A_{2}...A_{N+1})}.\]

The weighted volume equalities are important to compute the weighted Fermat point of an $N-$simplex in $\mathbb{R}^{N}$ and to construct the Lagrangian function of a weighted Fermat-Steiner multitree for the Frechet $N-$multisimplex in $\mathbb{R}^{N-1},$ by taking into account a system of variable weighted volume equalities derived by $(N-1)$ weighted Fermat points located inside each boundary $N-$simplex derived by the Frechet $N-$multisimplex.

C. By adding an optimal mass transport of a two-way communication network along $k$ directions in the plasticity of non-random plasticity
of intermediate weighted Fermat-Steiner Frechet multitrees having one weighted Fermat point, for boundary $m-$ closed polytopes in $\mathbb{R}^{N}$
for $m\ge N+2,$ we derive the $(m,k)$ "mutation" of intermediate weighted Fermat-Steiner Frechet multitrees having one weighted Fermat point, for boundary $m-$ closed polytopes in $\mathbb{R}^{N}$ (Theorem~\ref{mutationamultitrees}).

5.  Constructive weights for the vertices of an $N-$ Frechet multisimplex in $\mathbb{R}^{N}.$

An $\epsilon$ approximation of the value of the weight $B_{N+1},$ which corresponds to the vertex $A_{N+1}$ of an $N-$ simplex in $\mathbb{R}^{N}$ circumscribed in a $(N-1)$sphere $S^{N-1}$ of radius $r,$ and center $O,$ is given by Theorem~\ref{constructiveweights}, by applying the $N-$INWF problem in $\mathbb{R}^{N}.$ For $\epsilon \to 0,$ the limiting floating weighted Fermat tree solution coincides with the absorbing weighted Fermat tree solution of Theorem~\ref{theor1}.Therefore, we get an approximation for the weights of a Frechet $N-$multisimplex via an $\epsilon$ approximation of each corner of incongruent $N-$simplexes with a multiweighted Fermat-Frechet multitree in $\mathbb{R}^{N}.$

In Section~2, we mention some fundamental known results concerning the existence and uniqueness of solution for the weighted Fermat-Steiner problem and the weighted Fermat problem in $\mathbb{R}^{N}$ and the uniqueness of solution of the inverse weighted Fermat-problem for a boundary tetrahedron and boundary triangles in $\mathbb{R}^{3}.$ In Section.~3, we introduce a Lagrangian program, which detects the weighted Fermat-Steiner Frechet multitree for a Frechet multitetrahedron in $\mathbb{R}^{3}.$ In Section~4, we apply this program to detect a weighted Fermat-Steiner tree with respect to a boundary tetrahedron
having the maximum volume whose edge lengths form a Herzog sextuple of six consecutive natural numbers and the global weighted minimum length. In Section~5, we give a theoretical construction of a weighted Fermat-Steiner Frechet multitree for a Frechet multitetrahedron in $\mathbb{R}^{3},$ using two variable dihedral angles.

In Section~6, we study the uniqueness of solution of the inverse weighted Fermat problem for $N-$simplexes in $\mathbb{R}^{N}$ and the plasticity solutions of a weighted Fermat-Frechet multitree with one node (weighted Fermat point) for $m-$ boundary closed polytopes in $\mathbb{R}^{N}.$ In Sections~7,~8, we apply the uniqueness of the inverse weighted Fermat problem for $N-$simplexes in $\mathbb{R}^{N}$ to a Lagrangian program, in order to detect the weighted Fermat-Steiner Frechet multitree for an $N-$ Frechetmultisimplex for $N=4$ and $N>4,$ respectively. In Section~9, we extend the Lagrangian program for intermediate weighted Fermat-Steiner-Frechet multitrees for an $N-$ Frechetmultisimplex in $\mathbb{R}^{N}$ for an $N-$Frechet multisimplex whose $\#$ of nodes is less than $N-1.$ In Section~10, we give the "mutation" equations of an intermediate weighted Fermat-Frechet multitree having one node for $m-$boundary closed polytopes in $\mathbb{R}^{N},$ which is a generalization of the plasticity solutions given in section~5. In Section~12, we construct the weights of a Frechet $N-$multisimplex in $\mathbb{R}^{N},$ by introducing a method of $\epsilon$ approximation of variable weighted multitrees in $\mathbb{R}^{N}.$ In Section~13, we introduce the Bessel plasticity equations of an intermediate weighted Fermat-Steiner Frechet multitree having one node for $m-$boundary closed polytopes in $\mathbb{R}^{N},$ such that one of the weights follows a Bessel motion. In the final section, we conclude with two open questions, which deal with a "partition" of optimal networks in $\mathbb{R}^{N}.$  

\section{Prelimiminaries: The weighted Fermat-Steiner problem and the inverse weighted Fermat problem for tetrahedra in $\mathbb{R}^{3}$}
In this section, we state the weighted Fermat-Steiner problem for an $N-$simplex in $\mathbb{R}^{N},$ which is a generalization of the weighted Fermat problem in $\mathbb{R}^{N}$ and mention some fundamental results established by Ivanov-Tuzhilin concerning the existence and uniqueness of the Fermat-Steiner tree solutions. We also mention the characterization of solutions with respect to the weighted Fermat problem given by Sturm for the unweighted case and extended by Kupitz-Martini for the weighted case. We continue by giving some recent known results for the weighted Fermat-Steiner problem for tetrahedra in $\mathbb{R}^{3}$ regarding the position of the weighted Fermat-Steiner tree given by the author, which extends previous results given by Rubinstein-Thomas-Weng for the unweighted case. Finally, We state an inverse weighted Fermat problem for $m$ points in $\mathbb{R}^{N}$ and an explicit solution of the inverse weighted Fermat problem for $N=3$ given by Zouzoulas and the author and generalizes the solution of the inverse weighted Fermat problem for $N=2$ given by Gueron and Tessler.

We denote by $A_{0,1},A_{0,2},\ldots A_{0,N-1}$ $N-1$ points inside the $N-$simplex $A_{1}A_{2}\ldots A_{N+1}$ in $\mathbb{R}^{N},$ by $b_{i}$ a positive real number(weight), which corresponds to each vertex $A_{i},$ by $a_{(0,i),(0,j}$ the length of the line segment $A_{0,i}A_{0,j}$ and by $a_{k0}$ the length of the line segment $A_{k}A_{0},$ for $i,j=1,2,\ldots, N-1,$ $k=1,2,\ldots, N+1.$

\begin{problem}[The weighted Fermat-Steiner problem for $A_{1}A_{2}\ldots A_{N+1}$ in $\mathbb{R}^{N}$]\label{WFSN}
Find $A_{0,i}$ in $\mathbb{R}^{N},$ for $i=1,2\ldots, N-1$ with equal given weight $B_{ST},$ such that:

\begin{equation}\label{objectivewfsn}
f(\{A_{0,i}\})=\sum_{i=1,\- j_{i} \in \{1,2,\ldots,N-1\} }^{N+1}b_{i}a_{i,(0,j_{i})}+b_{ST}\sum_{i=1,\\
j_{i}\in \{1,2,\ldots,N-1\}}^{N-1}a_{(0,i),(0,j_{i})} \to min.
\end{equation}

\end{problem}

In \cite[Corollary~1.1, Proposition~1.2]{IvanovTuzhlin:92},\cite[Theorem~1.1, Corollary~1.2, Chapter~3]{IvanTuzh:094}, \cite{IvanovTuzhilin:95},\cite{IvanovTuzhilin:01b},  Ivanov and Tuzhilin studied the weighted Fermat-Steiner problem and many variations starting from an embedding of a Steiner topological graph (one dimensional cell complex) on a manifold $W.$

The following definitions of an immersed parametric network are given in \cite[Definitions,p.~56,p.~58]{IvanTuzh:094}:
A parametric network $\Gamma$ of type $G$ in the manifold $W$ is an arbitrary mapping $\phi: G\to W.$
A network $\Gamma=\{\phi:G \to W \}$ is called piecewise smooth if each of its edges is a piecewise curve. A piecewise smooth curved is called immersed if all its edges are nondegenerate. We note that a description of weighted minimum networks in $\mathbb{R}^{N}$ is given in \cite[Section~3]{IvanovTuzhilin:98}. Local minimal networks of type $G$ can be realized as weighted minimal networks type $G$ with constant weight function (positive real number) defined on the edge set of $G.$

\begin{lemma}{Weighted Local Minimality criterion, \cite[Corollary~1.2, Chapter~3]{IvanTuzh:094},\cite{IvanovTuzhilin:95}}\label{ivantuzhimp1}
Let $\Gamma$ be an immersed weighted parametric network (Weighted Fermat-Steiner network). The network is local minimal if and only if all the edges of $\Gamma$ are geodesic segments for any mobile vertex (Fermat-Steiner point) $A_{0,i}$ of $\Gamma$ the linear combination of weighted unit vectors of the directions of edges of $\Gamma$ going out of $A_{0,i}$ (with coefficients equal to the weights of these edges) vanishes.
\end{lemma}

We mention an important theorem of Ivanov-Tuzhilin (\cite{IvanTuzh:094}), which deals with the uniqueness for networks with boundaries in $\mathbb{R}^{N}.$ A strong local minimal network is a parametric network $\Gamma,$ if for any point $x$ of the parametric graph of the reduced network (without degenerate edges) corresponding to $\Gamma,$ there exists a strong local network $\Gamma^{x}_{sloc},$ such that any small deformation of the network
$\Gamma^{x}_{sloc},$ that preserves its boundary does not decrease the length of $\Gamma^{x}_{sloc}$ (see in \cite[Definitions,p.~91]{IvanTuzh:094}).

\begin{lemma}{Uniqueness theorem of Ivanov-Tuzhilin for networks with boundaries in $\mathbb{R}^{N},$ \cite[Theorem~3.1 Chpater~2, Proof of Theorem~3.1, Chapter~3]{IvanTuzh:094}}\label{ivantuzhimp2}
Let $M\subset \mathbb{R}^{N}$ be an arbitrary non-empty finite set of points from $\mathbb{R}^{N};$ denote by $G$ an acyclic topological graph and let $\beta$ be a boundary mapping that maps a subset of vertices of graph $G$ onto the set $M.$ Assume that among strong minimal networks in class $\mathcal{M}_{G}(\beta)$ there exists an embedded network $\Gamma$ which does not possess mobile vertices $A_{0,i}$ of degree two. Then all other strong minimal networks in this class coincide with $\Gamma$ (up to a parameterization); that is $\Gamma$ is unique.
\end{lemma}

By substituting $A_{0}\equiv A_{0,i},$ for $i=1,2,\ldots N-1$ in (\ref{objectivewfsn}), we derive the weighted Fermat problem $A_{1}A_{2}\ldots A_{N+1}$ in $\mathbb{R}^{N}.$

\begin{problem}[The weighted Fermat problem for $A_{1}A_{2}\ldots A_{N+1}$ in $\mathbb{R}^{N}$]\label{WFN}
Find $A_{0}$ in $\mathbb{R}^{N},$ such that:

\begin{equation}\label{objectivewfrn}
f(\{A_{0,i}\})=\sum_{i=1}^{N+1}b_{i}a_{0i}\to min.
\end{equation}

\end{problem}

In \cite{Sturm:84}, Sturm gave a complete characterization of the
solutions of the unweighted Fermat problem for $m$ given
points in $\mathbb{R}^{N},$ and Kupitz and Martini extended these characterization in the weighted case, in
\cite[Theorem~18.37, pp.~250]{BolMa/So:99}.

\begin{theorem} {\cite{Sturm:84},\cite[Theorem~18.37,pp.~250]{BolMa/So:99}}\label{theor1}

(I) The weighted Fermat point $A_{0}$ with respect
exists and is unique.

(II) If for each point  $A_{i}\in \{A_{1},A_{2},...,A_{m}\}$
\[ \|\sum_{j=1,j\ne i}^{m}b_{j}\vec {u}(A_j,A_i)\|>b_i, \]
for $i,j=1,2,\ldots,m,$ then

(a) the weighted Fermat point $A_{0}$ does not belong
to\\ $\{A_{1},A_{2},...,A_{m}\},$

(b)

\begin{equation}\label{weightedfloatingft}
 \sum_{i=1}^{m}b_{i}\vec {u}(A_0,A_i)=\vec{0}.
\end{equation}

(Weighted floating Fermat tree solution).

(III) If there is a point $A_{i}\in\{A_{1},A_{2},...,A_{m}\}$
satisfying
\[ \|{\sum_{j=1,j\ne i}^{m}b_{j}\vec {u}(A_j,A_i)}\|\le b_i.
\] for $i,j=1,2,...,m,$ then the weighted Fermat point  $A_{0} \equiv A_i$
(Weighted absorbed Fermat tree solution).

\end{theorem}

The inverse weighted Fermat problem for $m$ points (non-collinear and non-coplanar) in $\mathbb{R}^{N}$
($(m-1)-$INVWF)states that:

\begin{problem}[The $(m-1)-$INVWF problem in $\mathbb{R}^{N}$]
Given a point $A_{0}\notin \{A_{1},A_{2},...,A_{m}\},$ $A_{0},A_{i}\in \mathbb{R}^{N},$ does there exist
a unique set of positive weights $b_{i0},$ such that
\begin{eqnarray}\label{isoperimetricconditionweights}
 b_{10}+b_{20}+...+b_{m0} = c =const,
\end{eqnarray}
for which $A_{0}$
\begin{displaymath}
 f(A_{0})=\sum_{i=1}^{m}b_{i0}a_{i0}\to min.
\end{displaymath}

\end{problem}

The solution of the $2-$INVWF problem in $\mathbb{R}^{2}$ ($b_{4}=b_{5}=b_{6}=...=b_{m}=0$) is given in \cite{Gue/Tes:02}.
\begin{proposition}{Solution of the 2-INVWF problem in $\mathbb{R}^{2}$ \cite{Gue/Tes:02}}\label{3inverser2}
\begin{equation}\label{inverse11132}
b_{i}=\frac{C}{1+\frac{\sin{\alpha_{i0j}}}{\sin{\alpha_{j0k}}}+\frac{sin{\alpha_{i0k}}}{\sin{\alpha_{j0k}}}},
\end{equation}
for $i,j,k=1,2,3$ and $i \neq j\neq k.$
\end{proposition}
The weights $b_{1}, b_{2},$ $b_{3}$ depend on exactly two given angles $\alpha_{102},$ $\alpha_{203}$ because $\alpha_{103}=2\pi-\alpha_{102}-\alpha_{203}.$

\begin{problem}[The $3-$INVWF problem in $\mathbb{R}^{3}$ \cite{Zach/Zou:09}]
Given a point $A_{0}$ which belongs to the interior of $A_{1}A_{2}A_{3}A_{4}$ in $\mathbb{R}^{3},$
does there exist a unique set of positive weights $b_{i0},$ such
that:
\[b_{10}+b_{20}+b_{30}+b_{40}=C>0,\]
\[f(A_{0})=\sum_{i=1}^{4}b_{i0}a_{i0} \to min.\]

\end{problem}

The unique solution of the weighted Fermat problem for tetrahedra has been established in \cite{Zach/Zou:09}, by taking into account (\ref{tetraed1})-(\ref{tetraed4}) of Lemma~\ref{tetrahedroninv}.

\begin{lemma}[Solution of the inverse weighted Fermat problem for tetrahedra in $\mathbb{R}^{3}$ {\cite[(3.12),(3.13),p.~120]{Zach/Zou:09}}]
\begin{equation}\label{solinvtetrr3}
(\frac{b_{j0}}{b_{i0}})^{2}=\frac{\sin^{2}\alpha_{kom}-\cos^{2}\alpha_{moi}-\cos^{2}\alpha_{koi}+2\cos\alpha_{moi}\cos\alpha_{koi}\cos\alpha_{kom}}{\sin^{2}\alpha_{kom}-\cos^{2}\alpha_{moj}-\cos^{2}\alpha_{koj}+2\cos\alpha_{moj}\cos\alpha_{koj}\cos\alpha_{kom}}
\end{equation}
for $i,j,k=1,2,3,4$ and $i\ne j\ne k\ne m\ne i.$
\end{lemma}

\begin{lemma}\cite[Proposition~1]{Zachos:20}\label{lem4r32}
The ratios $\frac{b_{j0}}{b_{i0}}$
depend on exactly five given angles
$\alpha_{102},$ $\alpha_{103},$ $\alpha_{104},$ $\alpha_{203}$ and
$\alpha_{204}.$

The sixth angle $\alpha_{304}$ is calculated by the following formula:

\begin{eqnarray}\label{calcalpha3042}
&&\cos\alpha_{304}=\frac{1}{4} [4 \cos\alpha _{103} (\cos\alpha
_{104}-\cos\alpha _{102} \cos\alpha _{204})+\nonumber\\
&&+2 \left(b+2 \cos\alpha _{203} \left(-\cos\alpha _{102}
\cos\alpha _{104}+\cos\alpha _{204}\right)\right)] \csc{}^2\alpha
_{102}\nonumber\\
\end{eqnarray}

where

\begin{eqnarray}\label{calcalpha304auxvar}
b\equiv\sqrt{\prod_{i=3}^{4}\left(1+\cos\left(2 \alpha
_{102}\right)+\cos\left(2 \alpha _{10i}\right)+\cos\left(2 \alpha
_{20i}\right)-4 \cos\alpha _{102} \cos\alpha _{10i} \cos\alpha
_{20i}\right)}\nonumber\\.
\end{eqnarray}

for $i,k,m=1,2,3,4,$ and $i \ne k \ne m.$
\end{lemma}

Thus, the solution of the 3-INVWF problem in $\mathbb{R}^{3}$ depend on exactly five given angles $\alpha_{102},$ $\alpha_{103},$ $\alpha_{104},$ $\alpha_{203},$ $\alpha_{204}.$


We mention the definitions of a tree topology, the degree of a vertex, a Fermat tree
topology, a Fermat-Steiner tree  topology, an intermediate Fermat-Steiner tree topology, in order to study the
solution of the weighted Fermat-Steiner problem with respect to fixed tree topology for $A_{1}A_{2}\ldots A_{N+1}$ in $\mathbb{R}^{N}.$

\begin{definition}{\cite{GilbertPollak:68}}\label{topology}
A tree topology is a connection matrix specifying which pairs of
points from the list
$A_{1},A_{2},...,A_{N},A_{0,1},A_{0,2},...,A_{0,N-2}$ have a
connecting line segment (edge).
\end{definition}

\begin{definition}{\cite{IvanTuzh:094},\cite{Ci2}}\label{degreeSteinertree}
The degree of a vertex corresponds to the number of connections of
the vertex with line segments.
\end{definition}

\begin{definition}{\cite{GilbertPollak:68}}\label{Fermattopology}
A Fermat tree topology of a boundary $N-$simplex $A_{1}A_{2}\ldots A_{N+1}$ in $\mathbb{R}^{N}$ is a tree topology, such that each boundary vertex $A_{i}$ has degree one and the weighted Fermat point $A_{0,1}$ has degree $N+1.$
\end{definition}

\begin{definition}{\cite{GilbertPollak:68}}\label{Fermattopology}
A Fermat-Steiner tree topology of a boundary $N-$simplex $A_{1}A_{2}\ldots A_{N+1}$ in $\mathbb{R}^{N}$ is a tree topology, such that each boundary vertex $A_{i}$ has degree one and each weighted Fermat point $A_{0,i}$ has degree three for $i=1,2,\ldots, N-1.$
\end{definition}

\begin{definition}{\cite{GilbertPollak:68}}\label{Fermattopology}
An intermediate Fermat Steiner tree topology of a boundary $N-$simplex $A_{1}A_{2}\ldots A_{N+1}$ in $\mathbb{R}^{N}$ is a tree topology, such that each boundary vertex $A_{i}$ has degree one and the number of weighted Fermat points $A_{0,i}$ is less than $N-1$ having degree less than $N+1.$
\end{definition}

\begin{definition}\label{WeightedFermatrtree}
A tree of weighted minimum length with a Fermat tree topology  is called a weighted Fermat-tree.
\end{definition}

\begin{definition}\label{weightedFermatSteinerrtree}
A tree of weighted minimum length with a Fermat tree topology  is called a weighted Fermat-Steiner tree.
\end{definition}

\begin{definition}\label{intermediateweightedFermatSteinerrtree}
A tree of weighted minimum length with an intermediate Fermat-Steiner tree topology  is called an intermediate weighted Fermat-Steiner tree.
\end{definition}

By replacing $b_{5}=\ldots=b_{N+1}=0,$ $A_{0,1}\equiv A_{0}$ and $A_{0,N-1}\equiv A_{0^\prime}$ in Problem~\ref{WFSN}, we obtain the weighted Fermat-Steiner problem for a boundary tetrahedron $A_{1}A_{2}A_{3}A_{4}$ in $\mathbb{R}^{3}.$

We denote by $H$ the length of the common perpendicular (line segment) between the two lines defined by $A_{1}A_{2},$ $A_{4}A_{3},$ by $a_{ij}$ the length of the line segment $A_{i}A_{j}$ and by $\alpha_{ijk}$ the angle $\angle A_{i}A_{j}A_{k}$ at $A_{j},$  for $i,j,k=0,0^{\prime},1,2,3,4.$

The weighted Fermat-Steiner problem for $A_{1}A_{2}A_{3}A_{4}$ in $\mathbb{R}^{3}$
states that (\cite{Zachos:21}):

\begin{problem}{\cite[Problem~5]{Zachos:21}}\label{Steinertetrahedron}
Find $A_{0}(x_{0},y_{0},z_{0})$ and
$A_{0^{\prime}}(x_{0^{\prime}},y_{0^{\prime}},z_{0^{\prime}})$ with given
weights $b_{0}$ in $A_{0}$ and $b_{0^{\prime}}$ in $A_{0^{\prime}},$ such that
\begin{equation}\label{equat1}
f(A_{0},A_{0^{\prime}})=b_{1}a_{01}+b_{2}a_{02}+b_{3}a_{0^{\prime}3}+b_{4}a_{0^{\prime}4}+\frac{b_{0}+b_{0^{\prime}}}{2}a_{00^{\prime}}\to min.
\end{equation}
\end{problem}

\begin{figure} \label{fig:tas}
\centering
\includegraphics[scale=0.9]{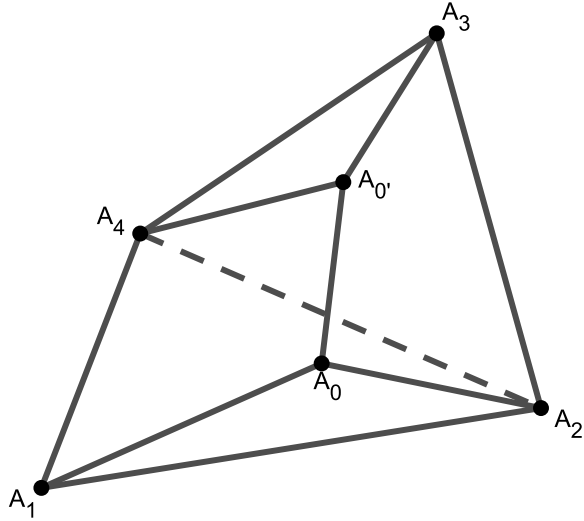}
\caption{A weighted Fermat-Steiner tree for a Frechet tetrahedron in $\mathbb{R}^{3}$}
\end{figure}

By substituting $a_{00^{\prime}}=0$  in (\ref{equat1}) for $A_{0^{\prime}}\to A_{0},$ we get the weighted Fermat problem for $A_{1}A_{2}A_{3}A_{4}$ (\cite{Zach/Zou:09}). The solution(s) of the weighted Fermat-Steiner problem is a weighted Fermat-Steiner tree and $A_{0},$ $A_{0^{\prime}}$ are named as weighted Fermat-Steiner points. If one of the two weighted Fermat-Steiner points is a fixed boundary vertex of $A_{1}A_{2}A_{3}A_{4},$  some degenerate cases may occur and the corresponding trees are degenerate Fermat-Steiner trees. The unique solution of the weighted Fermat problem for $A_{1}A_{2}A_{3}A_{4},$ is a weighted Fermat tree and $A_{0}$ is named as the weighted Fermat point. If $A_{0}$ is a fixed boundary vertex of $A_{1}A_{2}A_{3}A_{4},$ the corresponding tree is a degenerate (absorbing) weighted Fermat tree.

We proceed by giving the definitions of the weighted Fermat-Steiner tree and the weighted Fermat tree for $A_{1}A_{2}A_{3}A_{4}$ in $\mathbb{R}^{3}.$

\begin{definition}\label{weightedFermatSteinertreer3}
A weighted Fermat-Steiner tree $T_{S}(A_{1}A_{2};A_{3}A_{4})$ is a union of the weighted line segments $\{A_{1}A_{0},A_{2}A_{0},A_{0}A_{0^{\prime}},A_{3}A_{0^{\prime}},A_{4}A_{0^{\prime}}\}$ with corresponding weights $\{b_{1},b_{2},\frac{b_{0}+b_{0^{\prime}}}{2},b_{3},b_{4}\}.$
\end{definition}

\begin{definition}\label{weightedFermattreer3}
A weighted Fermat tree $T_{F}(A_{1}A_{2},A_{3}A_{4})$ is a union of the weighted line segments $\{A_{1}A_{0},A_{2}A_{0},A_{3}A_{0},A_{4}A_{0}\}$ with corresponding weights $\{b_{1},b_{2},b_{3},b_{4}\}.$
\end{definition}

We continue by mentioning two lemmas regarding the necessary and sufficient conditions for the existence of the two non-degenerate weighted Fermat points $A_{0}$ and $A_{0^{\prime}},$ (weighted Fermat-Steiner points) and the angular solutions with respect to $A_{0}$ and $A_{0^{\prime}}.$

We set
\[r_{12}\equiv \frac{a_{12}}{(b_{1}+b_{2}+\frac{b_{0}+b_{0^{\prime}}}{2})(b_{1}+b_{2}-\frac{b_{0}+b_{0^{\prime}}}{2})(b_{2}+\frac{b_{0}+b_{0^{\prime}}}{2}-b_{1})(b_{1}+\frac{b_{0}+b_{0^{\prime}}}{2}-b_{2})},\]
\[r_{34}\equiv \frac{a_{34}}{(b_{3}+b_{4}+\frac{b_{0}+b_{0^{\prime}}}{2})(b_{3}+b_{4}-\frac{b_{0}+b_{0^{\prime}}}{2})(b_{3}+\frac{b_{0}+b_{0^{\prime}}}{2}-b_{4})(b_{4}+\frac{b_{0}+b_{0^{\prime}}}{2}-b_{3})},\]
\[\beta_{12}=\arccos(\frac{a_{12}}{2r_{12}}),\]
\[\beta_{34}=\arccos(\frac{a_{34}}{2r_{34}}).\]

We suppose that $A_{1},$ $A_{2}$ lie on the x-axis and satisfy $x_{1}<x_{2}.$
Let $C=(x(C),y(C),z(C)) \in \mathbb{R}^{3}.$

\begin{lemma}{Existence of a weighted Fermat-Steiner tree in $\mathbb{R}^{3},$\cite[Theorem~2]{Zachos:21}}\label{conditionswst}
The following inequalities provide the necessary and sufficient conditions  for the existence of the two non-degenerate weighted Fermat points $A_{0}$ and $A_{0^{\prime}}:$
\begin{equation}\label{ineq1}
\frac{\sqrt{y(C)^2+z(C)^2}}{|x_{1}-x(C)|}>\tan(\arccos(\frac{(\frac{b_{0}+b_{0^{\prime}}}{2})^2-b_{1}^2-b_{2}^2)}{2 b_{1}b_{2}})),
\end{equation}

\begin{equation}\label{ineq3}
(\sqrt{y(C)^2+z(C)^2}+r_{12}\sin\beta_{12})^{2}+(\frac{x_{1}+x_{2}}{2}-x(C))^2>r_{12}^{2},
\end{equation}

\begin{equation}\label{ineq4}
\frac{\sqrt{y(C)^2+z(C)^2}}{|x_{4}-x(C)|}>\tan(\arccos(\frac{(\frac{b_{0}+b_{0^{\prime}}}{2})^2-b_{3}^2-b_{4}^2)}{2 b_{3}b_{4}})),
\end{equation}

\begin{equation}\label{ineq6}
(\sqrt{y(C)^2+z(C)^2}+r_{34}\sin\beta_{34})^{2}+(\frac{x_{4}+x_{3}}{2}-x(C))^2>r_{34}^{2}.
\end{equation}

\end{lemma}

We set $b_{ST}\equiv \frac{B_{0}+B_{0^{\prime}}}{2}.$

\begin{lemma}{\cite[Theorem~3]{Zachos:21}}\label{angularsolutionweightedFermatSteiner}
The solution of the weighted Steiner problem is a weighted Steiner tree in $\mathbb{R}^{3}$
whose nodes $A_{0}$ and $A_{0^{\prime}}$ (weighted Fermat points) are seen by the angles:
\begin{eqnarray}\label{FTangles}
 \cos\alpha_{102}& = & \frac{b_{ST}^2-b_{1}^2-b_{2}^2}{2b_{1}b_{2}},\nonumber \\
\cos\alpha_{012} & = & \frac{b_{1}^2-b_{2}^2-b_{ST}^2}{2b_{2}b_{ST}},\nonumber \\
\cos\alpha_{30^{\prime}4} & = &\frac{b_{ST}^2-b_{3}^2-b_{4}^2}{2b_{3}b_{4}},\nonumber \\
\cos\alpha_{340^{\prime}} & =& \frac{b_{4}^2-b_{3}^2-b_{ST}^2}{2b_{3}b_{ST}}.
\end{eqnarray}

\end{lemma}


In \cite{Melzak:61}, Melzak constructed an algorithm of circles to find Steiner tree topologies for an $N$convex polygon in $\mathbb{R}^{2}.$ Unfortunately, Melzak cannot be extended in $\mathbb{R}^{3}.$ In \cite{RubinsteinThomasWeng:02}, Rubinstein, Thomas and Weng succeeded in solving numerically the unweighted Fermat-Steiner problem, by applying a fixed point iteration method to a system of two equations with two variable line segmens. In \cite[Theorem~4]{Zachos:21}, we extended Rubinstein, Thomas and Weng method to solve the weighted Fermat-Steiner problem for tetrahedra in $\mathbb{R}^{3}.$ We note that extended Rubinstein Thomas method takes into account the coordinates $(x_{i},y_{i},z_{i})$ of each vertex $A_{i},$ for $i=1,2,3,4.$ Therefore, we need to find a method to use the six edge lengths of the tetrahedron and some variable edge lengths, in order to consider the weighted Fermat-Steiner-Frechet problem for Frechet tetrahedra in $\mathbb{R}^{3}.$


We denote by $D(S)$ the Cayley-Menger determinant:

\begin{equation}\label{CaleyMenger}
D(S) =\operatorname{det} \left(
\begin{array}{ccccc}
0      & a_{12}^2      & a_{13}^2 & a_{14}^2  &  1  \\
a_{12}^2 & 0 & a_{23}^2 &a_{24}^2 & 1          \\
a_{13}^2 & a_{23}^2 &0   &a_{34}^2  & 1         \\
a_{14}^2 & a_{24}^2 &a_{34}^2  & 0   & 1         \\
1 & 1 &1  & 1     &   0    \\
\end{array} \right).
\end{equation}

We consider the $3-INVWF$ problem in $\mathbb{R}^{3}.$  Let $A_{0}$ be a weighted Fermat point inside $A_{1}A_{2}A_{3}A_{4}$ in $\mathbb{R}^{3}.$ In \cite{Zach/Zou:09}, the following relation is proved:


\begin{lemma}{\cite[Formula (2.25)]{Zach/Zou:09}}\label{tetrahedroninv}
If $b_{i0}$ is the weight, which corresponds to the vertex $A_{i}:$

\begin{equation}\label{floatingcase}
\| \sum_{j=1, i\ne j}^{4}b_{j0}\vec{a}_{ij}\|>b_{i0},
\end{equation}
for $i,j=1,2,3,4$  holds, then
\begin{equation}\label{tetraed1}
\frac{b_{30}}{a_{03}\operatorname{Vol}(A_{0}A_{1}A_{2}A_{4})}=\frac{b_{40}}{a_{04}\operatorname{Vol}(A_{0}A_{1}A_{2}A_{3})}=C,
\end{equation}

\begin{equation}\label{tetraed2}
\frac{b_{30}}{a_{03}\operatorname{Vol}(A_{0}A_{1}A_{2}A_{4})}=\frac{b_{10}}{a_{01}\operatorname{Vol}(A_{0}A_{2}A_{3}A_{4})}=C,
\end{equation}

\begin{equation}\label{tetraed3}
\frac{b_{30}}{a_{03}\operatorname{Vol}(A_{0}A_{1}A_{2}A_{4})}=\frac{b_{20}}{a_{02}\operatorname{Vol}(A_{0}A_{1}A_{3}A_{4})}=C,
\end{equation}

and

\begin{equation}\label{tetraed4}
\frac{b_{10}}{a_{01}\operatorname{Vol}(A_{0}A_{2}A_{3}A_{4})}=\frac{b_{20}}{a_{02}\operatorname{Vol}(A_{0}A_{1}A_{3}A_{4})}=C,
\end{equation}

where
$C=\frac{\sum_{i=1}^{4}\frac{b_{i0}}{a_{0i}}}{\operatorname{Vol}(A_{1}A_{2}A_{3}A_{4})}.$

\end{lemma}

The volumes of $A_{0}A_{i}A_{j}A_{k}$ $\operatorname{Vol}(A_{0}A_{i}A_{j}A_{k})$ for $i,j,k=1,2,3,4,$
can be computed via the Caley-Menger determinant (\cite[pp.~249-255]{Uspensky:48}):
\[288 \operatorname{Vol}(A_{0}A_{i}A_{j}A_{k})^{2}
=D(\{a_{0i},a_{0j},a_{0k},a_{ij},a_{ik},a_{jk}\}).\]


\section{The weighted Fermat-Steiner-Frechet multitree for a given sextuple of positive real numbers determining the edge lengths of incongruent tetrahedra in $\mathbb{R}^{3}$}
In this section, we focus on the solution (multitree) of the weighted Fermat-Steiner-Frechet problem (P(Fermat-Steiner-Frechet)), by inserting some equality constraints derived by two independent solutions for two new variable weighted Fermat problems for the Frechet multitetrahedron derived by incongruent boundary tetrahedra in $\mathbb{R}^{3},$ which correspond to the same sextuple of positive real numbers (edge lengths) and an equality constraint derived by two different expressions of the line segment connecting the two weighted Fermat-Steiner points. The detection of the weighted Fermat-Steiner Frechet multitrees is achieved by applying the Lagrange multiplier rule.

We give a vector proof of the law of cosine law in $\mathbb{R}^{3},$ which has been introduced in \cite{Zach/Zou:09}, by using addition and inner product of vectors in $\mathbb{R}^{3}.$

We denote by $\alpha$ the dihedral
angle defined by the planes formed by $\triangle
A_{0}A_{1}A_{2}$ and $\triangle A_{1}A_{2}A_{3},$  the dihedral
angle $\alpha_{g_{i}}$ defined by the planes formed by $\triangle
A_{1}A_{2}A_{3}$ and $\triangle A_{1}A_{2}A_{i},$ by $h_{0,12}$
the height of $\triangle A_{0}A_{1}A_{2}$ from $A_{0},$ by
$h_{0,12m}$ the distance of $A_{0}$ from the plane defined by
$\triangle A_{1}A_{2}A_{m},$ for $i,j,k=0,1,2,3,4$ and $m=3,4,$
by $A_{0,12}$ the trace of the orthogonal projection of $A_{0}$ to $A_{1}A_{2}$
by $A_{0,123},$ the trace of the orthogonal projection of $A_{0}$ to the plane defined by $\triangle A_{1}A_{2}A_{3},$ by $x_{(0,12),2}$ the length of the line segment $A_{0,12}A_{2},$ by $x_{(0,123),2}$
the length of the line segment $A_{0,123}A_{2}.$

\begin{lemma}{Generalized cosine law in $\mathbb{R}^{3},$ \cite{Zach/Zou:09}}\label{calculationa0304a01a03alpha}

The line segment $a_{i0}$ depends on $a_{10}, a_{20}$ and $\alpha$ in $\mathbb{R}^{3}:$

\begin{equation}\label{eq:deral2}
a_{i0}^2=a_{20}^2+a_{2i}^2-2a_{2i}[\sqrt{a_{20}^2-h_{0,12}^2}\cos({\alpha_{12i}})+h_{0,12}\sin({\alpha_{12i}})\cos({\alpha_{g_{i}}}-\alpha)],
\end{equation}

or

\begin{equation}\label{eq:derall2n}
a_{i0}^2=a_{10}^2+a_{1i}^2-2a_{1i}[\sqrt{a_{10}^2-h_{0,12}^2}\cos({\alpha_{21i}})+h_{0,12}\sin({\alpha_{21i}})\cos({\alpha_{g_{i}}}-\alpha)].
\end{equation}

for $i=3,4.$

\end{lemma}

\begin{proof}

\begin{figure}\label{figg1}
\centering
\includegraphics[scale=0.80]{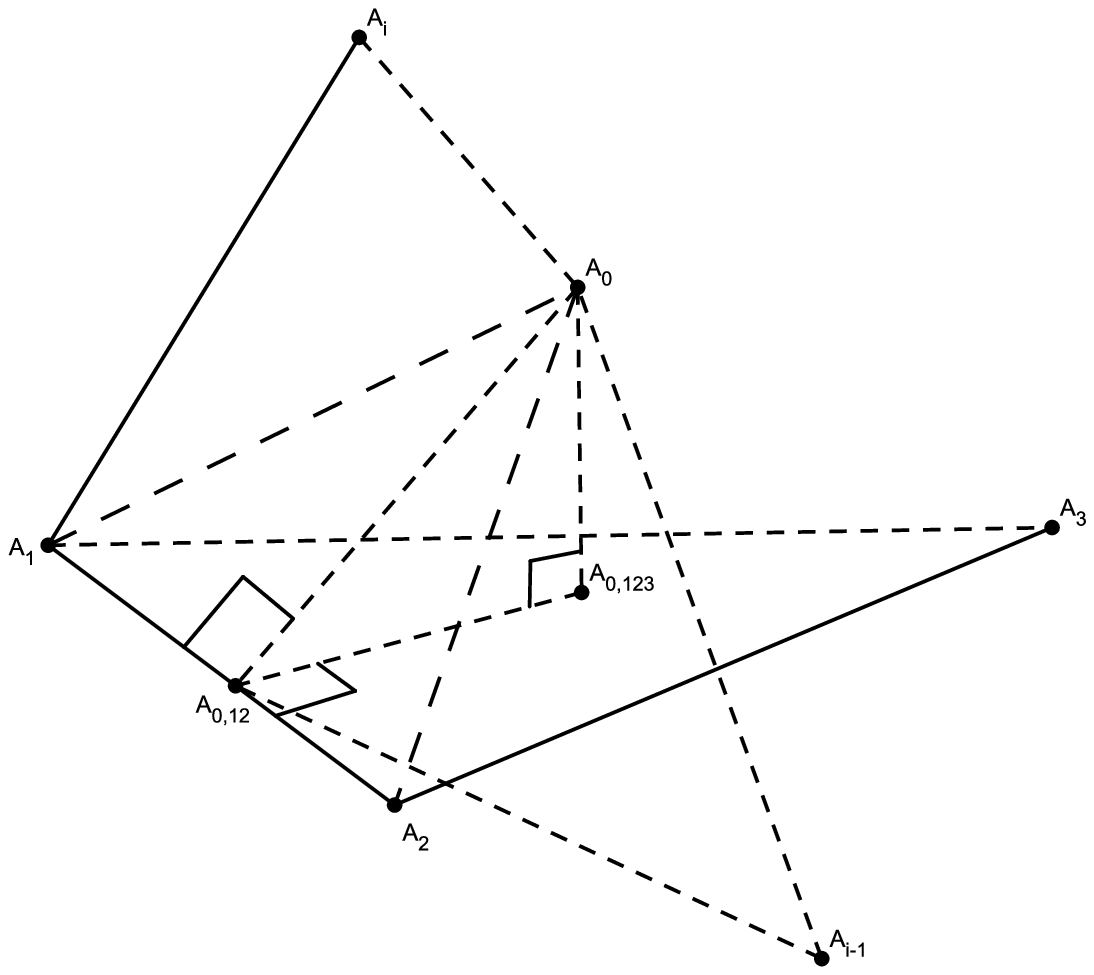}
\caption{} \label{figg1}
\end{figure}

First, we start with the elementary observation
\[\vec{a}_{02}=\vec{h}_{0,12}+\vec{x}_{(0,12),2}.\]
The inner product $\vec{a}_{02}\ddot \vec{a}_{02}$ yields the cosine law for $\triangle A_{0}A_{1}A_{2}$ in $\mathbb{R}^{2}:$

\[a_{20}^{2}=a_{10}^{2}+a_{12}^{2}-2a_{10}a_{12}\cos\alpha_{012}.\]

We consider the following vector equality in $\mathbb{R}^{3}:$

\[\vec{a}_{03}=\vec{h}_{0,123}+\vec{x}_{(0,123),2}+\vec{a}_{23}.\]

Taking the inner product $\vec{a}_{03}\ddot \vec{a}_{03}$  and by setting $\varphi=\angle A_{1}A_{2}A_{0,123},$ we get:
\begin{equation}\label{vecproof}
a_{03}^{2}=a_{02}^{2}+a_{23}^{2}-2x_{(0,123),2}a_{23}\cos(\alpha_{123}-\varphi).
\end{equation}

From $\triangle A_{0,12}A_{2}A_{0,123},$ we get:
 \begin{equation}\label{cosvec2}
 \cos\varphi=\frac{\sqrt{a_{02}^2 - h_{0,12}^{2}}}{x_{(0,123),2}}
 \end{equation}

\begin{equation}\label{sinvec3}
\sin\varphi=\frac{h_{0,12}\cos\alpha}{x_{(0,123),2}}.
 \end{equation}

By substituting (\ref{cosvec2}) and (\ref{sinvec3}) in (\ref{vecproof}), we obtain:
(\ref{eq:deral2}) for $i=3$ and $\alpha_{g_{3}}=0.$

By changing the index $3\to 4,$ we derive (\ref{eq:deral2}) for $i=4$ and $\alpha_{g_{4}}\ne 0.$

Taking into account the following vector equality in $\mathbb{R}^{3},$

\[\vec{a}_{03}=\vec{h}_{0,123}+\vec{x}_{(0,123),1}+\vec{a}_{13}.\]

and following the same process, we obtain (\ref{eq:derall2n}) for $i=3$ and $i=4.$

\end{proof}

If we substitute $\alpha_{g_{3}}=\alpha_{g_{4}}=\alpha$, we obtain a generalization of the
cosine law in $\mathbb{R}^{2}:$

\begin{lemma}[Generalized Cosine law in $\mathbb{R}^{2}$]\label{cosinelawr2}
The line segment $a_{i0}$ depends on $a_{10}, a_{20}$ in $\mathbb{R}^{2}:$

\begin{equation}\label{eq:deral2}
a_{i0}^2=a_{20}^2+a_{2i}^2-2a_{2i}[\sqrt{a_{20}^2-h_{0,12}^2}\cos({\alpha_{12i}})+h_{0,12}\sin({\alpha_{12i}})],
\end{equation}

or

\begin{equation}\label{eq:derall2n}
a_{i0}^2=a_{10}^2+a_{1i}^2-2a_{1i}[\sqrt{a_{10}^2-h_{0,12}^2}\cos({\alpha_{21i}})+h_{0,12}\sin({\alpha_{21i}})].
\end{equation}
where
\[h_{0,12}=\frac{a_{10}a_{20}}{a_{12}}\sqrt{1-\left(\frac{a_{10}^{2}+a_{20}^{2}-a_{12}^2}{2a_{10}a_{20}}
\right)^{2}}\]

for $i=3,4.$

\end{lemma}


By solving (\ref{eq:deral2}) with respect to $\alpha,$ for $i=3,$ we get:

\[\alpha=\arccos\left(
\frac{\left(\frac{a_{02}^2+a_{23}^2-a_{03}^2}{2 a_{23}}
\right)-\sqrt{a_{02}^2-h_{0,12}^2}\cos\alpha_{123}}{h_{0,12}\sin\alpha_{123}}
\right).\]

Therefore, by substituting $\alpha$ function in (\ref{eq:deral2}) for $i=4,$
we obtain a functional dependence of $a_{40}$ in terms of lengths of line segments (\cite{Zachos:16}).


\begin{lemma}{\cite[Proposition~1]{Zachos:16}}\label{a04a01a02a03r3}
The variable length $a_{40}$ depends on the three variable
lengths $a_{10},$ $a_{20},$ $a_{30}$ and the given sextuple of
positive real numbers $S=\{a_{12},a_{13},a_{14},a_{34},a_{24},a_{23}\}$ determining the edge lengths of incongruent tetrahedra in $\mathbb{R}^{3},$ by taking into account the following relations:

\[\cos\alpha_{123}=\frac{a_{12}^2+a_{23}^2-a_{13}^2}{2 a_{12}a_{23}},\]

\[\sin\alpha_{123}=\frac{\sqrt{(a_{12}+a_{23}+a_{13})(a_{23}+a_{13}-a_{12})(a_{12}+a_{13}-a_{23})(a_{12}+a_{23}-a_{13})}}{2
a_{12} a_{23}},\]

\[h_{0,12}=h_{0,12}(a_{01},a_{02};a_{12})=\frac{a_{01}a_{02}}{a_{12}}\sqrt{1-\left(\frac{a_{01}^{2}+a_{02}^{2}-a_{12}^2}{2a_{01}a_{02}}
\right)^{2}},\]

\[\cos\alpha_{124}=\frac{a_{12}^2+a_{24}^2-a_{14}^2}{2 a_{12}
a_{24}},\]

\[\sin\alpha_{124}=\frac{\sqrt{(a_{12}+a_{24}+a_{14})(a_{24}+a_{14}-a_{12})(a_{12}+a_{14}-a_{24})(a_{12}+a_{24}-a_{14})}}{2
a_{12} a_{24}}\]

\[\alpha_{g_{4}}=\arccos\left(
\frac{\left(\frac{a_{42}^2+a_{23}^2-a_{43}^2}{2 a_{23}}
\right)-\sqrt{a_{42}^2-h_{4,12}^2}\cos\alpha_{123}}{h_{4,12}\sin\alpha_{123}}
\right)\]

and

\[h_{4,12}=h_{4,12}(a_{41},a_{42},a_{12})=\frac{a_{41}a_{42}}{a_{12}}\sqrt{1-\left(\frac{a_{41}^{2}+a_{42}^{2}-a_{12}^2}{2a_{41}a_{42}}
\right)^{2}}.\]

\end{lemma}

The weighted Fermat-Steiner-Frechet problem for a given sextuple of edge lengths determining tetrahedra in $\mathbb{R}^{3},$ states that:

\begin{problem}[The weighted Fermat-Steiner Frechet in $\mathbb{R}^{3}$]\label{FermatSteinerFrechettetrahedronr3}
Given a sextuple of weights $\{b_{1},b_{2},b_{3},b_{4},b_{ST},b_{ST}\},$ and a given sextuple of positive real number (edge lengths) $\{a_{12},a_{13},a_{14},a_{23},a_{24},a_{34}\},$ determining a Frechet multitetrahedron $F(A_{1}A_{2}A_{3}A_{4}),$   find the position of $A_{0}$ and / or
$A_{0^{\prime}}$ with given
weights $b_{ST}$ in $A_{0}$ and $b_{ST}$ in $A_{0^{\prime}},$ such that
\begin{equation}\label{equat1L0}
f_{0}(a_{01},a_{02},a_{0^{\prime}3},a_{0^{\prime}4},a_{0^{\prime}0})=b_{1}a_{01}+b_{2}a_{02}+b_{3}a_{0^{\prime}3}+b_{4}a_{0^{\prime}4}+b_{ST} a_{00^{\prime}}\to min.
\end{equation}
\end{problem}

\begin{definition}[A non degenerate weighted Fermat-Steiner tree for $\{A_{1}A_{2}A_{3}A_{4}\}$]

A non degenerate weighted Fermat-Steiner tree is the minimum of the weighted Fermat-Steiner trees $T_{S}(A_{1}A_{2};A_{3}A_{4}),$\\ $T_{S}(A_{1}A_{4};A_{2}A_{3})$ and $T_{S}(A_{1}A_{3};A_{2}A_{4}).$

\end{definition}

\begin{definition}[A degenerate weighted Fermat-Steiner (Gauss) tree for $A_{1}A_{2};A_{3}A_{4}$]
A degenerate weighted Fermat-Steiner tree or Gauss tree is a weighted minimal tree, such that one of two vertices $A_{0}$ or $A_{0^{\prime}}$ coincides with $A_{1}$ or $A_{2}$ or $A_{3}$ or $A_{4},$ respectively.
\end{definition}

We continue by constructing the Lagrangian function \[\mathcal{L}(\tilde{x},\tilde{\lambda})=\sum_{i=0}^{7}\lambda_{i}f_{i}(\tilde{x}).\]
where the point
\[\tilde{x}=\{x_{1},\ldots,x_{12}\}=\] \[=\{a_{10},a_{20},a_{30}, a_{20^{\prime}}, a_{30^{\prime}},a_{40^{\prime}},B_{10},B_{20},B_{30},B_{10^{\prime}},B_{20^{\prime}},B_{30^{\prime}}\}\in \mathbb{R}^{12}\] is inside the parallelepiped $\Pi(p_{1},q_{1};\cdots ;p_{12},q_{12}),$
where $p_{i}<x_{i}<q_{i},$ for $i=1,2,\cdots, 12$ and the Lagrange multiplication vector is given by:

\[\tilde{\lambda}=\{\lambda_{0},\lambda_{1},\cdots,\lambda_{7}\}.\]

We shall deal with the weighted Fermat-Steiner-Frechet problem (P(Fermat-Steiner-Frechet)), by inserting some equality constraints derived by two independent solutions for two new weighted Fermat problems for $A_{1}A_{2}A_{3}A_{4}$ in $\mathbb{R}^{3},$ which give a connection with the initial weighted Fermat-Steiner objective function and an equality constraint derived by two different expressions of $a_{00^{\prime}}.$
\begin{problem}[The weighted Fermat-Steiner-Frechet (P(Fermat-Steiner-Frechet)) problem in $\mathbb{R}^{3}$ with equality constraints]
\begin{equation*}
\begin{aligned}
& & f_{0}(\tilde{x})\to min, \\
& & f_{i}(\tilde{x}) = 0, \; i = 1, \cdots,7
\end{aligned}
\end{equation*}

\begin{equation}\label{fundzero}
f_{0}(\tilde{x})=b_{1}a_{01}+b_{2}a_{02}+b_{3}a_{0^{\prime}3}+b_{4}a_{0^{\prime}4}+b_{ST} a_{00^{\prime}},
\end{equation}

\begin{equation}\label{fundone}
f_{1}(\tilde{x})=\frac{B_{10}}{a_{10}\operatorname{Vol}(A_{0}A_{2}A_{3}A_{4})}-\frac{1-B_{10}-B_{20}-B_{30}}{a_{40}(a_{10},a_{20},a_{30})\operatorname{Vol}(A_{0}A_{1}A_{2}A_{3})},
\end{equation}

\begin{equation}\label{fundseond}
f_{2}(\tilde{x})=\frac{B_{20}}{a_{20}\operatorname{Vol}(A_{0}A_{1}A_{3}A_{4})}-\frac{1-B_{10}-B_{20}-B_{30}}{a_{4}(a_{10},a_{20},a_{30})\operatorname{Vol}(A_{0}A_{1}A_{2}A_{3})},
\end{equation}

\begin{equation}\label{fundthird}
f_{3}(\tilde{x})=\frac{B_{30}}{a_{30}\operatorname{Vol}(A_{0}A_{1}A_{2}A_{4})}-\frac{1-B_{10}-B_{20}-B_{30}}{a_{40}(a_{10},a_{20},a_{30})\operatorname{Vol}(A_{0}A_{1}A_{2}A_{3})},
\end{equation}


\begin{equation}\label{fundfourth}
f_{4}(\tilde{x})=\frac{B_{40^{\prime}}}{a_{40^{\prime}}\operatorname{Vol}(A_{0^{\prime}}A_{1}A_{2}A_{3})}-\frac{1-B_{20^{\prime}}-B_{30^{\prime}}-B_{40^{\prime}}}{a_{10^{\prime}}(a_{20^{\prime}},a_{30^{\prime}},a_{40^{\prime}})\operatorname{Vol}(A_{0^{\prime}}A_{2}A_{3}A_{4})},
,
\end{equation}

\begin{equation}\label{fundfifth}
f_{5}(\tilde{x})=\frac{B_{20^{\prime}}}{a_{20^{\prime}}\operatorname{Vol}(A_{0^{\prime}}A_{1}A_{3}A_{4})}-\frac{1-B_{20^{\prime}}-B_{30^{\prime}}-B_{40^{\prime}}}{a_{10^{\prime}}(a_{20^{\prime}},a_{30^{\prime}},a_{40^{\prime}})\operatorname{Vol}(A_{0^{\prime}}A_{2}A_{3}A_{4})},
,
\end{equation}

\begin{equation}\label{fundsixth}
f_{6}(\tilde{x})=\frac{B_{30^{\prime}}}{a_{30^{\prime}}\operatorname{Vol}(A_{0^{\prime}}A_{1}A_{2}A_{4})}-\frac{1-B_{20^{\prime}}-B_{30^{\prime}}-B_{40^{\prime}}}{a_{10^{\prime}}(a_{20^{\prime}},a_{30^{\prime}},a_{40^{\prime}})\operatorname{Vol}(A_{0^{\prime}}A_{2}A_{3}A_{4})},
\end{equation}

\begin{equation}\label{seven}
f_{7}(\tilde{x})=a_{00^{\prime}}(a_{10},a_{20},a_{10^{\prime}},a_{20^{\prime}})-a_{00^{\prime}}(a_{30^{\prime}},a_{40^{\prime}},a_{30},a_{40}(a_{10},a_{20},a_{30})).
\end{equation}


\end{problem}

The next theorem is a direct consequence of the Lagrange multiplier rule given in \cite[p.~112]{Tikhomirov:90} in \cite[Theorem~3.1,p.~586]{BrinkhisTikhomirov:05} and a particular case of an ordinary convex program involving only equalities in \cite[Theorem~28.1]{Rockafellar:97}

\begin{theorem}[Lagrange multiplier rule for the weighted Fermat-Steiner Frechet multitree in $\mathbb{R}^{3}$]\label{Lagrangerulemultitree}
If the admissible point $\tilde{x}_{i}$ yields a weighted minimum multitree for $1\le i \le 30,$ which correspond to a Frechet multitetrahedron derived by a sextuple of edge lengths determining upto 30 incongruent tetrahedra, then there are numbers $\lambda_{0i},\lambda_{1i},\lambda_{2i},\ldots \lambda_{7i},$ such that:

\begin{equation}\label{lagrangemultitreer3cond1}
\frac{\partial \mathcal{L}_{i}(\tilde{x}_{i},\tilde{\lambda}_{i})}{\partial x_{ji}}=0
\end{equation}
for $j=1,2,\ldots,12,$

\[\tilde{x}_{i}=\{(a_{10})_{i},(a_{20})_{i},(a_{30})_{i}, (a_{20^{\prime}})_{i}, (a_{30^{\prime}})_{i},(a_{40^{\prime}})_{i},\]\[(B_{10}){i},(B_{20})_{i},(B_{30}){i},(B_{10^{\prime}})_{i},(B_{20^{\prime}})_{i},(B_{30^{\prime}})_{i},\}\]

$\tilde{\lambda}_{i}=\{\lambda_{0i},\lambda_{1i}\lambda_{2i},\ldots, \lambda_{7i}\}.$

\end{theorem}

\begin{proof}
Taking into account that $\frac{\partial (f_{k})_{i} }{(x_{ji})}$ are continuous in each parallelepiped $\Pi_{i},$ for $1\le i\le 30,$ $k=0,1,2,\ldots, 7,$ $j=1,2,\ldots, 12$ and by applying Lagrange multiplier rule,
yields the Lagrangian vector $\tilde{\lambda}_{i}=\{\lambda_{0i},\lambda_{1i}\lambda_{2i},\ldots, \lambda_{7i}\},$ such that (\ref{lagrangemultitreer3cond1}) is valid.

\end{proof}

We note that the system (\ref{fundone})-(\ref{seven}), (\ref{lagrangemultitreer3cond1}) contains 19 equations with 19 variables, since we can set one of the Lagrange multipliers 1 ($(\lambda_{0})_{i}=1$), by the definition of the P(Fermat-Steiner-Frechet)problem.


\section{A Lagrange program detecting the most natural sextuples of six consecutive natural numbers}
In this section, we apply a Lagrange program to detect weighted Fermat-Frechet multitree for a given sextuple of edge lengths determining 30 incongruent tetrahedra (Frechet multitetrahedron) in $\mathbb{R}^{3},$ by using Blumenthal, Herzog and Dekster Wilker sextuples. An interesting application of seeking unweighted Fermat-Frechet multitrees with two equally weighted Fermat-Steiner points inside the Frechet nultitetrahedron, is to detect the most natural of six consecutive natural numbers (Herzog sextuples) for $N\ge 7.$ This result is achieved by seeking an upper bound for these two equal weights, which yield a global weighted Fermat-Steiner tree of minimum length for the boundary tetrahedron having the maximum volume among the 30 incongruent tetrahedra in $\mathbb{R}^{3}.$

\begin{problem}[The Fermat-Steiner-Frechet (P(Fermat-Steiner-Frechet)) problem in $\mathbb{R}^{3}$ with equality constraints]
\begin{equation*}
\begin{aligned}
& & y_{0}(\tilde{x})\to min, \\
& & y_{i}(\tilde{x}) = 0, \; i = 1, \cdots,7
\end{aligned}
\end{equation*}
where
\begin{equation}\label{fundzerounw}
y_{0}(\tilde{x})=a_{01}+a_{02}+a_{0^{\prime}3}+a_{0^{\prime}4}+b_{ST}a_{00^{\prime}},
\end{equation}

and

\[y_{i}(\tilde{x})=f_{i}(\tilde{x}).\]

\end{problem}

\begin{proposition}[Lagrange multiplier rule for the Fermat-Steiner Frechet multitree in $\mathbb{R}^{3}$]\label{LagrangerulemultitreeBlumenthal}
If the admissible point $\tilde{x}_{i}$ yields a minimum multitree for $i=1,2,\ldots, 30,$ which correspond to a Frechet multitetrahedron derived by the Blumenthal, Herzog or Dekster-Wilker sextuples of edge lengths determining 30 incongruent tetrahedra, then there are numbers $\lambda_{0i},\lambda_{1i},\lambda_{2i},\ldots \lambda_{7i},$ such that:

\begin{equation}\label{lagrangemultitreer3cond1bl}
\frac{\partial \mathcal{L}_{i}(\tilde{x}_{i},\tilde{\lambda}_{i})}{\partial x_{ji}}=0
\end{equation}
for $j=1,2,\ldots,12,$

\[\tilde{x}_{i}=\{(a_{10})_{i},(a_{20})_{i},(a_{30})_{i}, (a_{20^{\prime}})_{i}, (a_{30^{\prime}})_{i},(a_{40^{\prime}})_{i},\]\[(B_{10}){i},(B_{20})_{i},(B_{30}){i},(B_{10^{\prime}})_{i},(B_{20^{\prime}})_{i},(B_{30^{\prime}})_{i},\}\]
$\tilde{\lambda}_{i}=\{\lambda_{0i},\lambda_{1i}\lambda_{2i},\ldots, \lambda_{7i}\}.$

\end{proposition}

\begin{proof}
It is a direct consequence of Theorem~\ref{Lagrangerulemultitree} for Blumenthal, Herzog, Dekster, Wilker sextuples determining thirty incongruent tetrahedra in $\mathbb{R}^{3}.$
\end{proof}

\begin{remark}\label{controlledrelationsr3}
From Lemmas~\ref{solinvtetrr3}, ~\ref{lem4r32}, (\ref{calcalpha3042}) yields that the weight $B_{i0}$ depends on five given angles $\alpha_{102},$ $\alpha_{103},$ $\alpha_{104},$
$\alpha_{203},$ $\alpha_{204},$ such that:

\[\alpha_{102}=\arccos(\frac{b_{ST}^2 -b_{1}^2-b_{2}^2}{2 b_{1}b_{2}}),\]

\[\alpha_{i0j}=\arccos(\frac{a_{i0}^2+a_{j0}^2-a_{ij}^2}{2 a_{i0}a_{j0}}).\]
Therefore, we get:
\[B_{i0}=B_{i0}(a_{10},a_{20},a_{30};b_{1};b_{2};b_{ST}),\]
for $i=1,2,3,4.$
By following a similar process, we get:
\[B_{i0}=B_{i0}(a_{20^{\prime}},a_{30^{\prime}},a_{40^{\prime}};b_{3};b_{4};b_{ST}),\]
for $i=1,2,3,4.$
\end{remark}

\begin{theorem}\label{mostnaturalconsecutivesextuples}
The most natural sextuple of numbers from six consecutive natural numbers $\{a+5,a+4,a+3,a+2,a+1,a,\}$
for $a\ge 7$ is a sextuple of edge lengths having the maximum volume (maximum sextuple) among the 30 incongruent tetrahedra, which corresponds a Fermat-Steiner tree of minimum total weighted length (global minimum solution), such that the upper bound for the weight $B_{ST}$ is determined by the rest Fermat-Steiner minimal trees having larger or equal weighted minimal total length.
\end{theorem}

\begin{proof}
By applying Proposition~\ref{LagrangerulemultitreeBlumenthal}  for the Herzog sextuple of edge lengths $\{a+5,a+4,a+3,a+2,a+1,a,\}$ for $a\ge 7$ forming 30 incongruent tetrahedra in $\mathbb{R}^{3},$ we obtain 90 minimum trees for a given weight $B_{ST},$  $T_{S}((A_{1}A_{2};A_{3}A_{4})_{j}),$ $T_{S}((A_{1}A_{4};A_{2}A_{3})_{j})$ and $T_{S}((A_{1}A_{3};A_{2}A_{4})_{j})$ which correspond to each derived
tetrahedron $(A_{1}A_{2};A_{3}A_{4})_{j},$ for $j=1,2,\ldots, 30.$ If $B_{ST}$ yields a global minimum tree of the tetrahedron with edge lengths that belongs to $\{a+5,a+4,a+3,a+2,a+1,a,\}$ having maximum volume, then we derive the most natural sextuple of the six consecutive natural numbers $\{a+5,a+4,a+3,a+2,a+1,a,\}$ otherwise we consider the variable weight $B_{ST},$ in order to perturb the length of the minimum tree only for the maximum sextuple. Hence, we obtain an upper bound for $B_{ST},$ by comparing the length of the minimum tree derived for the maximum sextuple with the rest Fermat-Steiner trees, which correspond to the rest 29 incongruent tetrahedra.
\end{proof}

\begin{example}
Consider 30 incongruent tetrahedra derived by six consecutive natural numbers having edge lengths $\{7, 8, 9, 10, 11, 12\}$ where $f_{0}=a_{1}+a_{2}+a_{3}+a_{4},$ for $b_{1}=b_{2}=b_{3}=b_{4}=1,$ and $a_{00^{\prime}}=0.$
This is the first tetrahedral sextuple of sequential positive integers forming 30 tetrahedra (Blumenthal-Herzog).
We take these 30 deformations of a tetrahedron (plasticity of the boundary of a tetrahedron) and we
compute the length of each Fermat tree. The radius $R$ corresponds to the circumscribed sphere of a tetrahedron having six edges
$a_{12},$ $a_{43},$ $a_{13},$  $a_{23},$ $a_{24},$ $a_{14}.$
We may expect that nature chooses the minimum communication among these 30 deformation on a boundary tetrahedron having the maximum volume or on a sphere having maximum volume, but computations do not give such a result for Fermat trees having one node (Fermat point) inside each derived tetrahedron.
The minimum of the minimum communication is achieved by the edge lengths $\{a_{12},a_{43}, a_{13}, a_{23}, a_{24} ,a_{14}\}=\{12, 7, 11, 10,	8,	9\},$ having a Fermat tree of minimum length $22.7838,$ with respect to the derived boundary tetrahedron without having the maximum volume and the corresponding circumscribed sphere with radius $6.59837$ without having the maximum volume.

\begin{tabular}{|c|c|c|c|c|c|c|c|c|}
\hline

$a_{12}$ & $a_{43}$ & $a_{13}$ & $a_{23}$ & $a_{24}$ & $a_{14}$ & minf & Radius R &  $D(a_{12},a_{43}, a_{13},a_{23},a_{24},a_{14})$ \\
\hline

12	&7	&11	&10	&9	&8	&22.8131	&6.62431	&1905982\\
12	&7	&11	&10	&8	&9	&22.7838	&6.59837	&1994518\\
12	&7	&11	&9	&10	&8	&23.0364	&6.29963	&1843168\\
12	&7	&11	&9	&8	&10	&22.9123	&6.28226	&2200288\\
12	&7	&11	&8	&9	&10	&22.9827	&6.1946  	&2179582\\
12	&7	&11	&8	&10	&9	&23.0773	&6.18682	&1910998\\
12	&8	&11	&10	&9	&7	&22.8788	&6.69308	&1808302\\
12	&8	&11	&10	&7	&9	&22.8186	&6.64231	&1914478\\
12	&8	&11	&9	&10	&7	&23.149	    &6.33008	&1811038\\
12	&8	&11	&9	&7	&10	&22.955	    &6.29235	&2133358\\
12	&8	&11	&7	&10	&9	&23.2132	&6.15014	&1918558\\
12	&8	&11	&7	&9	&10	&23.0802	&6.16018	&2134702\\
12	&9	&11	&10	&8	&7	&22.9099	&6.77582	&1642518\\
12	&9	&11	&10	&7	&8	&22.8789	&6.7513	    &1660158\\
12	&9	&11	&8	&7	&10	&23.0802	&6.1715	    &1986750\\
12	&9	&11	&8	&10	&7	&23.3948	&6.17336	&1823958\\
12	&9	&11	&7	&10	&8	&23.4179	&6.10726	&1863648\\
12	&9	&11	&7	&8	&10	&23.1348	&6.12701	&2008800\\
12	&10	&11	&9	&8	&7	&23.7593	&6.35084	&1397038\\
12	&10	&11	&9	&7	&8	&23.8075	&6.32303	&1362238\\
12	&10	&11	&8	&9	&7	&23.433	    &6.1913  	&1575742\\
12	&10	&11	&8	&7	&9	&23.2132	&6.17608	&1469950\\
12	&10	&11	&7	&8	&9	&23.3136	&6.09502	&1557550\\
12	&10	&11	&7	&9	&8	&23.4634	&6.09011	&1628542\\
12	&11	&10	&9	&8	&7	&23.4331	&6.24487	&664558\\
12	&11	&10	&9	&7	&8	&23.3949	&6.24406	&612118\\
12	&11	&10	&8	&9	&7	&23.4178	&6.05007	&863968\\
12	&11	&10	&8	&7	&9	&23.5755	&6.05327	&652000\\
12	&11	&10	&7	&8	&9	&23.5829	&6.00785	&717550\\
12	&11	&10	&7	&9	&8	&23.4634	&6.00985	&877078\\

  \hline
\end{tabular}

The maximum volume of the 30 incongruent tetrahedra corresponds to the following edge lengths
$\{a_{12},a_{43}, a_{13}, a_{23}, a_{24} ,a_{14}\}=\{12, 7,	11,	9, 8, 10\},$ which yields a Fermat tree
having minimal length $22.9123>22.7838.$	

Hence, we need to add one more node $A_{0^{\prime}},$ in order to obtain the Fermat-Steiner-Frechet multitree for the thirty incongruent tetrahedra (Frechet multitetrahedron) derived by the consecutive sextuple of natural numbers $\{12,11,10,9,8\}$ and to reduce the total length for each component of the Fermat-Steiner-Frechet multitree.

Thus, by applying the Lagrangian program of Theorem~\ref{mostnaturalconsecutivesextuples}\\ for $\{a_{12},a_{43}, a_{13}, a_{23}, a_{24} ,a_{14}\}=\{12, 7,	11,	9, 8, 10\},$ we may derive an upper bound for the variable weight $B_{ST}.$

\end{example}

\section{A Theoretical construction of a weighted Fermat-Steiner-Frechet multitree for a Frechet multitetrahedron in $\mathbb{R}^{3}$}

In this section, we describe a theoretical construction to locate a weighted Fermat-Steiner-Frechet multitree for Blumenthal-Herzog, Dekster-Wilker sextuples determining the edge lengths of Frechet multitetrahedra in $\mathbb{R}^{3},$ giving all the necessary notations, which are used to develop a system of two equations, which depend on two variable dihedral angles and some given metric Euclidean elements. This system of equations may be solved using fixed point Banach-Peano functional iteration.

We denote by

$\bullet$ $l_{1}$ a line, which passes through $M_{12}$ and is parallel to the line defined by $A_{3}, A_{4},$

$\bullet$ $l_{2}$ a line, which passes through $T_{12}$ and is parallel to $l_{1},$

$\bullet$ $A_{4}^{\prime},$ $T_{12}^{\prime},$ $H_{34}^{\prime}$ the orthogonal projection of $A_{4},$ $T_{12},$ $H_{34},$ to $l_{1},$ respectively,

$\bullet$ $E_{34}$ the trace of the orthogonal projection of $T_{12}$ to $l_{1}$

$\bullet$ $E_{12}$ the trace of the orthogonal projection of $T_{34}^{\prime}$ to the line defined by $A_{1}, A_{2},$

$\bullet$ $Q$ the intersection point of the line defined by $A_{34}, H_{34}$  with the plane defined by $l_{1}$ and $A_{1},$

$\bullet$ $P$ the intersection point of the line, which passes through $T_{34}$ and is parallel to the line defined by $A_{34}, H_{34}$  with the plane defined by $l_{1}$ and $A_{1},$

$\bullet$ $\delta_{12}$ the dihedral angle defined by
$\triangle A_{1}A_{2}A_{12}$ and the plane perpendicular to
$\vec{a}_{12}\times \vec{a}_{43}$ and by $\delta_{34}$ the
dihedral angle defined by the plane $A_{3}A_{4}A_{34}$ and the
plane perpendicular to $\vec{a}_{12}\times \vec{a}_{43}$

$\bullet$ $A_{34}^{\prime\prime}$ is the trace of the orthogonal projection of $A_{34}$ to the plane, which passes through the line defined by $A_{4}, A_{3}$ and is parallel to the plane defined by $l_{1},$ $A_{1},$

$\bullet$ $A_{34}^{\prime}$ is the trace of the orthogonal projection of $A_{34}^{\prime\prime}$ to the plane defined by $l_{1},$ $A_{1},$

$\bullet$ $A_{12}^{\prime}$ is the trace of the orthogonal projection of $A_{12}$ to the plane defined by $l_{1},$ $A_{1}.$

We set

$\bullet$ $\delta_{34}\equiv \angle A_{34}^{\prime\prime}H_{34}A_{34},$ $\delta_{12}\equiv \angle A_{12}^{\prime}H_{12}A_{12} ,$

$\bullet$ $\omega_{12}\equiv\angle PT_{12}T_{34}^{\prime},$ $\omega_{34}\equiv \angle E_{12}T_{34}^{\prime}T_{12},$

$\bullet$ $\alpha\equiv \angle A_{34}T_{34}A_{34}^{\prime\prime}=\angle A_{12}T_{12}A_{12}^{\prime}.$

We assume that $45^{\circ}<\varphi<90^{\circ}$ and $E_{12}\in [M_{12}, A_{1}],$   $T_{12}\in [A_{1},H_{12}]$
and $T_{34}\in [A_{4},H_{34}]$ (Fig~\ref{Fig111}).

\begin{theorem}\label{mainres1}
$T_{34}^{\prime},$ $P,$ $T_{12},$ $E_{12},$ $E_{34}$ belong to the same circle.
\end{theorem}

\begin{proof}
$P$ is the intersection point of the line $l_{2},$ and the line, which passes through $T_{34}^{\prime}$ and is parallel to the line defined by $H_{34}^{\prime}Q.$ Thus, we obtain that $\delta_{34}=\angle T_{34}^{\prime}PT_{34}=\angle H_{34}^{\prime}QH_{34}$ and $T_{12},$ $P,$ $Q$ are collinear and belong to $l_{2}.$
Therefore, $\angle T_{34}^{\prime}PT_{12}=90^{\circ}$ and taking into account that $\angle T_{34}^{\prime}E_{12}T_{12}=\angle T_{12}E_{34}T_{34}^{\prime}=90^{\circ},$ we derive that $T_{34}^{\prime}T_{12}$ is a diameter of a circle, which is seen by $90^{\circ}$ from $P,$ $E_{12},$ $E_{34}.$ Therefore,  $T_{34}^{\prime},$ $P,$ $T_{12},$ $E_{12},$ $E_{34}$ are concircular.
\end{proof}

\begin{theorem}\label{mainres2}
The position of the weighted Simpson line defined by $T_{12}, T_{34}$ is given by the following two
equations, which depend on $\delta_{12}$ and $\delta_{34}:$

\begin{equation}\label{equat20}
\cot\delta_{12}=\frac{M_{34}H_{34}\sin\varphi+h_{34}\cos\delta_{34}\cos\varphi}{H+h_{34}\sin\delta_{34}}
\end{equation}

\begin{equation}\label{equat23}
\cot\delta_{34}=\frac{M_{12}H_{12}\sin\varphi+h_{12}\cos\delta_{12}\cos\varphi}{H+h_{12}\sin\delta_{12}}.
\end{equation}

\end{theorem}

\begin{proof}

From the theoretical construction of the weighted Simpson line defined by $O_{12}O_{34},$ which intersects
the two edges $A_{4}A_{3},$ $A_{1}A_{2}$ of the tetrahedron $A_{1}A_{2}A_{3}A_{4},$ we get (Fig.\ref{Fig111}):

\begin{figure}
\centering
\includegraphics[scale=0.90]{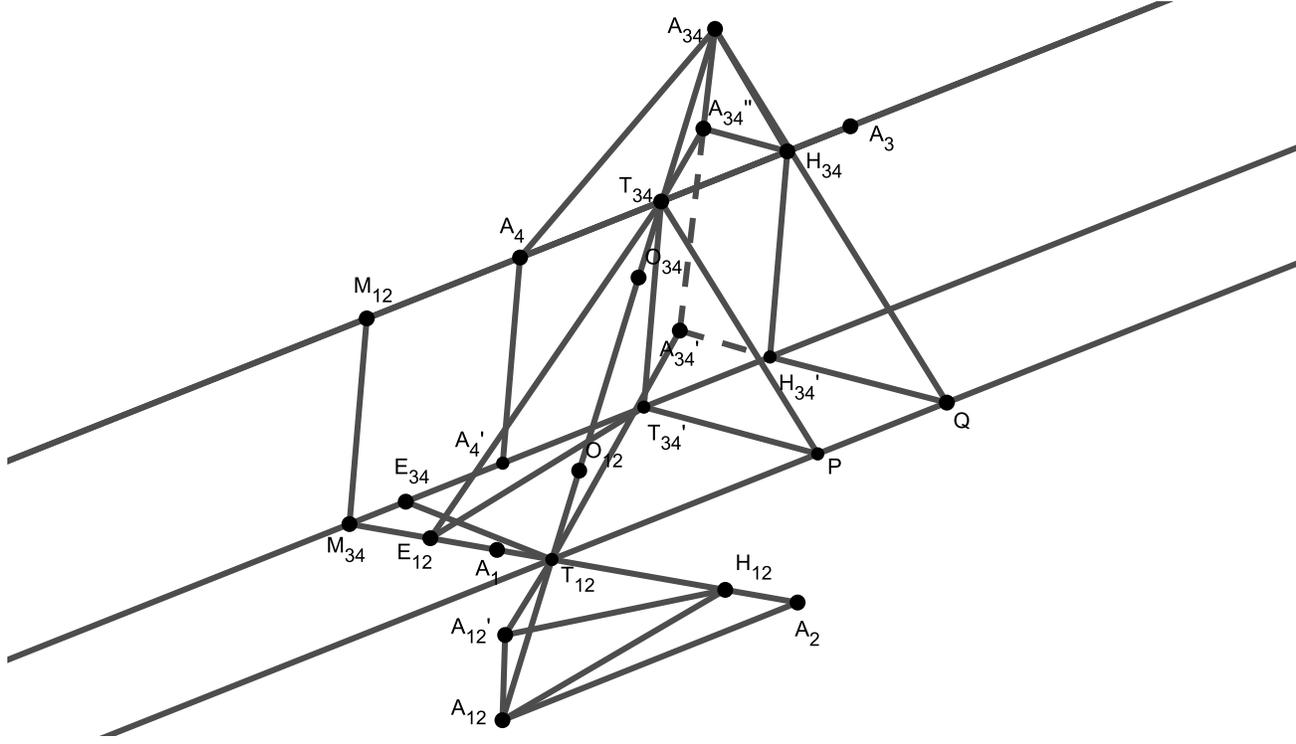}
\caption{Theoretical construction for weighted Fermat-Steiner-Frechet multitree in $\mathbb{R}^{3}$} \label{Fig111}
\end{figure}

\begin{figure}
\centering
\includegraphics[scale=0.90]{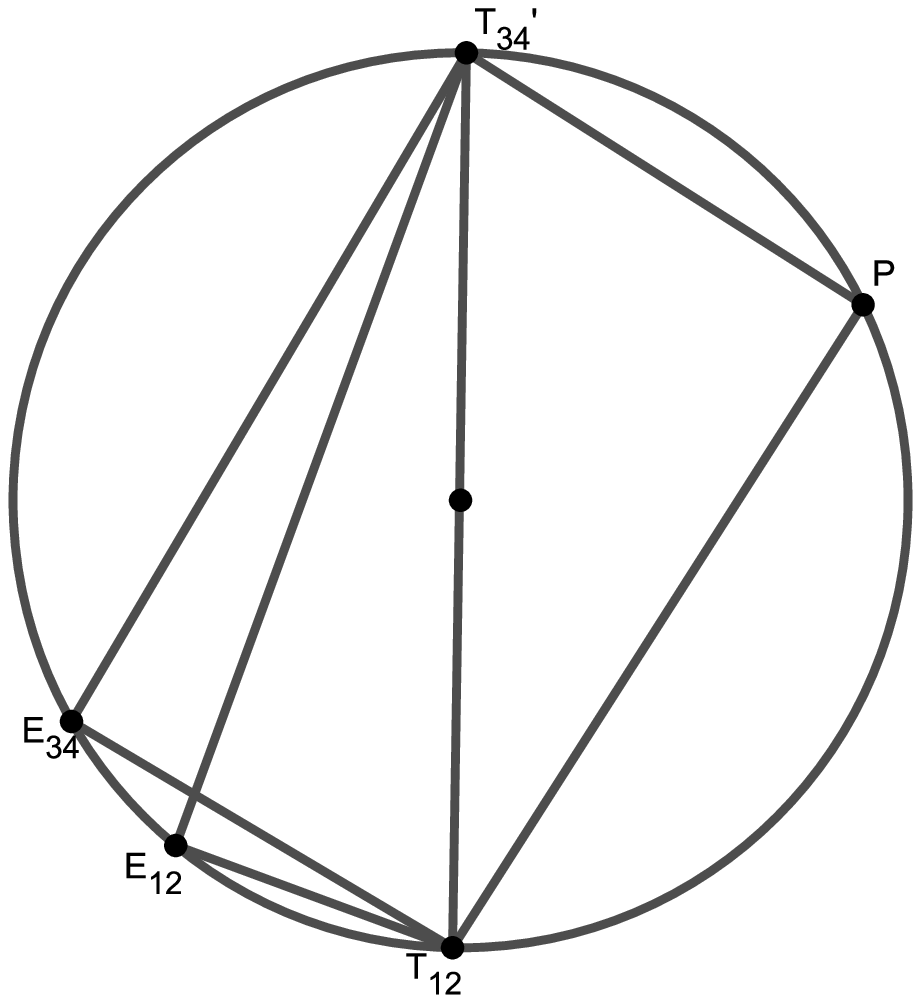}
\caption{} \label{Fig21}
\end{figure}

We mention some useful relations derived by this theoretical construction:

\[\angle T_{34}^{\prime}E_{12}T_{34}=\delta_{12},\]

\begin{equation}\label{mr21}
T_{34}^{\prime}P= H\cot \delta_{34},
\end{equation}

\begin{equation}\label{mr22}
T_{34}^{\prime}T_{12}= H\cot \alpha,
\end{equation}

\begin{equation}\label{mr25}
T_{34}^{\prime}E_{12}= H\cot \delta_{12},
\end{equation}

\begin{equation}\label{mr23}
\angle E_{12}T_{12}P=\pi-\varphi,
\end{equation}

\begin{equation}\label{mr24}
\angle E_{12}T_{34}^{\prime}P=\varphi.
\end{equation}

The similarity of $\triangle A_{12}A_{12}^{\prime}H_{12}$ and
$\triangle T_{34}T_{34}^{\prime}E_{12},$ which are perpendicular to the
line defined by $M_{12}A_{1}H_{12}A_{2}E_{12}$ yields:

\begin{equation}\label{equat16bis}
E_{12}T_{34}^{\prime}=H
\cot\delta_{12}=(M_{12}T_{34}^{\prime})\sin\varphi.
\end{equation}

The similarity of triangles $\triangle A_{34}A_{34}^{\prime\prime}H_{34}$ and
$\triangle T_{34}T_{34}^{\prime}P,$ which are perpendicular to the
parallel lines defined by $M_{34}A_{4}H_{34}A_{3}$ and
$M_{12}A_{4}^{\prime}H_{34}^{\prime}A_{3}^{\prime},$ respectively,
yields:

\begin{equation}\label{equat17}
PT_{34}^{\prime}=H \cot\delta_{34}=(M_{12}T_{12})\sin\varphi,
\end{equation}

because $T_{12}P$ is equal and parallel to $E_{34}T_{34}^{\prime}$
and $T_{12}E_{34}$ is equal and parallel to $PT_{34}^{\prime}.$

Thus, we get:

\begin{equation}\label{mainres2n1}
M_{34}T_{34}=M_{12}T_{34}^{\prime}=M_{12}H_{34}^{\prime}-T_{34}^{\prime}H_{34}^{\prime}
\end{equation}
and

\begin{equation}\label{equat18}
T_{34}^{\prime}H_{34}^{\prime}=T_{34}H_{34}=\sqrt{-h_{34}^{2}\cos^{2}\delta_{34}+(h_{34}\sin\delta_{34}\cot\alpha)^2}=h_{34}\sin\delta_{34}\sqrt{\cot^{2}\alpha-\cot^{2}\delta_{34}}
\end{equation}

By applying Theorem~\ref{mainres1} the quintuple of points $\{E_{34}, T_{12}, E_{12}, P,
T_{34}^{\prime}\}$ belong to the same circle (see Fig.~\ref{Fig21}).

Therefore, we get:
\[\frac{\cot\delta_{12}}{\cot\alpha}=\sin(\omega_{12}+\varphi)=\frac{\cot\delta_{34}}{\cot\alpha}\cos\varphi+\frac{\sin\varphi}{\cot\alpha}\sqrt{\cot^{2}\alpha-\cot^{2}\delta_{34}}\]

or

\begin{equation}\label{equat19}
\sqrt{\cot^{2}\alpha-\cot^{2}\delta_{34}}=\frac{\cot\delta_{12}-\cot\delta_{34}\cos\varphi}{\sin\varphi}.
\end{equation}

By substituting (\ref{equat18}), (\ref{equat19})
(\ref{equat16bis}), (\ref{mainres2n1}), we obtain

\[\frac{H\cot\delta_{12}}{\sin\varphi}=M_{12}H_{34}^{\prime}-h_{34}\sin\delta_{34}\frac{\cot\delta_{12}-\cot\delta_{34}\cos\varphi}{\sin\varphi}\]
which yields (\ref{equat20}).

By following a similar process, we get:

\begin{equation}\label{mr2nn2}
M_{12}T_{12}=M_{12}H_{12}-H_{12}T_{12}
\end{equation}
where

\begin{equation}\label{equat21}
M_{12}T_{12}=h_{12}\sin\delta_{12}\sqrt{\cot^{2}\alpha-\cot^{2}\delta_{12}}.
\end{equation}

Taking into account that the points $E_{34}, T_{12}, E_{12}, F,
T_{34}^{\prime}$ belong to the same circle (Fig.~\ref{Fig21}),
we obtain:

\[\frac{\cot\delta_{34}}{\cot\alpha}=\cos(\varphi -\omega_{34})=\frac{\cot\delta_{12}}{\cot\alpha}\cos\varphi+\frac{\sqrt{\cot^{2}\alpha-\cot^{2}\delta_{12}}}{\cot\alpha}\sin\varphi\]

or

\begin{equation}\label{equat22}
\sqrt{\cot^{2}\alpha-\cot^{2}\delta_{12}}=\frac{\cot\delta_{34}-\cot\delta_{12}\cos\varphi}{\sin\varphi}.
\end{equation}

By substituting (\ref{equat21}), (\ref{equat22}),
(\ref{equat17}) in (\ref{mr2nn2}), we get:

\[\frac{H\cot\delta_{34}}{\sin\varphi}=M_{12}H_{12}+h_{12}\sin\delta_{12}\frac{\cot\delta_{12}\cos\varphi -\cot\delta_{34}}{\sin\varphi},\]

which yields (\ref{equat23}).

\end{proof}

A fixed point iteration (Banach or Picard-Peano) method may be used, in order
to derive the numerical values of $\delta_{12},$ $\delta_{34}.$

We set $x=\cot\delta_{12}$ and $y=\cot \delta_{34}.$

\begin{theorem}\label{mainres2b}
The following two equations

\begin{equation}\label{equat201}
x=\frac{M_{34}H_{34}\sin\varphi\sqrt{1+y^2}+h_{34}\cos\varphi y}{H\sqrt{1+y^2}+h_{34}}
\end{equation}

\begin{equation}\label{equat232}
y=\frac{M_{12}H_{12}\sin\varphi\sqrt{1+x^2}+h_{12}\cos\varphi x}{H\sqrt{1+x^2}+h_{12}}.
\end{equation}

represent a functional iteration w.r. to $x$ or $y:$

\[x=f(x)\]

or

\[y=g(y).\]

\end{theorem}

\begin{proof}

By substituting $x=\cot \delta_{12},$ $y=\cot \delta_{34}$ in (\ref{equat20})
and (\ref{equat23}), we derive (\ref{equat20})and (\ref{equat23}).

We proceed by computing $H,$ $h_{12},$ $h_{34}$ $M_{12}H_{12},$ $M_{34}H_{34}$
with respect to the coordinates of $A_{1},A_{2},A_{3},A_{4}$ and the weights $b_{1},b_{2},b_{3},b_{4},b_{ST}.$

The common perpendicular distance $H=M_{12}M_{34}$ is given by:
\begin{equation}\label{equat8}
H =\frac{ \left|
\begin{array}{ccccc}
x_{4}-x_{1}      & y_{4}-y_{1}      & z_{4}-z_{1}   \\
x_{2}-x_{1}      & y_{2}-y_{1}      & z_{2}-z_{1}          \\
x_{3}-x_{4}      & y_{3}-y_{4}      & z_{3}-z_{4}          \\

\end{array} \right|}{a_{12}a_{34}\sin\varphi}.
\end{equation}

The angle $\varphi$ is given by:

\[\varphi=\arccos(\frac{\vec{a}_{12}\cdot
\vec{a}_{43}}{a_{12}a_{43}}).\]


The line segments $M_{12}A_{1}$ and $M_{12}A_{4}^{\prime}$ are
given by

\begin{equation}\label{equat9}
M_{12}A_{1} =\frac{ \left|
\begin{array}{ccccc}
x_{4}-x_{1}      & y_{4}-y_{1}      & z_{4}-z_{1}   \\
x_{3}-x_{4}      & y_{3}-y_{4}      & z_{3}-z_{4}          \\
v_{1}      & v_{2}      & v_{3}          \\

\end{array} \right|}{v_{1}^2+v_{2}^2+v_{3}^2}a_{12}
\end{equation}

and

\begin{equation}\label{equat10}
M_{12}A_{4}^{\prime}=M_{34}A_{4}=\frac{ \left|
\begin{array}{ccccc}
x_{4}-x_{1}      & y_{4}-y_{1}      & z_{4}-z_{1}   \\
x_{2}-x_{1}      & y_{2}-y_{1}      & z_{2}-z_{1}          \\
v_{1}      & v_{2}      & v_{3}          \\

\end{array} \right|}{v_{1}^2+v_{2}^2+v_{3}^2}a_{43}
\end{equation}

where

\[v_{1}=(y_{2}-y_{1})(z_{3}-z_{4})-(y_{3}-y_{4})(z_{2}-z_{1}),\]
\[v_{2}=(z_{2}-z_{1})(x_{3}-x_{4})-(x_{2}-x_{1})(z_{3}-z_{4}),\]
and
\[v_{3}=(x_{2}-x_{1})(y_{3}-y_{4})-(x_{3}-x_{4})(y_{2}-y_{1}).\]

By taking into account the equations derived by the angular solution of Lemma~\ref{angularsolutionweightedFermatSteiner}, we get:

\begin{eqnarray}\label{FTangles}
 \cos\alpha_{12}& = & \frac{b_{ST}^2-b_{1}^2-b_{2}^2}{2b_{1}b_{2}},\nonumber \\
\cos\alpha_{1} & = & \frac{b_{1}^2-b_{2}^2-b_{ST}^2}{2b_{2}b_{ST}},\nonumber \\
\cos\alpha_{34} & = &\frac{b_{ST}^2-b_{3}^2-b_{4}^2}{2b_{3}b_{4}},\nonumber \\
\cos\alpha_{4} & =& \frac{b_{4}^2-b_{3}^2-b_{ST}^2}{2b_{3}b_{ST}}.
\end{eqnarray}

or

\begin{equation}\label{equat5}
\frac{b_{1}}{\sin\alpha_{1}}=\frac{b_{2}}{\sin\alpha_{2}}=\frac{b_{ST}}{\sin\alpha_{12}}.
\end{equation}

\begin{equation}\label{equat6}
\frac{b_{3}}{\sin\alpha_{3}}=\frac{b_{4}}{\sin\alpha_{4}}=\frac{b_{ST}}{\sin\alpha_{34}}.
\end{equation}

\begin{eqnarray}\label{eqar1}
 \cos\alpha_{12}& = & \frac{b_{ST}^2-B_{1}^2-b_{2}^2}{2b_{1}b_{2}}\nonumber \\
\cos\alpha_{1} & = & \frac{b_{1}^2-b_{2}^2-b_{ST}^2}{2b_{2}b_{ST}}\nonumber \\
\cos\alpha_{34} & = &\frac{b_{ST}^2-b_{3}^2-b_{4}^2}{2b_{3}b_{4}}\nonumber \\
\cos\alpha_{4} & =& \frac{b_{4}^2-b_{3}^2-b_{ST}^2}{2b_{3}b_{ST}}
\end{eqnarray}

From (\ref{equat5}) and $\triangle A_{1}A_{2}A_{12}$ we get:

\begin{equation}\label{equat11}
\left(A_{1}A_{12}\right)=a_{12}\frac{b_{2}}{b_{ST}}
\end{equation}

and the height

\begin{equation}\label{equat12}
\left(A_{12}H_{12}\right)\equiv
h_{12}=A_{1}A_{2}\left(\frac{b_{2}}{b_{ST}}
\right)\sqrt{1-\left(\frac{b_{1}^2-b_{2}^2-b_{ST}^2}{2b_{2}b_{ST}}
\right)^{2}},
\end{equation}

\begin{equation}\label{equat13}
\left(A_{1}H_{12}\right)=A_{1}A_{2}\frac{b_{2}}{b_{ST}}\frac{-b_{1}^2+b_{2}^2+b_{ST}^2}{2b_{2}b_{ST}}.
\end{equation}

By adding (\ref{equat9}) and (\ref{equat13}), we obtain:

\[M_{12}H_{12}=M_{12}A_{1}+A_{1}H_{12}.\]

From (\ref{equat6}) and $\triangle A_{3}A_{4}A_{34}$ we get:

\begin{equation}\label{equat14}
\left(A_{4}A_{34}\right)=a_{34}\frac{b_{3}}{b_{ST}}
\end{equation}

and the height

\begin{equation}\label{equat15}
\left(A_{34}H_{34}\right)\equiv
h_{34}=a_{34}\left(\frac{b_{3}}{b_{ST}}
\right)\sqrt{1-\left(\frac{b_{4}^2-b_{3}^2-b_{ST}^2}{2b_{3}b_{ST}}
\right)^{2}},
\end{equation}

\begin{equation}\label{equat16}
\left(A_{4}H_{34}\right)=a_{34}\frac{b_{3}}{b_{ST}}\frac{-b_{4}^2+b_{3}^2+b_{ST}^2}{2b_{3}b_{ST}}.
\end{equation}



By adding (\ref{equat10}) and (\ref{equat16}), we obtain:

\[M_{34}H_{34}=M_{34}A_{4}+A_{4}H_{34}.\]

\end{proof}

\begin{theorem}\label{mainres3}

The location of $O_{12}$ and $O_{34}$ is given by the computation
of the line segments $A_{1}O_{12},$ $A_{2}O_{12},$ $A_{3}O_{34},$ $A_{4}O_{34},$ with respect to
$b_{1}, b_{2},b_{3},b_{ST},$ $A_{1}T_{12},$ $A_{2}T_{2},$ $A_{3}T_{34},$ $A_{4}T_{34}:$

\begin{equation}\label{loc1}
A_{2}O_{12}=A_{1}O_{12}\left(\frac{A_{2}T_{12}}{A_{1}T_{12}}\right)\left(\frac{b_{2}}{b_{1}}\right),
\end{equation}
\begin{equation}\label{loc2}
A_{1}O_{12}=A_{1}A_{2}\left[\sqrt{\left(\frac{A_{2}T_{12}}{A_{1}T_{12}}\right)^{2}\left(\frac{b_{2}}{b_{1}}\right)^{2}-2\left(\frac{A_{2}T_{12}}{A_{1}T_{12}}\right)\left(\frac{b_{2}}{b_{1}}\right)\cos\alpha_{12}+1}\right]^{-1},
\end{equation}

\begin{equation}\label{loc3}
A_{3}O_{34}=A_{4}O_{34}\left(\frac{A_{3}T_{34}}{A_{4}T_{34}}\right)\left(\frac{b_{3}}{b_{4}}\right),
\end{equation}

\begin{equation}\label{loc4}
A_{4}O_{34}=A_{3}A_{4}\left[\sqrt{\left(\frac{A_{3}T_{34}}{A_{4}T_{34}}\right)^{2}\left(\frac{b_{3}}{b_{4}}\right)^{2}-2\left(\frac{A_{3}T_{34}}{A_{4}T_{34}}\right)\left(\frac{b_{3}}{b_{4}}\right)\cos\alpha_{34}+1}\right]^{-1}.
\end{equation}

taking into account that:

\[A_{2}T_{12}=A_{1}A_{2}-A_{1}T_{12},\]
\[A_{1}T_{12}=M_{12}T_{12}-M_{12}A_{1},\]

\[A_{4}T_{34}=M_{34}T_{34}-M_{34}A_{4},\]
\[A_{3}T_{34}=A_{3}A_{4}-A_{4}T_{34}.\]

\end{theorem}

\begin{proof}

By applying the sine law in $\triangle A_{1}O_{12}T_{12},$ $\triangle A_{2}O_{12}T_{12},$ $A_{1}O_{12}A_{12},$
$\triangle A_{2}O_{12}A_{12},$ we get (\ref{loc1}).

By applying the sine law in $\triangle A_{3}O_{34}T_{34},$ $\triangle A_{4}O_{34}T_{34},$ $A_{3}O_{34}A_{34},$
$\triangle A_{4}O_{34}A_{34},$ we get (\ref{loc3}).

By substituting (\ref{loc1}) and  $\cos\alpha_{12}$ from the cosine law in $\triangle A_{1}O_{12}A_{2},$ on the right hand side of (\ref{loc2}),  we derive $A_{1}O_{12}.$

By substituting (\ref{loc3}) and  $\cos\alpha_{34}$ from the cosine law in $\triangle A_{3}O_{34}A_{4},$ on the right hand side of (\ref{loc4}),  we derive $A_{4}O_{34}.$

\end{proof}

\begin{example}\label{ex1}
Given

$b_{1}=0.85,$ $b_{2}=0.88,$ $b_{3}=0.83,$ $b_{4}=1.08,$ $b_{ST}=1,$

$A_{1}=(2,0,0),$ $A_{2}=(6.86,1.37,0),$ $A_{3}=(0,6,5),$ $A_{4}=(0,0,5)$
we derive that $\varphi=74,25^{\circ},$ $H=5,$ $M_{12}=(0,-0.56,0),$ $M_{34}=(0,-0.56,5).$
By inserting these data in Theorem~\ref{mainres2b}, we obtain that

$\delta_{12}=58.58^{\circ},$ $\delta_{34}=57.75^{\circ},$ $\alpha=52.08^{\circ}.$

We note that the quintuple of points deduced from this construction $P=(3.16,2.61,0),$ $E_{34}=(0,0.33,0),$ $T_{34}^{\prime}=(0,2.61,0),$ $T_{12}=(3.16,0.33,0)$
$E_{12}=(0.83,-0.33,0)$ belong to the same circle.

\end{example}

\begin{theorem}\label{mainresmultitreetetr2}
There are upto $4!\cdot 30$ weighted Simpson lines defined by $T_{12}, T_{34},$ such that $O_{12}, O_{34}$ in $[T_{12},T_{34}],$ which give the position of the weighted Fermat-Steiner-Frechet multitree for 30 incongruent tetrahedra determined by Blumenthal, Herzog and Dekster-Wilker sextuples of edge lengths in $\mathbb{R}^{3}$
\end{theorem}

\begin{proof}
It is a direct consequence of Theorems~\ref{mainres2},\ref{mainres3}, taking into account 30 incongruent tetrahedra derived by Blumenthal,Herzog,Dekster-Wilker sextuples multiplied by the permutation of the four weights $\{b_{1},b_{2},b_{3},b_{4}\}$ $4!,$ such that the two equal weights $b_{12}$ and $b_{34}$ that correspond
to the two weighted Fermat-Steiner points have the same constant value $b_{ST}.$ Therefore, we derive upto $4!\dot 30$ weighted Fermat-Steiner trees, which yield a weighted Fermat-Steiner-Frechet multitree in $\mathbb{R}^{3}.$
\end{proof}

\begin{remark}
The position of the weighted Fermat-Steiner-Frechet multitree for 30 incongruent tetrahedra determined by Blumenthal, Herzog and Dekster-Wilker sextuples of edge lengths in $\mathbb{R}^{3}$ may also be derived by a fixed point functional iteration method, which computes some variable lengths instead of variable dihedral angles (see in \cite{Zachos:21}).
\end{remark}


\section{Plasticity of weighted Fermat-Frechet multitrees for boundary closed polytopes in $\mathbb{R}^{N}$}

In this section, we find the unique solution of the $4-$INVWF problem in $\mathbb{R}^{4},$ which depends on exactly nine given angles. We continue, by deriving the unique solution w.r to the $N-$INVWF problem in $\mathbb{R}^{N},$ which depends on exactly $\frac{N(N+1)}{2}-1$ and the non-unique solution (dynamic plasticity) of the $(N+1)-$ INVWF problem in $\mathbb{R}^{N}.$ The dynamic plasticity of the $(N+1)-$ INVWF problem in $\mathbb{R}^{N}$ leads to the plasticity of weighted Fermat-Frechet multitrees for boundary closed polytopes in $\mathbb{R}^{N}.$


\begin{theorem}[Solution of the $4-$INVWF problem in $\mathbb{R}^{4}$]\label{5inverseR4}
The weight $B_{i}$ is uniquely determined by:
\begin{eqnarray}\label{inverse111}
B_{i}=\frac{C}{1+\abs{\frac{\sin{\alpha_{i,k0lm}}}{\sin{\alpha_{j,k0lm}}}}+\abs{\frac{\sin{\alpha_{i,j0lm}}}{\sin{\alpha_{k,j0lm}}}}+\abs{\frac{\sin{\alpha_{i,k0jm}}}{\sin{\alpha_{l,k0jm}}}}
+\abs{\frac{\sin{\alpha_{i,k0jl}}}{\sin{\alpha_{m,k0jl}}}}},
\end{eqnarray}
for $i,j,k,l,m=1,2,3,4,5$ and $i \neq j\neq k\neq l\neq m.$
\end{theorem}

\begin{proof}
We consider the following five unit vectors $\vec{u}(A_{0},A_{i}),$\\ for $i=1,2,...,5,$ which meet at the weighted Fermat point $A_{0}:$
\begin{eqnarray}\label{veca1}
\vec{u}(A_{0},A_{1})=(1,0,0,0),
\end{eqnarray}
\begin{eqnarray}\label{veca2}
\vec{u}(A_{0},A_{2})=(\cos \alpha_{102},\sin \alpha_{102},0,0),
\end{eqnarray}
\begin{eqnarray}\label{veca3}
\vec{u}(A_{0},A_{3})=(\cos a_{3,102} \cos \omega_{3,102},\cos a_{3,102}\sin\omega_{3,102},\sin a_{3,102},0),
\end{eqnarray}
\begin{eqnarray}\label{veca4}
\vec{u}(A_{0},A_{4})=(\cos a_{4,1023} \cos \omega_{4,1023}\cos z_{4,1023},
\cos a_{4,1023} \cos \omega_{4,1023}\sin z_{4,1023},\nonumber\\
\cos a_{4,1023} \sin \omega_{4,1023},\sin a_{4,1023}),\nonumber\\
\end{eqnarray}
\begin{eqnarray}\label{veca5}
\vec{u}(A_{0},A_{5})=(\cos a_{5,1023} \cos \omega_{5,1023}\cos z_{5,1023},
\cos a_{5,1023} \cos \omega_{5,1023}\sin z_{5,1023},\nonumber\\
\cos a_{5,1023} \sin \omega_{5,1023},\sin a_{5,1023}). \nonumber\\
\end{eqnarray}

Taking into account (\ref{veca1})-(\ref{veca4}) the inner products $\vec{u}(A_{0},A_{4})\cdot \vec{u}(A_{0},A_{i})$  for $i=1,2,3$ yield:
\begin{eqnarray}\label{104}
\cos a_{4,1023}\cos \omega_{4,1023}\cos z_{4,1023}=\cos \alpha_{104},
\end{eqnarray}

\begin{eqnarray}\label{204}
\cos a_{4,1023}\cos \omega_{4,1023}\sin z_{4,1023}=\frac{1}{\sin \alpha_{102}}(\cos \alpha_{204}-\cos \alpha_{102}\cos \alpha_{104}),\nonumber\\
\end{eqnarray}

\begin{eqnarray}\label{304}
\cos a_{4,1023}\sin\omega_{4,1023}=\frac{1}{\sin a_{3,102}}(\cos \alpha_{304}-\cos \alpha_{103}\cos\alpha_{104}-\nonumber \\
-(\frac{\cos \alpha_{203}-\cos\alpha_{102}\cos\alpha_{103}}{\sin \alpha_{102}})(\frac{\cos\alpha_{204}-\cos\alpha_{102}\cos\alpha_{104}}{\sin\alpha_{102}})).\nonumber\\
\end{eqnarray}

By squaring both parts of (\ref{104}) and (\ref{204}) and by adding the two derived equations, we eliminate $z_{4,1023}:$

\begin{eqnarray}\label{newimp1}
\cos^{2} a_{4,1023}\cos^{2} \omega_{4,1023}=\frac{1}{\sin^{2} \alpha_{102}}(\cos \alpha_{204}-\cos \alpha_{102}\cos \alpha_{104})^2 + \cos^{2} \alpha_{104}.\nonumber\\
\end{eqnarray}

By squaring both parts of (\ref{newimp1}) and (\ref{304}) and by adding the two derived equations, we eliminate $\omega_{4,1023},$ which yields that
$\cos^{2} a_{4,1023}$ depends on six angles $\alpha_{102},$ $\alpha_{103},$ $\alpha_{104},$ $\alpha_{203},$ $\alpha_{204}$ and $\alpha_{304},$ because $a_{3,102}$ depends on $\alpha_{102},$ $\alpha_{103},$ $\alpha_{203}$ (\cite[(3.10),p.~120]{Zach/Zou:09}):

\[\cos^{2} a_{3,102}=\frac{\cos^{2} \alpha_{203}+\cos^{2} \alpha_{103}-2\cos\alpha_{102}\cos\alpha_{103}\cos\alpha_{203}}{\sin^{2}\alpha_{102}}.\]

By taking into account (\ref{veca1})-(\ref{veca4}) the inner products $\vec{u}(A_{0},A_{5})\cdot \vec{u}(A_{0},A_{i})$  for $i=1,2,3$ yield:
\begin{eqnarray}\label{105}
\cos a_{5,1023}\cos \omega_{5,1023}\cos z_{5,1023}=\cos \alpha_{105},
\end{eqnarray}

\begin{eqnarray}\label{205}
\cos a_{5,1023}\cos \omega_{5,1023}\sin z_{5,1023}=\frac{1}{\sin\alpha_{102}}\cos \alpha_{205}-\cos \alpha_{102}\cos \alpha_{105},\nonumber\\
\end{eqnarray}
\begin{eqnarray}\label{305}
\cos a_{5,1023}\sin\omega_{5,1023}=\frac{1}{\sin a_{3,102}}(\cos \alpha_{305}-\cos \alpha_{103}\cos\alpha_{105}-\nonumber \\
(\frac{\cos \alpha_{203}-\cos\alpha_{102}\cos\alpha_{103}}{\sin \alpha_{102}})(\frac{\cos\alpha_{205}-\cos\alpha_{102}\cos\alpha_{105}}{\sin\alpha_{102}})).\nonumber\\
\end{eqnarray}
\begin{eqnarray}\label{405}
\cos \alpha_{405}=\cos \alpha_{105}\cos\alpha_{104}+\nonumber \\
+(\frac{\cos \alpha_{205}-\cos\alpha_{102}\cos\alpha_{105}}{\sin \alpha_{102}})(\frac{\cos\alpha_{204}-\cos\alpha_{102}\cos\alpha_{104}}{\sin\alpha_{102}})+\nonumber\\
+\cos a_{4,1023}\sin \omega_{4,1023} \cos a_{5,1023}\sin \omega_{5,1023}+\nonumber\\
+\sin a_{4,1023}\sin a_{5,1023}.
\end{eqnarray}

By squaring both parts of (\ref{105}) and (\ref{205}) and by adding the two derived equations, we eliminate $z_{5,1023}:$

\begin{eqnarray}\label{newimp15}
\cos^{2} a_{5,1023}\cos^{2} \omega_{5,1023}=\frac{1}{\sin^{2} \alpha_{102}}(\cos \alpha_{205}-\cos \alpha_{102}\cos \alpha_{105})^2 + \cos^{2} \alpha_{105}.\nonumber\\
\end{eqnarray}

By squaring both parts of (\ref{newimp15}) and (\ref{305}) and by adding the two derived equations, we eliminate $\omega_{5,1023},$ which yields that
$\cos^{2} a_{5,1023}$ depends on six angles $\alpha_{102},$ $\alpha_{103},$ $\alpha_{105},$ $\alpha_{203},$ $\alpha_{205}$ and $\alpha_{305}.$ Thus, we get:\\ $a_{5,1023}=a_{5,1023}(\alpha_{102},\alpha_{203},\alpha_{103},\alpha_{105},\alpha_{205},\alpha_{305}).$
\\Therefore, (\ref{405}) yields:

$\alpha_{405}=\alpha_{405}(\alpha_{102},\alpha_{203},\alpha_{103},\alpha_{104},\alpha_{204},\alpha_{304},\alpha_{105},\alpha_{205}, \alpha_{305}).$

The angle $a_{i,1023}$ is the angle formed by $\vec{u}(A_{0},A_{i})$ and $\vec{u}(A_{0},A_{i,1023}),$ where $A_{i,1023}$ is the projection of $A_{i}$ w.r. to the subspace of $A_{1},A_{0},A_{2},A_{4},$ for $i=4,5.$

The normal $\vec{n}_{1023}$ is orthogonal to $A_{1}A_{0}A_{2}A_{3}.$

Thus, we obtain:
\begin{eqnarray}\label{extprodnsin}
\vec{n}_{0,123}\cdot \vec{u}(A_{0},A_{i})=\sin a_{i,1023},
\end{eqnarray}
for $i=4,5.$

The weighted floating condition (\ref{weightedfloatingft})for $N=5$ yields:

\begin{equation}\label{flotcond1}
\sum_{i=1}^{5}B_{i}\vec{u}(A_{0},A_{i})=\vec{0},.
\end{equation}

By taking the inner product of (\ref{flotcond1}) with $ \vec{n}_{0,123},$ we derive that:

\[\sum_{i=1}^{5}B_{i}\vec{u}(A_{0},A_{i})\cdot \vec{n}_{0,123}=\vec{0}\]
or
\begin{eqnarray}\label{flotcond1extern0123}
B_{4}\sin a_{4,1023}=B_{5} \sin a_{5,1023},
\end{eqnarray}

where $A_{4}$ belongs to the upper half space defined by $A_{1}A_{0}A_{2}A_{3}$ and $A_{5}$ belongs to the
corresponding lower half space.

By working similarly and by exchanging the indices $\{1,2,3,4 \}$ of the vertices $A_{i}$ for $i=1,...,4,$
we obtain:

\begin{eqnarray}\label{flotcond1extern0ijk}
B_{m}\sin a_{m,i0jk}=B_{n} \sin a_{n,i0jk},
\end{eqnarray}
for $m,n,i,j,k=1,2,3,4,5, m\neq n \neq i\neq j\neq k.$

By replacing (\ref{flotcond1extern0ijk}) in (\ref{isoperimetricconditionweights}) for $m=5,$ we derive (\ref{inverse111}). Thus, the weights $B_{i}$ depend on nine given angles\\  $\alpha_{102}, \alpha_{203}, \alpha_{103}, \alpha_{104}, \alpha_{204}, \alpha_{304}, \alpha_{105}, \alpha_{205}, \alpha_{305}.$

\end{proof}

\begin{theorem}[Solution of the $N-$INVWF problem in $\mathbb{R}^{N}$]\label{nplus1inverseRn}
The weight $B_{i}$ is uniquely determined by:
\begin{eqnarray}\label{inverse111n}
B_{i}=\frac{C}{1+\abs{\frac{\sin{\alpha_{i,0k_{1}k_{2}\ldots k_{N-1}}}}{\sin{\alpha_{k_{N},0k_{1}k_{2}\ldots k_{N-1}}}}}+\abs{\frac{\sin{\alpha_{i,0k_{1}k_{2}\ldots k_{N-2}k_{n}}}}{\sin{\alpha_{k_{N-1},0k_{1}k_{2}\ldots k_{N-2}k_{N}}}}}+\ldots+\abs{\frac{\sin{\alpha_{i,0k_{2}\ldots k_{N-1}k_{N}}}}{\sin{\alpha_{k_{1},0k_{2}\ldots k_{N-1}k_{N}}}}}},\nonumber\\
\end{eqnarray}
for $i, k_{1},k_{2},...,k_{N}=1,2,...,N+1$ and $k_{1} \neq k_{2}\neq...\neq k_{N}.$
\end{theorem}

\begin{proof}
We consider $N+1$ unit vectors $\vec{u}(A_{0},A_{i})\in \mathbb{R}^{N},$ \\ for $i=1,2,\ldots,N+1,$ which meet at the weighted Fermat point $A_{0}:$
\begin{eqnarray}\label{veca1n}
\vec{u}(A_{0},A_{1})=(1,0,\ldots,0),
\end{eqnarray}
\begin{eqnarray}\label{veca2n}
\vec{u}(A_{0},A_{2})=(\cos \alpha_{102},\sin \alpha_{102},0,\ldots,0),
\end{eqnarray}
\begin{eqnarray}\label{veca3n}
\vec{u}(A_{0},A_{3})=(\cos a_{3,102} \cos \omega_{3,102},\cos a_{3,102}\sin\omega_{3,102},\sin a_{3,102},0,\ldots 0),\nonumber\\
\end{eqnarray}
\begin{eqnarray}\label{veca4n}
\vec{u}(A_{0},A_{4})=(\cos a_{4,1023} \cos \omega_{4,1023}\cos z_{4,1023},
\cos a_{4,1023} \cos \omega_{4,1023}\sin z_{4,1023},\nonumber\\
\cos a_{4,1023} \sin \omega_{4,1023},\sin a_{4,1023},0,\ldots,0),\nonumber\\
\end{eqnarray}
$\vdots$
\begin{eqnarray}\label{vecann}
\vec{u}(A_{0},A_{N})=(\cos a_{N,1023\ldots N-1} \cos \omega_{(1)N,1023\ldots N-1} \ldots
\cos\omega_{(N-2)N,1023\ldots N-1},\nonumber\\
\cos a_{N,1023\ldots N-1} \cos \omega_{(1)N,1023\ldots N-1} \ldots
\sin\omega_{(N-2)N,1023\ldots N-1},\nonumber\\
\ldots,\sin a_{N,1023\ldots N-1}), \nonumber\\
\end{eqnarray}
\begin{eqnarray}\label{vecannplusone}
\vec{u}(A_{0},A_{N+1})=(\cos a_{N+1,1023\ldots N-1} \cos \omega_{(1)(N+1),1023\ldots N-1} \ldots
\cos\omega_{(N-2)(N+1),1023\ldots N-1},\nonumber\\
\cos a_{N+1,1023\ldots N-1} \cos \omega_{(1)(N+1),1023\ldots N-1} \ldots
\sin\omega_{(N-2)(N+1),1023\ldots N-1},\nonumber\\
\ldots,\sin a_{N+1,1023\ldots N-1}). \nonumber\\
\end{eqnarray}

By following the same process that used in the proof of Theorem~\ref{5inverseR4}
and by using induction the angles $\omega_{(i)(N),1023\ldots N-1},\omega_{(i)(N+1),1023\ldots N-1},$
$\dots \omega_{4,1023},$

$\omega_{3,102},z_{4,1023}$ are eliminated.

$\cos^{2} a_{N,1023\ldots N-1},$ depend on $\frac{N(N-1)}{2}$ angles $\alpha_{102},$ $\alpha_{103},$ $\alpha_{10(N-2},$ $\alpha_{10N},$ $\alpha_{203},$ $\ldots$ $\alpha_{(N-2)0(N}.$
$\cos^{2} a_{N+1,1023\ldots N-1}$ depend on $\frac{N(N-1)}{2}$ angles $\alpha_{102},$ $\alpha_{103},$ $\alpha_{10(N-2},$ $\alpha_{10N-2},$ $\alpha_{10(N+1)},$ $\alpha_{203},$ $\ldots$ $\alpha_{(N-2)0(N+1}.$

The inner product $\vec{u}(A_{0}A_{N})\cdot \vec{u}(A_{0},A_{N+1})$ (\ref{405}) yields:

$\alpha_{N0N+1}=\alpha_{N0N+1}(\alpha_{102},\alpha_{103},\alpha_{10N-2},\alpha_{10N},\alpha_{10(N+1)},\ldots\alpha_{(N-2)0(N+1)}).$

The angle $a_{i,1023\ldots(N-1)}$ is the angle formed by $\vec{u}(A_{0},A_{i})$ and $\vec{u}(A_{0},A_{i,1023\ldots(N-1)}),$ where $A_{i,1023\ldots(N-1}$ is the projection of $A_{i}$ w.r. to the subspace of $A_{1}A_{0}A_{2}A_{4}\ldots A_{N-1}$ for $i=N,N+1.$

The normal $\vec{n}_{1023\ldots(n-1)}$ is orthogonal to $A_{1}A_{0}A_{2}\ldots A_{N-1}.$

Hence, we get:
\begin{eqnarray}\label{extprodnsinn}
\vec{n}_{0,123\ldots(N-1)}\cdot \vec{u}(A_{0},A_{i})=\sin a_{i,1023\ldots(N-1)},
\end{eqnarray}
for $i=N,N+1.$

The weighted floating condition (\ref{weightedfloatingft}) yields:

\begin{equation}\label{flotcond1n}
\sum_{i=1}^{N+1}B_{i}\vec{u}(A_{0},A_{i})=\vec{0}.
\end{equation}

By taking the inner product of (\ref{flotcond1n}) with $ \vec{n}_{0,123\ldots(N-1)},$ we derive that:

\[\sum_{i=1}^{N+1}B_{i}\vec{u}(A_{0},A_{i})\cdot \vec{n}_{0,123\ldots(N-1)}=\vec{0}\]
or
\begin{eqnarray}\label{flotcond1extern0123n}
B_{N}\sin a_{N,1023\ldots(N-1)}=B_{N+1} \sin a_{N+1,1023\ldots(N-1)}.
\end{eqnarray}

By working similarly and by exchanging the indices $\{1,2,3,\ldots,N,N+1 \}$ of the vertices $A_{i}$ for $i=1,\ldots,N+1,$
we obtain:

\begin{eqnarray}\label{flotcond1extern0ijkn}
B_{i_{m}}\sin a_{i_{m},i_{1}0i_{2}\ldots i_{N-1}}=B_{i_{N}} \sin a_{i_{N},i_{1}0i_{2}\ldots i_{N-1}},
\end{eqnarray}
for $m,N,i_{m},i_{N}\in \{1,2,\ldots,N+1\}, m\neq n , i_{m} \neq i_{N}.$

By replacing (\ref{flotcond1extern0ijkn}) in (\ref{isoperimetricconditionweights}), we obtain (\ref{inverse111n}). Thus, the weights $B_{i}$ depend on $\frac{N(N+1)}{2}-1$ given angles  $\alpha_{i0j}$, for $i,j=1,2,\ldots,N+1, i\neq j$ and $\alpha_{N0N+1}=\alpha_{N0N+1}(\alpha_{i0j}).$

\end{proof}

We set $\vec{a}_{i}\equiv \overrightarrow{A_{0}A_{i}}.$ The corner $(A_{0},A_{1}A_{2}\ldots A_{n})$with $N$ edges is determined by $\vec{a}_{1},\vec{a}_{2},\ldots \vec{a}_{N}.$
The $N-$dimensional polar sine of the corner $(A_{0},A_{1}A_{2}\ldots A_{N})$ is defined in \cite[p.~76]{Eriksson:78}:

\[polsin_{N}(A_{0},A_{1}A_{2}\ldots A_{N})=\frac{\abs{[\vec{a}_{1},\vec{a}_{2},\ldots,\vec{a}_{N}]}}{\abs{\vec{a}_{1}}\abs{\vec{a}_{2}}\ldots \abs{\vec{a}_{N}}},\]
where $[\vec{a}_{1},\vec{a}_{2},\ldots,\vec{a}_{N}]$ is the content of an $N$ dimensional parallelotope with sides
$\abs{\vec{a}_{1}}\abs{\vec{a}_{2}}\ldots \abs{\vec{a}_{N}}.$


\begin{theorem}[Weighted volume equalities in $\mathbb{R}^{N}$ ]\label{RatioVolumesSimplex}

\begin{eqnarray}\label{ratiovolumesn}
\frac{B_{1}}{a_{1}Vol(A_{0}A_{2}A_{3}...A_{N+1})}=\frac{B_{2}}{a_{2}Vol(A_{1}A_{0}...A_{N+1})}=...=\\ \nonumber
=\frac{B_{N+1}}{a_{N+1}Vol(A_{1}A_{2}...A_{0})}=\frac{\sum_{i=1}^{N+1}\frac{B_{i}}{a_{i}}}{Vol(A_{1}A_{2}...A_{N+1})}.
\end{eqnarray}
\end{theorem}

\begin{proof}

By multiplying both members of (\ref{flotcond1extern0123n}) with\\ $\prod_{i=1}^{N+1}\abs{\vec{a}_{i}} polsin_{N-1}(A_{0},A_{1}A_{2}\ldots A_{N-1}),$
we get:

\begin{eqnarray}\label{volumepolsines1}
B_{N}\prod_{i=1}^{N+1}\abs{\vec{a}_{i}} polsin_{N-1}(A_{0},A_{1}A_{2}\ldots A_{N-1})\sin a_{N,1023\ldots N-1}=\nonumber\\
=B_{N+1}\prod_{i=1}^{N+1}\abs{\vec{a}_{i}} polsin_{N-1}(A_{0},A_{1}A_{2}\ldots A_{N-1}) \sin a_{N+1,1023\ldots N-1},
\end{eqnarray}

The volumes of the $N-$-dimensional simplexes $A_{0}A_{1}A_{2}\ldots A_{N-1}A_{N},$\\ $A_{0}A_{1}A_{2}\ldots A_{n-1}A_{n+1},$ are given by:
\begin{eqnarray}\label{volumepolsines2}
N! Vol(A_{0}A_{1}A_{2}\ldots A_{N-1}A_{N})=\prod_{i=1}^{N+1}\abs{\vec{a}_{i}} polsin_{N}(A_{0},A_{1}A_{2}\ldots A_{N})
\end{eqnarray}
\begin{eqnarray}\label{volumepolsines3}
N! Vol(A_{0}A_{1}A_{2}\ldots A_{N-1}A_{N+1})=\prod_{i=1}^{N+1}\abs{\vec{a}_{i}} polsin_{N}(A_{0},A_{1}A_{2}\ldots A_{N+1})
\end{eqnarray}

By substituting (\ref{volumepolsines2}), (\ref{volumepolsines3}) in (\ref{volumepolsines1}), we obtain:
\begin{eqnarray}\label{vol4imp}
\frac{B_{N+1}}{a_{N+1}Vol(A_{0}A_{1}...A_{N-1}A_{N})}=\frac{B_{N}}{a_{N}Vol(A_{0}A_{1}A_{2}...A_{N-1}A_{N+1})}.
\end{eqnarray}

By exchanging the indices cyclically $i\to j$ for $i,j=1,2,\ldots,N,N+1$ and taking into account (\ref{flotcond1extern0ijkn}), we get (\ref{ratiovolumesn}).

\end{proof}

For $N=3,$ the following corollary is derived for tetrahedra in $\mathbb{R}^{3}:$

\begin{corollary}{Weighted volume equalities in $\mathbb{R}^{N}$ ,\cite{Zach/Zou:09}}\label{ratiovolumes3}
\begin{eqnarray}\label{ratiovolumesn3}
\frac{B_{1}}{a_{1}Vol(A_{0}A_{2}A_{3}A_{4})}=\frac{B_{2}}{a_{2}Vol(A_{1}A_{0}A_{3}A_{4})}=\\\nonumber=
\frac{B_{3}}{a_{3}Vol(A_{0}A_{1}A_{2}A_{4})}=\frac{B_{4}}{a_{4}Vol(A_{1}A_{2}A_{3}A_{0})}=\\\nonumber
=\frac{\sum_{i=1}^{4}\frac{B_{i}}{a_{i}}}{Vol(A_{1}A_{2}A_{3}A_{4})}.
\end{eqnarray}

\end{corollary}

\begin{definition}\label{plasticitympolytopes}
 We call \textit{dynamic plasticity} of a variable weighted Fermat tree (weighted network) whose endpoints correspond to a closed polytope in $\mathbb{R}^{N},$ which is formed by $(N+2)$ weighted
line segments meeting at the weighted Fermat point $A_{0},$ the
set of solutions of the $(n+2)$ variable weights with respect to the
$(N+1)-$INVWF problem in $\mathbb{R}^{N},$ for a given
constant value $c,$ which correspond to a family of variable weighted networks that preserve the weighted Fermat point and the boundary of the closed polytope, such that the $(N+1)$ variable
weights depend on a variable weight and the value of $c.$

\end{definition}
We denote by $(B_{i})_{12\ldots N+2}$ the weight which corresponds to the vertex that
lies on the ray $A_{0}A_{i},$ for $i=1,2,\ldots,N+2$ and the weight
$(B_{j})_{i_{1}i_{2}\ldots i_{N+1}}$ corresponds to the vertex $A_{j}$ that lies on
the ray $A_{0}A_{j}$ with respect to the $N-$ simplex
$A_{i_{1}}A_{i_{2}}\ldots A_{i_{N+1}},$ for $i_{1},i_{2},\ldots,i_{N+1}\in \{1,2,\ldots N+2\}$ and $i_{1}\ne i_{2}\ne\ldots\ne i_{N+1}.$
\begin{theorem}\label{theorplastmpol}
The following equations point out the dynamic plasticity of a weighted Fermat tree for $(N+1)$ weighted boundary
closed polytopes with respect to the non-negative variable weights
$(B_{i})_{12\ldots N+2}$ in $\mathbb{R}^{N}$:

\begin{eqnarray}\label{dynamicplasticity2}
(\frac{B_{1}}{B_{N+1}})_{12\ldots N+2}=(\frac{B_{1}}{B_{N+1}})_{12\ldots N+1}(1-(\frac{B_{N+2}}{B_{N+1}})_{12\ldots (N+2)}(\frac{B_{N+1}}{B_{N+2}})_{2\ldots N+2})\nonumber\\
\end{eqnarray}
\begin{eqnarray}\label{dynamicplasticity3}
(\frac{B_{2}}{B_{N+1}})_{12\ldots N+2}=(\frac{B_{2}}{B_{N+1}})_{12\ldots N+1}(1-(\frac{B_{N+2}}{B_{N+1}})_{12\ldots N+2}(\frac{B_{N+1}}{B_{N+2}})_{13\ldots N+2})\nonumber\\
\end{eqnarray}
$\vdots$
\begin{eqnarray}\label{dynamicplasticity1}
(\frac{B_{N}}{B_{N+1}})_{12\ldots N+2}=(\frac{B_{N}}{B_{N+1}})_{12\ldots N+1}(1-(\frac{B_{N+2}}{B_{N+1}})_{12\ldots N+2}(\frac{B_{N+1}}{B_{N+2}})_{12\ldots (N-1)(N+1)(N+2)}).\nonumber\\
\end{eqnarray}

\end{theorem}

\begin{proof}
We consider the weighted floating case for $A_{0}\notin \{A_{1}A_{2}\ldots A_{N+2}\}:$
\begin{eqnarray}\label{floatcondnplus2}
\sum_{i=1}^{N+2}(B_{i})_{12\ldots N+2}\vec{u}(A_{0},A_{i})=\vec{0}.
\end{eqnarray}
The normals $\vec{n}_{0,23\ldots N},$
$\vec{n}_{0,13\ldots N},$ and $\vec{n}_{0,12\ldots(N-1)}$ are orthogonal to the subspaces $A_{2}A_{3}\ldots A_{N},$
$A_{1}A_{3}\ldots A_{N},$ $A_{1}A_{2}\ldots A_{N},$ respectively.

We set
\[sgn_{i,203\ldots N}=\begin{cases} +1,& \text{if $A_{i}$
is upper from the subspace  $A_{2}A_{0}\ldots A_{N}$ },\\
0,& \text{if $A_{i}$ belongs to the subspace $A_{2}A_{0}\ldots A_{N}$},\\
-1, & \text{if $A_{i}$ is under the subspace $A_{2}A_{0}\ldots A_{N}$ } ,
\end{cases},\]
for $i=1, N+1, N+2.$

The inner product of (\ref{floatcondnplus2}) with $\vec{n}_{0,23\ldots N},$
$\vec{n}_{0,13\ldots N},$ and $\vec{n}_{0,12\ldots (N-1)}$ yield

\begin{eqnarray}\label{eq1weights}
(B_{1})_{12\ldots N+2}sgn_{1,203\ldots N}\sin a_{1,203\ldots N}+\nonumber\\
+(B_{N+1})_{12\ldots N+2}sgn_{N+1,203\ldots N}\sin a_{N+1,203\ldots N}+\nonumber\\
+(B_{N+2})_{12\ldots N+2}sgn_{N+2,203\ldots n}\sin a_{N+2,203\ldots N}=0,\nonumber\\
\end{eqnarray}

\begin{eqnarray}\label{eq2weights}
(B_{2})_{12\ldots N+2}sgn_{2,103\ldots N}\sin a_{2,103\ldots N}+\nonumber\\
+(B_{N+1})_{12\ldots N+2}sgn_{N+1,103\ldots N}\sin a_{N+1,103\ldots N}+\nonumber\\
+(B_{N+2})_{12\ldots N+2}sgn_{N+2,103\ldots N}\sin a_{N+2,103\ldots N}=0,\nonumber\\
\end{eqnarray}
$\vdots$
\begin{eqnarray}\label{eqnweights}
(B_{N})_{12\ldots N+2}sgn_{N,102\ldots N-1}\sin a_{N,102\ldots N-1}+\nonumber\\
+(B_{N+1})_{12\ldots N+2}sgn_{N+1,102\ldots N-1}\sin a_{N+1,102\ldots N-1}+\nonumber\\
+(B_{N+2})_{12\ldots N+2}sgn_{N+2,102\ldots N-1}\sin a_{N+2,102\ldots N-1}=0.\nonumber\\
\end{eqnarray}

By substituting $B_{N+2}=0$ in (\ref{eq1weights}), (\ref{eq2weights}), (\ref{eqnweights}),
we derive:
\begin{eqnarray}\label{ratio1n+1}
(\frac{B_{1}}{B_{N+1}})_{12\ldots N+1}=-\frac{sgn_{N+1,203\ldots N}\sin a_{N+1,203\ldots N}}{sgn_{1,203\ldots N}\sin a_{1,203\ldots N}},
\end{eqnarray}
\begin{eqnarray}\label{ratio2n+1}
(\frac{B_{2}}{B_{N+1}})_{12\ldots N+1}=-\frac{sgn_{N+1,103\ldots N}\sin a_{N+1,103\ldots N}}{sgn_{2,103\ldots N}\sin a_{2,103\ldots N}},
\end{eqnarray}
\begin{eqnarray}\label{rationn+1}
(\frac{B_{N}}{B_{N+1}})_{12\ldots N+1}=-\frac{sgn_{N+1,102\ldots N-1}\sin a_{N+1,102\ldots N-1}}{sgn_{N,102\ldots N-1}\sin a_{N,102\ldots N-1}}.
\end{eqnarray}


By substituting $(B_{1})_{12\ldots N+2}=0, $ $(B_{2})_{12\ldots N+2}=0,$ $(B_{N})_{12\ldots N+2}=0,$ in (\ref{eq1weights}), (\ref{eq2weights}), (\ref{eqnweights}), respectively, we derive:

\begin{eqnarray}\label{n+2n+110}
(\frac{B_{N+2}}{B_{N+1}})_{2\ldots N+2}=-\frac{sgn_{N+1,203\ldots N}\sin a_{N+1,203\ldots N}}{sgn_{N+2,203\ldots N}\sin a_{N+2,203\ldots N}},
\end{eqnarray}

\begin{eqnarray}\label{n+2n+120}
(\frac{B_{N+2}}{B_{N+1}})_{13\ldots N+2}=-\frac{sgn_{N+1,103\ldots N}\sin a_{N+1,103\ldots N}}{sgn_{N+2,103\ldots N}\sin a_{n+2,103\ldots N}},
\end{eqnarray}

\begin{eqnarray}\label{n+2n+1n0}
(\frac{B_{N+2}}{B_{N+1}})_{12\ldots (N-1)(N+1)(N+2)}=-\frac{sgn_{N+1,102\ldots N-1}\sin a_{N+1,102\ldots N-1}}{sgn_{N+2,102\ldots N-1}\sin a_{N+2,102\ldots N-1}}.
\end{eqnarray}

By substituting (\ref{ratio1n+1}), (\ref{ratio2n+1}), (\ref{rationn+1}), (\ref{n+2n+110}), (\ref{n+2n+120}), (\ref{n+2n+1n0}) in (\ref{eq1weights}), (\ref{eq2weights}), (\ref{eqnweights}), we obtain (\ref{dynamicplasticity2}),
(\ref{dynamicplasticity3}) and (\ref{dynamicplasticity1}).

\end{proof}

A direct consequence of Theorem~\ref{theorplastmpol} is the following corollary, by setting\\ $\sum_{12\ldots N+2}B\equiv \sum_{i=1}^{N+2} (B_{i})_{12\ldots N+2},$\\
$\sum_{i_{1}i_{2}\ldots i_{N+1}}B=(B_{i_{1}})_{i_{1}i_{2}\ldots i_{N+1}}+\ldots +(B_{i_{N+1}})_{i_{1}i_{2}\ldots i_{N+1}},$\\ for $i_{1},\ldots i_{N+1}.
\in\{1,2,\ldots N+2\}.$

\begin{corollary}\label{dependenceweights}
If $\,\sum_{12\ldots N+2}B=\sum_{i_{1}i_{2}\ldots i_{N+1}}B,$ for every $i_{1},\ldots i_{N+1} \in\{1,2,\ldots N+2\},$ where
$\sum_{12\ldots N+2}B:=(B_{N+1})_{12\ldots N+2}(1+\sum_{i=1, i\neq N+1}^{N+2}(\frac{B_i}{B_{N+1}})_{1,2,\ldots N+2}),$
then
\begin{eqnarray}\label{B12nplus2}
 (B_{i})_{12\ldots N+2}=a_i (B_{N+2})_{12\ldots N+2}+ b_i,\quad i=1,2\ldots, N+1,
\end{eqnarray}
where

\begin{eqnarray}\label{abnplusone}
\nonumber
&&(a_{N+1},\,b_{N+1})=\nonumber\\&&{} (\frac{(\frac{B_{1}}{B_{N+1}})_{12\ldots N+1}(\frac{B_{N+1}}{B_{N+2}})_{2\ldots N+2}+\dots +
(\frac{B_{N}}{B_{N+1}})_{12\ldots N+1}(\frac{B_{N+1}}{B_{N+2}})_{12\ldots (N-1)(N+1)N+2}  -1}
{\sum_{i=1}^{N+1}(\frac{B_{i}}{B_{N+1}})_{12\ldots N+1}  },\ \nonumber\\&&{}(B_{N+1})_{12\ldots N+1}),
\nonumber
\end{eqnarray}

\begin{eqnarray}\label{abn}
\nonumber
&&(a_{N},\,b_{N})=(a_{N+1}(\frac{B_{N}}{B_{N+1}})_{12\ldots N+1}-(\frac{B_{N}}{B_{N+1}})_{12\ldots N+1}(\frac{B_{N+1}}{B_{N+2}})_{12\ldots (N-1)(N+1)(N+2)},\ \nonumber\\&&{}(B_{N})_{12\ldots N+1}),
\nonumber
\end{eqnarray}
$\vdots$
\begin{eqnarray}\label{ab1}
\nonumber
&&(a_{1},\,b_{1})=(a_{N+1}(\frac{B_{1}}{B_{N+1}})_{12\ldots N+1}-(\frac{B_{1}}{B_{N+1}})_{12\ldots N+1}(\frac{B_{N+1}}{B_{N+2}})_{23\ldots(N+2)},\ \nonumber\\&&{}(B_{1})_{12\ldots N+1}).
\nonumber
\end{eqnarray}

\end{corollary}

Suppose that $A_{0}$ is an interior weighted Fermat point for the $N-$simplex
with respect to the non-negative given weights $\{B_{1}(0),B_{2}(0),\ldots,B_{N}(0),B_{N+1}(0)\}$ in $\mathbb{R}^{N}.$
Therefore, the topology of the branches $A_{i}A_{0}$ which meet at $A_{0}$ form a unique floating weighted Fermat tree.
The unique solution of the $N-$INVWF problem for $A_{1}A_{2}\ldots A_{N+1}$ is responsible for the cancellation of the dynamic plasticity of simplexes. We assume that after time $t$ an $(N+2)$ branch $A_{0}A_{N+2}$ starts to grow from $A_{0}$ and the new branch $A_{0}A_{N+2}$ is located inside the cone $C(ray(A_{0}A_{1}),ray(A_{0}A_{3})\ldots, ray(A_{0}A_{N+1}))$ and $A_{0}A_{2}$ is located outside\\ $C(ray(A_{0}A_{1}),ray(A_{0}A_{3})\ldots, ray(A_{0}A_{N+1})).$ We assume that $A_{N+1}$ is upper from the hyperplane formed by the first, second,$\ldots$ and the $(N-1)$th ray and $A_{N},$ $A_{N+2}$ are under this hyperplane.
\begin{theorem}\label{plasticityn2polytopes}
An increase to the weight that corresponds to the $(N+2)th$ ray causes a decrease to the weight that corresponds to the $(N+1)th$ ray and a variation to the weight that corresponds to the $ith$ ray depends on the difference $(B_{i})_{123\ldots N(N+1)}-(B_{i})_{123\ldots N(N+2)},$ for $i=1,2,\ldots N.$
\end{theorem}

\begin{proof}
First, we will show that $a_{N+1}<0.$ Taking into account (\ref{abnplusone}), we obtain that\\ $a_{N+1}=-\frac{(B_{N+1})_{123\ldots N+1}}{(B_{N+2})_{123\ldots N N+2}}<0,$\\ because $(B_{N+1})_{123\ldots N+1}, (B_{N+2})_{123\ldots N N+2}>0.$ Thus, we get:
\begin{eqnarray}\label{eqnplus2nplus1}
(B_{N+1})_{123\ldots N+2}=-\frac{(B_{N+1})_{123\ldots N+1}}{(B_{N+2})_{123\ldots N N+2}}(B_{N+2})_{123\ldots N+2}+(B_{N+1})_{123\ldots N+1}.
\end{eqnarray}

By inserting (\ref{eqnplus2nplus1}) into (\ref{eq1weights}), (\ref{eq2weights}),$\dots$ (\ref{eqnweights}),
we derive that:
\begin{eqnarray}
(B_{i})_{123\ldots N+2}=\frac{(B_{i})_{123\ldots N(N+2)}-(B_{i})_{123\ldots (N+1)}}{(B_{N+2})_{123\ldots N N+2}}(B_{N+2})_{123\ldots N+2}+ \nonumber\\ +(B_{i})_{123\ldots N+1},
\end{eqnarray}
for $i=1,2,\ldots N.$
Therefore, the difference $\delta B_{i}$ of the two positive weights $(B_{i})_{123\ldots N(N+2)}-(B_{i})_{123\ldots (N+1)},$ which corresponds to the $ith$ branch $A_{0}A_{i}$ of the boundary simplexes $A_{1}A_{2}\ldots A_{N}A_{N+2},$ $A_{1}A_{2}\ldots A_{N}A_{N+1},$
yields the following two results:

(a) If $\delta B_{i}>0,$ the weight $(B_{i})_{123\ldots N+2}$ is increased,

(b) If $\delta B_{i}<0,$ the weight $(B_{i})_{123\ldots N+2}$ is decreased.

\end{proof}

\begin{proposition}\label{plasticityn3polytopes}
An increase to the weight that corresponds to the $(N+2)th$ ray and a decrease to the weight that corresponds to the $ith$ ray for $i=1,3,\ldots,N,$ causes a decrease to the weight that corresponds to the $(N+1)th$ ray and an increase to the weights that corresponds to the second ray.
\end{proposition}

\begin{proof}
By applying Theorem~\ref{plasticityn2polytopes} and taking into account that\\ $a_{1}, a_{3},\ldots a_{N}<0,$
the length of the vectors $\vec{X}_{i}=(B_{i})_{12\ldots n+2}(\delta t)\vec{u}(A_{0},A_{i})$ are decreased after time $\delta t.$ We proceed by arranging some vector terms of the weighted floating balancing condition for $\vec{X}_{i}$ (Theorem~\ref{theor1}):
\[\vec{X}_{2}+\vec{X}_{n+1}=-((\sum_{i=1, i\neq 2}^{N}\vec{X}_{i})+ \vec{X}_{N+2}).\]
We observe that after time $\delta t$ the length of the vector $(\sum_{i=1, i\neq 2}^{N}\vec{X}_{i}$ is decreased.
Thus, on the right hand side of the above equation a vector of reduced length is composed with a vector of increased length $\vec{X}_{N+2}$ and on the left hand side a vector of reduced length $\vec{X}_{N+1}$ is composed with $\vec{X}_{2}$ whose length is increased. Therefore, we obtain that $a_{2}>0.$

\end{proof}
Suppose that at time $t=0,$ an evolutionary multitree occurs, such that $(N+1)! \dot k$ weighted Fermat trees correspond to $k$ simplexes, which form a Frechet $N-$multisimplex in $\mathbb{R}^{N},$ with $(N+1)!$ permutation of the weights $\{B_{1}(0),B_{2}(0),\ldots,B_{N}(0),B_{N+1}(0)\}$ in $\mathbb{R}^{N}$ for boundary incongruent $N-$simplexes $(A_{1})_{k}(A_{2})_{k}\ldots (A_{N+1})_{k},$ for $k \le \frac{\frac{1}{2}N(N+1)!}{(N+1)!}.$
After time $t,$ a $(N+2)th$ branch $(A_{0})_{k}(A_{N+2})_{k}$ starts to grow from the weighted Fermat point $(A_{0})_{k}$ and the new branch $(A_{0})_{k}(A_{N+2})_{k}$ is located inside the cone $C(ray((A_{0})_{k}(A_{1})_{k}),ray((A_{0})_{k}(A_{3})_{k})\ldots, ray((A_{0})_{k}(A_{N+1})_{k}))$ and $(A_{0})_{k}(A_{2})_{k}$ is located outside $C(ray(((A_{0}){k}A_{1})_{k}),ray((A_{0})_{k}(A_{3})_{k})\ldots, ray((A_{0})_{k}(A_{N+1})_{k}))$ and let $(A_{N+1})_{k}$ be upper from the hyperplane formed by the first, second,$\ldots$ and the $(N-1)$th ray and $(A_{N})_{k},$ $(A_{N+2)_{k}}$ are under this hyperplane.

\begin{theorem}\label{plasticitymultitreern}
An increase to the weight that corresponds to the $(N+2)th$ ray causes a decrease to the weight that corresponds to the $(N+1)th$ ray and a variation to the weight that corresponds to the $ith$ ray depends on the difference $(B_{i})_{123\ldots N(N+1)}-(B_{i})_{123\ldots N(N+2)},$ such that the geometric structure of the weighted Fermat-Frechet multitree with respect to a boundary $N-$multisimplex, remains the same, for $i=1,2,\ldots N,$ $1\le k \le \frac{\frac{1}{2}N(N+1)!}{(N+1)!}.$
\end{theorem}

\begin{proof}
By applying Theorem~\ref{plasticityn2polytopes} starting from a weighted Fermat-Frechet multitree with respect to a boundary $N-$multisimplex in $\mathbb{R}^{N},$ we obtain the plasticity of a weighted Fermat-Frechet-multitree by adding the ray $(A_{0})_{k}(A_{N+2})_{k},$ for $i=1,2,\ldots N,$ $1\le k \le \frac{\frac{1}{2}N(N+1)!}{(N+1)!}.$
\end{proof}

------------------------------
\section{The weighted Fermat-Steiner-Frechet multitree for a given tentuple of positive real numbers determining the edge lengths of incongruent $4-$simplexes in $\mathbb{R}^{4}$}
In this section, we deal with the solution (multitree) of the weighted Fermat-Steiner-Frechet problem (P(Fermat-Steiner-Frechet)) for a given tentuple of positive real numbers determining incongruent $4-$simplexes in $\mathbb{R}^{4},$ by inserting three equality constraints derived by three independent solutions for three variable weighted Fermat problems for the Frechet $4-$multisimplex derived by incongruent boundary $4-$simplexes in $\mathbb{R}^{4},$ which correspond to the same tentuple of positive real numbers (edge lengths) and two equality constraints derived by two different expressions of the line segments connecting the three weighted Fermat-Steiner points. The detection of the weighted Fermat-Steiner Frechet multitrees is achieved by applying the Lagrange multiplier rule. By applying a Lagrange program to detect unweighted Fermat-Frechet multitree for a given tentuple of edge lengths determining 30.240 incongruent $4-$ simplexes (Frechet $4-$multisimplex) in $\mathbb{R}^{4},$ by using Dekster Wilker tenttuples, we derive an interesting result of seeking unweighted Fermat-Frechet multitrees with three equally weighted Fermat-Steiner points inside the Frechet $4-$multisimplex. This result may provide an approach to detect the most natural of six consecutive natural numbers (Herzog sextuples) for $N\ge 30$ and it is achieved by seeking an upper bound for these three equal weights, which yield a global weighted Fermat-Steiner tree of minimum length for the boundary tetrahedron having the maximum volume among the 30.240 incongruent $4-$simplexes in $\mathbb{R}^{4}.$

Let $\{A_{1},A_{2},A_{3},A_{4},A_{5}\}$ be the vertices of a $4-$simplex in $\mathbb{R}^{4}$ and $A_{0}$ be a point inside $A_{1}A_{2}A_{3}A_{4}A_{5}.$

We denote by $A_{0,1234}$ the orthogonal projection of $A_{0}$ to the hyperplane defined by the tetrahedron $A_{1}A_{2}A_{3}A_{4},$ by $A_{0,123}$ the orthogonal projection of $A_{0}$ to the plane defined by $\triangle A_{1}A_{2}A_{3},$ by $a_{ij}$ the length of the line segment $A_{i}A_{j},$ for $i,j=1,2,3,4,5$ by $h_{0,123j}$ the length of $A_{0}A_{0,123j},$ for $j=4,5,$ by $h_{0,123}$ the length of $A_{0}A_{0,123},$ by $h_{0,12}$ the length of $A_{0}A_{0,12}.$

We set $\beta\equiv \angle A_{0}A_{0,123}A_{0,1234},$ $\alpha\equiv \angle A_{0}A_{0,12}A_{0,123}$ $x_{i}\equiv A_{0,1234}A_{i},$ $y_{i}\equiv A_{0,1235}A_{i},$ for $i=1,2,3,4,5$ $h_{(0,123k),(0,12)}\equiv A_{0,123k}A_{0,12}$ $z_{j}\equiv A_{0,123}A_{j}$ for $k=4,5$ $j=1,2,3.$

\begin{theorem}[Generalized cosine law in $\mathbb{R}^{4}$] \label{calculationa0405a01a02a03beta}

The line segment $a_{40},$ $a_{50}$ depend on $a_{10}, a_{20}, a_{30}$ and $\beta$ in $\mathbb{R}^{4}:$

\begin{equation}\label{a04a01a02a03beta}
a_{40}^{2}=h_{0,1234}^{2}(a_{10},a_{20},a_{30},\beta)+ x_{4}^{2}(a_{10},a_{20},a_{30},\beta)
\end{equation}

\begin{equation}\label{a05a01a02a03beta}
a_{50}^{2}=h_{0,1235}^{2}(a_{10},a_{20},a_{30},\beta)+ y_{5}^{2}(a_{10},a_{20},a_{30},\beta)
\end{equation}


\end{theorem}

\begin{figure}\label{figg2}
\centering
\includegraphics[scale=0.80]{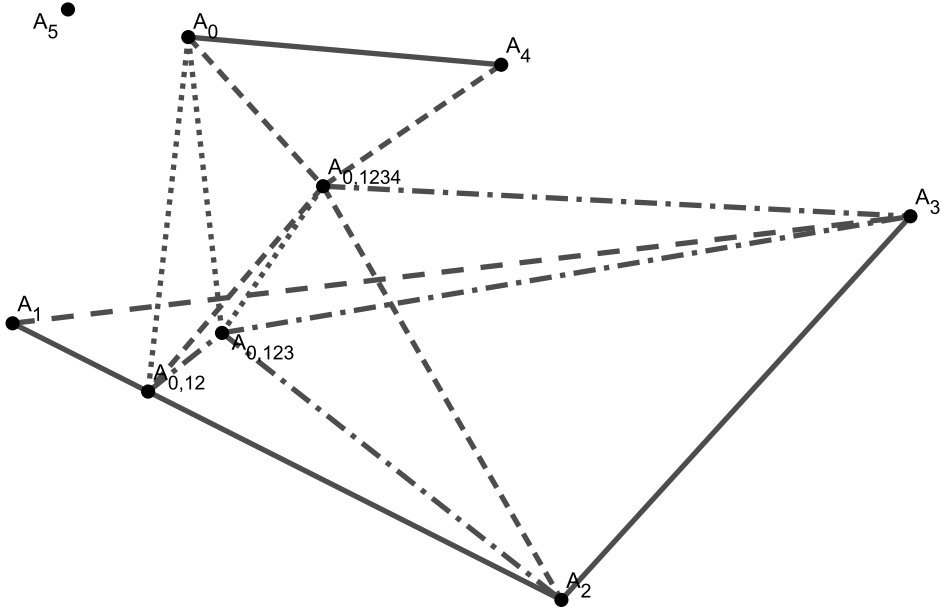}
\caption{Distance characterization of a weighted Fermat-Steiner-Frechet multitree in $\mathbb{R}^{4}$} \label{figg2}
\end{figure}

\begin{proof}
From $\triangle A_{0}A_{4}A_{0,1234},$ $\triangle A_{0}A_{2}A_{0,1234},$ we get, respectively:
\begin{equation}\label{a4r41}
a_{40}^{2}=h_{0,1234}^2+x_{4}^{2}.
\end{equation}

\begin{equation}\label{a4r42}
a_{20}^{2}=h_{0,1234}^2+x_{2}^{2}.
\end{equation}

From $\triangle A_{0}A_{0,1234}A_{0123},$ we obtain a relation for Schlafli's angle $\beta :$

\begin{equation}\label{a4r43}
h_{0,1234}=h_{0,123}\sin\beta
\end{equation}

where
\begin{equation}\label{a4r44}
h_{0,123}=h_{0,12}\sin\alpha
\end{equation}

By substituting $A_{0}\to A_{0,1234}$ and the notations $x_{i}\equiv A_{0,1234}A_{i}$ in (\ref{eq:deral2}) from Lemma~\ref{calculationa0304a01a03alpha} and taking into account Lemma~\ref{a04a01a02a03r3}, we derive that:

\begin{equation}\label{x4x2beta}
x_{4}^{2}=x_{2}^{2}+a_{24}^{2}-2a_{24}[ \sqrt{x_{2}^{2}-h_{(0,1234),(0,12)}^{2}}\cos\alpha_{124}+h_{(0,1234),(0,12)}\sin\alpha_{124}\cos (\alpha_{g_{4}}-\alpha^{\prime}).
\end{equation}

From $\triangle A_{0}A_{0,12}A_{0,1234},$ $\triangle A_{0}A_{0,12}A_{2},$ we get:

\begin{equation}\label{height01234012}
\sqrt{x_{2}^{2}-h_{(0,1234),(0,12)}^{2}}=\sqrt{a_{20}^{2}-h_{0,12}^{2}},
\end{equation}

where

\[h_{0,12}=h_{0,12}(a_{01},a_{02};a_{12})=\frac{a_{01}a_{02}}{a_{12}}\sqrt{1-\left(\frac{a_{01}^{2}+a_{02}^{2}-a_{12}^2}{2a_{01}a_{02}}
\right)^{2}}.\]

Hence, (\ref{height01234012}) depends on $a_{10},$ $a_{20}.$

From $\triangle A_{0}A_{0,12}A_{0,1234},$ we get:

\begin{equation}\label{height01234012n}
h_{(0,1234),(0,12)}=\sqrt{h_{0,12}^{2}-h_{0,123}^{2}}.
\end{equation}

By substituting (\ref{height01234012}), (\ref{x4x2beta}), (\ref{height01234012n}) in (\ref{a4r41}), we obtain:

\begin{align}\label{a4r455}
a_{40}^{2}=a_{20}^{2}+a_{24}^{2}-2a_{24}[\sqrt{a_{20}^{2}-h_{0,12}^{2}}\cos\alpha_{124}+\nonumber\\ +\sqrt{h_{0,12}^{2}-h_{0,123}^{2}}\sin \alpha_{124}\cos(\alpha_{g_{4}}-\alpha^{\prime})]
\end{align}

By substituting $A_{0}\to A_{0,1234}$ and the notations $x_{i}\equiv A_{0,1234}A_{i}$ in (\ref{eq:deral2}) for $i=3$ from Lemma~\ref{calculationa0304a01a03alpha}, we derive that:

\[\alpha^{\prime}=\arccos\left(
\frac{\left(\frac{x_{2}^2+a_{23}^2-x_{3}^2}{2 a_{23}}
\right)-\sqrt{x_{2}^2-h_{(0,1234),(0,12)}^2}\cos\alpha_{123}}{h_{(0,1234),(0,12)}\sin\alpha_{123}}
\right),\]

or

\[\alpha^{\prime}=\arccos\left(
\frac{\left(\frac{x_{2}^2+a_{23}^2-x_{3}^2}{2 a_{23}}
\right)-\sqrt{a_{20}^2-h_{0,12}^2}\cos\alpha_{123}}{\sqrt{h_{0,12}^{2}-h_{0,123}^{2}}\sin\alpha_{123}}
\right).\]

By solving (\ref{eq:deral2}) taken from Lemma~\ref{calculationa0304a01a03alpha} with respect to $\alpha,$ we get:

\begin{equation}\label{alphanew}
\alpha=\arccos\left(
\frac{\left(\frac{a_{02}^2+a_{23}^2-a_{03}^2}{2 a_{23}}
\right)-\sqrt{a_{02}^2-h_{0,12}^2}\cos\alpha_{123}}{h_{0,12}\sin\alpha_{123}}
\right).
\end{equation}


By replacing (\ref{a4r44}), (\ref{alphanew}) in (\ref{a4r43}), we have:

\begin{equation}\label{h0123a01a02a03}
h_{0,123}= \frac{a_{01}a_{02}}{a_{12}}\sqrt{1-\left(\frac{a_{01}^{2}+a_{02}^{2}-a_{12}^2}{2a_{01}a_{02}}
\right)^{2}}\nonumber\\ \sin (\arccos\left(
\frac{\left(\frac{a_{02}^2+a_{23}^2-a_{03}^2}{2 a_{23}}
\right)-\sqrt{a_{02}^2-h_{0,12}^2}\cos\alpha_{123}}{h_{0,12}\sin\alpha_{123}}
\right)),
\end{equation}

\begin{align}\label{h01234}
h_{0,1234}= \frac{a_{01}a_{02}}{a_{12}}\sqrt{1-\left(\frac{a_{01}^{2}+a_{02}^{2}-a_{12}^2}{2a_{01}a_{02}}
\right)^{2}}\nonumber\\ \sin (\arccos\left(
\frac{\left(\frac{a_{02}^2+a_{23}^2-a_{03}^2}{2 a_{23}}
\right)-\sqrt{a_{02}^2-h_{0,12}^2}\cos\alpha_{123}}{h_{0,12}\sin\alpha_{123}}
\right))\sin\beta.
\end{align}

Thus, (\ref{h0123a01a02a03}), (\ref{h01234}) yield:

\[h_{0,123}=h_{0,123}(a_{10},a_{20},a_{30}),\]

\[h_{0,1234}=h_{0,1234}(a_{10},a_{20},a_{30},\beta)\]
and $x_{i}=\sqrt{a_{i0}^2 - h_{0,1234}((a_{10},a_{20},a_{30},\beta))^2}=x_{i}(a_{10},a_{20},a_{30},\beta),$ for $i=2,3.$

Therefore, we derive from (\ref{a04a01a02a03beta}) that $a_{40}$ depends on $a_{10},$ $a_{20},$ $a_{30}$ and $\beta.$

By following a similar process for $i=5,$ we derive from (\ref{a05a01a02a03beta}) that $a_{50}$ depend on $a_{10},$ $a_{20},$ $a_{30}$ and $\beta.$

\end{proof}

Let $A_{0,1},$ $A_{0,2},$ $A_{0,3}$ three points inside the $4-$simplex $A_{1}A_{2}A_{3}A_{4}A_{5}$ in $\mathbb{R}^{4}.$
We denote by $a_{(0,i),j}$ the length of the line segment $A_{0,i}A_{j},$ by $A_{(0,i),jkl},$ $A_{(0,i),jklm}$ the orthogonal projections of $A_{0,i},$ with respect to the plane defined by $\triangle A_{j}A_{k}A_{k}$ and the hyperplane defined by $A_{j}A_{k}A_{l}A_{m},$ respectively,  for $i=1,2,3,$ $j,k,l,m=1,2,3,4,5,$ and we set  $\beta_{1}\equiv \angle A_{0,1}A_{(0,1),123}A_{(0,1),1234},$ $\beta_{2}\equiv \angle A_{0,2}A_{(0,2),234}A_{(0,2),2345},$ $\beta_{3}\equiv \angle A_{0,3}A_{(0,3),234}A_{(0,1),2345},$ where $\beta_{i}$ are the Schafli angles for $i=1,2,3.$
The weighted Fermat-Steiner-Frechet problem for a given tentuple of edge lengths determining incongruent $4-$simplexes in $\mathbb{R}^{4},$ states that:

\begin{problem}[The weighted Fermat-Steiner Frechet in $\mathbb{R}^{4}$]\label{FermatSteinerFrechettetrahedronr3}
Given an octuple of weights $\{b_{1},b_{2},b_{3},b_{4},b_{5},b_{ST},b_{ST},b_{ST}\},$ and a given tentuple of positive real numbers (edge lengths) $\{a_{ij}\}\},$ determining a Frechet $4-$simplex $F(A_{1}A_{2}A_{3}A_{4}A_{5}),$   find the position of $A_{0,1}$ and / or
$A_{0,2}$ and/or $A_{0,3}$  with given
weights $b_{ST}$ in $A_{0,1},$  $b_{ST}$ in $A_{0,2},$ $b_{ST}$ in $A_{0,3}$ such that
\begin{align}\label{equat1L0r4}
f_{0}(a_{(0,1),1},a_{(0,1),2},a_{(0,1),(0,2)},a_{(0,2),3},a_{(0,2),(0,3)},a_{(0,3),4},a_{(0,3),5})=\nonumber\\
b_{1}a_{(0,1),1}+b_{2}a_{(0,1),2}+b_{3}a_{(0,2),3}+b_{4}a_{(0,3),4}+b_{5}a_{(0,3),5}+\nonumber\\
+b_{ST}(a_{(0,1),(0,2)}+a_{(0,2),(0,3)})\to min.
\end{align}
\end{problem}
\begin{definition}[A non degenerate weighted Fermat-Steiner tree for $\{A_{1}A_{2}A_{3}A_{4}A_{5}\}$]

A non degenerate weighted Fermat-Steiner tree consists of the line segments\\
 $\{A_{0,1}A_{1}, A_{0,1}A_{2},A_{0,1}A_{0,2},A_{0,2}A_{3},A_{0,2}A_{0,3}, A_{0,3}A_{4},A_{0,3}A_{5}\}$ with corresponding weights
 $\{b_{1},b_{2},b_{ST},b_{3},b_{ST},b_{4},b_{5},\}$ such that the weighted Fermat-Steiner points $A_{0,i}$ have degree (connections) three.

\end{definition}

\begin{figure}\label{figg3}
\centering
\includegraphics[scale=0.80]{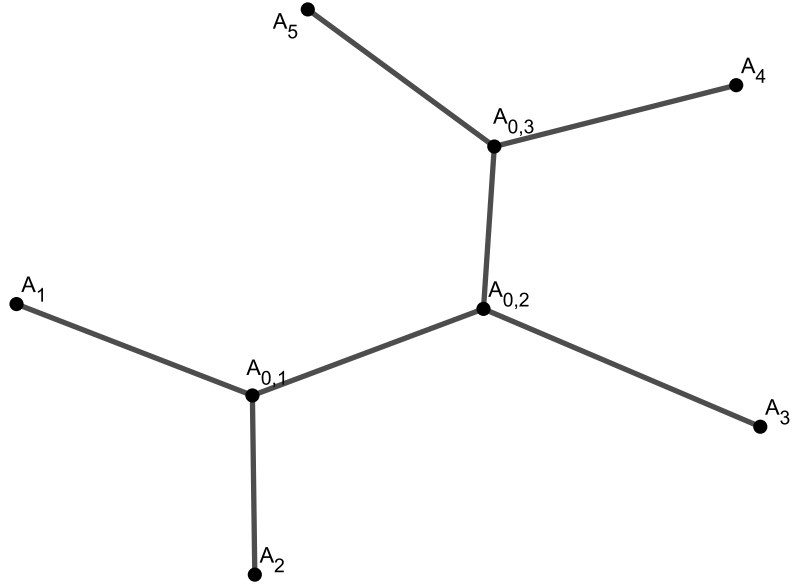}
\caption{A weighted Fermat-Steiner tree for $A_{1}A_{2}A_{3}A_{4}A_{5}$ in $\mathbb{R}^{4}$} \label{figg3}
\end{figure}

We will construct the Lagrangian function \[\mathcal{L}(\tilde{x},\tilde{\lambda})=\sum_{i=0}^{13}\lambda_{i}f_{i}(\tilde{x}),\]
where the point
\[\tilde{x}=\{x_{1},\ldots,x_{24}\}=\] \[=\{a_{(0,1),1},a_{(0,1),2},a_{(0,1),3},\beta_{1},
a_{(0,2),3},a_{(0,2),4},a_{(0,5),3},\beta_{2}, a_{(0,3),3},a_{(0,3),4},a_{(0,3),5},\beta_{3},\]\[ w_{(0,1),1},\ldots, w_{(0,1),4},w_{(0,2),2},\ldots, w_{(0,2),5},w_{(0,3),2},\ldots, w_{(0,3),5}\}\in \mathbb{R}^{24}\] is inside the parallelepiped $\Pi(p_{1},q_{1};\ldots ;p_{24},q_{24}),$
where $p_{i}<x_{i}<q_{i},$ for $i=1,2,\cdots, 29$ and the Lagrange multiplication vector is given by:

\[\tilde{\lambda}=\{\lambda_{0},\lambda_{1},\ldots,\lambda_{13}\}.\]

We extend the weighted Fermat-Steiner-Frechet problem (P(Fermat-Steiner-Frechet)) in $\mathbb{R}^{4},$ by inserting 12 equality constraints derived by three independent solutions for three new weighted Fermat problems for $A_{1}A_{2}A_{3}A_{4}A_{5}$ in $\mathbb{R}^{4},$ which give a connection with the initial weighted Fermat-Steiner objective function and one equality constraint derived by two different expressions of $a_{(0,1),(0,2)}.$

We note that the volume of an $(N-1)$-simplex $A_{1}A_{2}\cdots A_{N}$ in $\mathbb{R}^{N-1}$ is given by the Caley-Menger determinant in terms of edge lengths (\cite[(5.1), p.~125]{Sommerville:58}):
\begin{multline*}
   \operatorname{Vol}(A_{1}A_{2}\cdots A_{N})^{2} = \\ \frac{1}{(-1)^{N} 2^{N-1} ((N-1)!)^{2}} \begin{vmatrix}
  0 & 1 & 1 & \cdots & 1 & 1 \\
  1 & 0 & a_{12}^{2} & \cdots & a_{1(N-1)}^{2} & a_{1N}^{2} \\
  1 & a_{21}^{2} & 0 & \cdots & a_{2(N-1)}^{2} & a_{2N}^{2} \\
  . & . & . & . & . & . \\
  1 & a_{N1}^{2} & a_{N2}^{2} & \cdots & a_{N(N-1)}^{2} & 0
\end{vmatrix}.
\end{multline*}

\begin{problem}[The weighted Fermat-Steiner-Frechet (P(Fermat-Steiner-Frechet)) problem in $\mathbb{R}^{4}$ with equality constraints]
\begin{equation*}
\begin{aligned}
& & f_{0}(\tilde{x})\to min, \\
& & f_{i}(\tilde{x}) = 0, \; i = 1, \ldots,13
\end{aligned}
\end{equation*}

\begin{align}\label{fundzeror4}
f_{0}(\tilde{x})=b_{1}a_{(0,1),1}+b_{2}a_{(0,1),2}+b_{3}a_{(0,2),3}+b_{4}a_{(0,3),4}+\nonumber\\+b_{5}a_{(0,3),5}
+b_{ST}(a_{(0,1),(0,2)}(a_{(0,1),1},a_{(0,1),2},a_{(0,2),3},a_{(0,2),4},a_{(0,2),5},\beta_{2})+\nonumber\\
+a_{(0,2),(0,3)}(a_{(0,3),4},a_{(0,3),5},a_{(0,2),4}),a_{(0,2),5}),
\end{align}

\begin{align}\label{fundoner4}
f_{1}(\tilde{x})=\frac{w_{(0,1),1}}{a_{(0,1),1}\operatorname{Vol}(A_{0,1}A_{2}A_{3}A_{4}A_{5})}-\nonumber\\ -\frac{1-\sum_{i=1}^{4}w_{(0,1),i}}{a_{(0,1),5}(a_{(0,1),1},a_{(0,1),2},a_{(0,1),3},\beta_{1})\operatorname{Vol}(A_{0}A_{1}A_{2}A_{3}A_{4})},
\end{align}

\begin{align}\label{fundtwor4}
f_{2}(\tilde{x})=\frac{w_{(0,1),2}}{a_{(0,1),2}\operatorname{Vol}(A_{1}A_{0,1}A_{3}A_{4}A_{5})}-\nonumber\\ -\frac{1-\sum_{i=1}^{4}w_{(0,1),i}}{a_{(0,1),5}(a_{(0,1),1},a_{(0,1),2},a_{(0,1),3},\beta_{1})\operatorname{Vol}(A_{0,1}A_{1}A_{2}A_{3}A_{4})},
\end{align}

\begin{align}\label{fundthirdr4}
f_{3}(\tilde{x})=\frac{w_{(0,1),3}}{a_{(0,1),3}\operatorname{Vol}(A_{1}A_{2}A_{0,1}A_{4}A_{5})}-\nonumber\\ -\frac{1-\sum_{i=1}^{4}w_{(0,1),i}}{a_{(0,1),5}(a_{(0,1),1},a_{(0,1),2},a_{(0,1),3},\beta_{1})\operatorname{Vol}(A_{0,1}A_{1}A_{2}A_{3}A_{4})},
\end{align}

\begin{align}\label{fundfourthr4}
f_{4}(\tilde{x})=\frac{w_{(0,1),4}}{a_{(0,1),4}(a_{(0,1),1},a_{(0,1),2},a_{(0,1),3},\beta_{1})\operatorname{Vol}(A_{1}A_{2}A_{3}A_{0,1}A_{5})}-\nonumber\\ -\frac{1-\sum_{i=1}^{4}w_{(0,1),i}}{a_{(0,1),5}(a_{(0,1),1},a_{(0,1),2},a_{(0,1),3},\beta_{1})\operatorname{Vol}(A_{0,1}A_{1}A_{2}A_{3}A_{4})},
\end{align}


\begin{align}\label{fundfiver4}
f_{5}(\tilde{x})=\frac{w_{(0,2),5}}{a_{(0,2),5}\operatorname{Vol}(A_{1}A_{2}A_{3}A_{4}A_{0,2})}-\nonumber\\ -\frac{1-\sum_{i=2}^{5}w_{(0,2),i}}{a_{(0,2),1}(a_{(0,2),3},a_{(0,2),4},a_{(0,2),5},\beta_{2})\operatorname{Vol}(A_{0,2}A_{2}A_{3}A_{4}A_{5})},
\end{align}

\begin{align}\label{fundsixr4}
f_{6}(\tilde{x})=\frac{w_{(0,2),4}}{a_{(0,2),4}\operatorname{Vol}(A_{1}A_{2}A_{3}A_{0,2}A_{5})}-\nonumber\\ -\frac{1-\sum_{i=2}^{5}w_{(0,2),i}}{a_{(0,2),1}(a_{(0,2),3},a_{(0,2),4},a_{(0,2),5},\beta_{2})\operatorname{Vol}(A_{0,2}A_{2}A_{3}A_{4}A_{5})},
\end{align}

\begin{align}\label{fundsevenr4}
f_{7}(\tilde{x})=\frac{w_{(0,2),5}}{a_{(0,2),3}\operatorname{Vol}(A_{1}A_{2}A_{0,2}A_{4}A_{5})}-\nonumber\\ -\frac{1-\sum_{i=2}^{5}w_{(0,2),i}}{a_{(0,2),1}(a_{(0,2),3},a_{(0,2),4},a_{(0,2),5},\beta_{2})\operatorname{Vol}(A_{0,2}A_{2}A_{3}A_{4}A_{5})},
\end{align}

\begin{align}\label{fundeightr4}
f_{8}(\tilde{x})=\frac{w_{(0,2),5}}{a_{(0,2),2}(a_{(0,2),3},a_{(0,2),4},a_{(0,2),5},\beta_{2})\operatorname{Vol}(A_{1}A_{0,2}A_{3}A_{4}A_{0,2})}-\nonumber\\ -\frac{1-\sum_{i=2}^{5}w_{(0,2),i}}{a_{(0,2),1}(a_{(0,2),3},a_{(0,2),4},a_{(0,2),5},\beta_{2})\operatorname{Vol}(A_{0,2}A_{2}A_{3}A_{4}A_{5})},
\end{align}
\begin{align}\label{fund9r4}
f_{9}(\tilde{x})=\frac{w_{(0,3),5}}{a_{(0,3),5}\operatorname{Vol}(A_{1}A_{2}A_{3}A_{4}A_{0,3})}-\nonumber\\ -\frac{1-\sum_{i=2}^{5}w_{(0,3),i}}{a_{(0,3),1}(a_{(0,3),3},a_{(0,3),4},a_{(0,3),5},\beta_{3})\operatorname{Vol}(A_{0,3}A_{2}A_{3}A_{4}A_{5})},
\end{align}

\begin{align}\label{fund10r4}
f_{10}(\tilde{x})=\frac{w_{(0,3),4}}{a_{(0,3),4}\operatorname{Vol}(A_{1}A_{2}A_{3}A_{0,3}A_{5})}-\nonumber\\ -\frac{1-\sum_{i=2}^{5}w_{(0,3),i}}{a_{(0,3),1}(a_{(0,3),3},a_{(0,3),4},a_{(0,3),5},\beta_{3})\operatorname{Vol}(A_{0,3}A_{2}A_{3}A_{4}A_{5})},
\end{align}

\begin{align}\label{fund11r4}
f_{11}(\tilde{x})=\frac{w_{(0,3),5}}{a_{(0,3),3}\operatorname{Vol}(A_{1}A_{2}A_{0,2}A_{4}A_{5})}-\nonumber\\ -\frac{1-\sum_{i=2}^{5}w_{(0,3),i}}{a_{(0,3),1}(a_{(0,3),3},a_{(0,3),4},a_{(0,3),5},\beta_{3})\operatorname{Vol}(A_{0,3}A_{2}A_{3}A_{4}A_{5})},
\end{align}

\begin{align}\label{fund12r4}
f_{12}(\tilde{x})=\frac{w_{(0,3),5}}{a_{(0,3),2}(a_{(0,3),3},a_{(0,3),4},a_{(0,3,5},\beta_{3})\operatorname{Vol}(A_{1}A_{0,3}A_{3}A_{4}A_{0,2})}-\nonumber\\ -\frac{1-\sum_{i=2}^{5}w_{(0,3),i}}{a_{(0,3),1}(a_{(0,3),3},a_{(0,3),4},a_{(0,3),5},\beta_{3})\operatorname{Vol}(A_{0,3}A_{2}A_{3}A_{4}A_{5})},
\end{align}

\begin{align}\label{fund13r4}
f_{13}(\tilde{x})=a_{(0,1),(0,2)}(a_{(0,1),1},a_{(0,1),2},a_{(0,2),3},a_{(0,2),4},a_{(0,2),5},\beta_{2};b_{1};b_{2};b_{ST})-\nonumber\\-
a_{(0,1),(0,2)}(a_{(0,1),3},a_{(0,2),3},a_{(0,2),3},a_{(0,2),4},a_{(0,2),5},a_{(0,3),3},a_{(0,3),4},a_{(0,3),5},\beta_{2};b_{3};b_{ST}),
\end{align}

\end{problem}

We note that:

$\bullet$ (\ref{fundzeror4}) is the objective function of the weighted Fermat-Steiner problem for $A_{1}A_{2}A_{3}A_{4}A_{5}$ in $\mathbb{R}^{4}$ having three weighted Fermat-Steiner points $A_{0,i},$ such that:

\[\alpha_{1(0,1)2}=\arccos(\frac{b_{ST}^2 -b_{1}^2-b_{2}^2}{2 b_{1}b_{2}}),\]
\[\alpha_{1(0,1)(0,2)}=\arccos(\frac{b_{2}^2 -b_{1}^2-b_{ST}^2}{2 b_{1}b_{ST}}),\]
\[\alpha_{2(0,1)(0,2}=\arccos(\frac{b_{1}^2 -b_{2}^2-b_{ST}^2}{2 b_{2}b_{ST}}),\]

$A_{0,1}$ is the weight Fermat-Steiner point with respect to the boundary triangle $\triangle A_{1}A_{2}A_{0,2},$

\[\alpha_{(0,1)(0,2)3}=\arccos(\frac{-b_{3}}{2 b_{ST}}),\]
\[\alpha_{(0,1)(0,2)(0,3)}=\arccos(\frac{b_{3}^2 -2b_{ST}^2}{2 b_{ST}^2}),\]
\[\alpha_{3(0,2)(0,3}=\arccos(\frac{-b_{3}}{2 b_{ST}}),\]

$A_{0,2}$ is the weight Fermat-Steiner point with respect to the boundary triangle $\triangle A_{0,1}A_{3}A_{0,3},$

\[\alpha_{4(0,3)5}=\arccos(\frac{b_{ST}^2 -b_{4}^2-b_{5}^2}{2 b_{4}b_{5}}),\]
\[\alpha_{4(0,3)(0,2)}=\arccos(\frac{b_{5}^2 -b_{4}^2-b_{ST}^2}{2 b_{4}b_{ST}}),\]
\[\alpha_{5(0,3)(0,2}=\arccos(\frac{b_{4}^2 -b_{5}^2-b_{ST}^2}{2 b_{5}b_{ST}}),\]

$A_{0,3}$ is the weight Fermat-Steiner point with respect to the boundary triangle $\triangle A_{4}A_{5}A_{0,2}.$
These weighted angular relations are derived as a special case of Lemmas~\ref{ivantuzhimp1},~\ref{ivantuzhimp2}.

$\bullet$ (\ref{fundoner4})-(\ref{fundfourthr4})deal with the solution of the weighted Fermat problem for $A_{1}A_{2}A_{3}A_{4}A_{5}$ with
weights $w_{(0,1),1},w_{(0,1),2},w_{(0,1),3},w_{(0,1),4}$ and $w_{(0,1),5}=1-\sum_{i=1}^{4}w_{(0,1),i},$ which is determined by Theorem~\ref{RatioVolumesSimplex} for $N=4$ and the weighted Fermat point $A_{0,1}.$

$\bullet$ (\ref{fundfiver4})-(\ref{fundeightr4}) deal with the solution of the weighted Fermat problem for $A_{1}A_{2}A_{3}A_{4}A_{5}$ with
weights $w_{(0,2),2},w_{(0,2),3},w_{(0,2),4},w_{(0,2),5},$ and $w_{(0,2),1}=1-\sum_{i=2}^{5}w_{(0,2),i},$ which is determined by Theorem~\ref{RatioVolumesSimplex} for $N=4$ and the weighted Fermat point $A_{0,2}.$

$\bullet$ (\ref{fund9r4})-(\ref{fund12r4}) deal with the solution of the weighted Fermat problem for $A_{1}A_{2}A_{3}A_{4}A_{5}$ with
weights $w_{(0,3),2},w_{(0,3),3},w_{(0,3),4},w_{(0,3),5},$ and $w_{(0,3),1}=1-\sum_{i=2}^{5}w_{(0,2),i},$ which is determined by Theorem~\ref{RatioVolumesSimplex} for $N=4$ and the weighted Fermat point $A_{0,3}.$

$\bullet$  (\ref{fund13r4}) is a derivation of two expressions of $a_{(0,1),(0,2)}$ with respect to the boundary triangles $\triangle A_{1}A_{2}A_{0,2}$
$\triangle A_{0,1}A_{0,3}A_{3},$ by applying the generalized cosine law in $\mathbb{R}^{2}$ given in lemma~\ref{cosinelawr2}.

\begin{theorem}[Lagrange multiplier rule for the weighted Fermat-Steiner Frechet multitree in $\mathbb{R}^{4}$]\label{Lagrangerulemultitreer4}
If the admissible point $\tilde{x}_{i}$ yields a weighted minimum multitree for $1\le i \le 30.240,$ which correspond to a Frechet multifivesimplex derived by a tentuple of edge lengths determining upto 30.240 incongruent $4-$simplexes, then there are numbers $\lambda_{0i},\lambda_{1i},\lambda_{2i},\ldots \lambda_{13i},$ such that:

\begin{equation}\label{lagrangemultitreer3cond1r4}
\frac{\partial \mathcal{L}_{i}(\tilde{x}_{i},\tilde{\lambda}_{i})}{\partial x_{ji}}=0
\end{equation}
for $j=1,2,\ldots,24,$

\[\tilde{x}_{i}= \{a_{(0,1),1},a_{(0,1),2},a_{(0,1),3},\beta_{1},
a_{(0,2),3},a_{(0,2),4},a_{(0,5),3},\beta_{2}, a_{(0,3),3},a_{(0,3),4},a_{(0,3),5},\beta_{3},\]\[ w_{(0,1),1},\ldots, w_{(0,1),4},w_{(0,2),2},\ldots, w_{(0,2),5},w_{(0,3),2},\ldots, w_{(0,3),5}\}  \]

$\tilde{\lambda}_{i}=\{\lambda_{0},\lambda_{1},\ldots,\lambda_{13}\}\}.$

\end{theorem}

\begin{proof}
By taking into account that $\frac{\partial (f_{k})_{i} }{(x_{ji})}$ are continuous in each parallelepiped $\Pi_{i},$ for $1\le i\le 30.240,$ $k=0,1,2,\ldots, 13,$ $j=1,2,\ldots, 24$ and by applying Lagrange multiplier rule,
we obtain the Lagrangian vector $\tilde{\lambda}_{i}=\{\lambda_{0i},\lambda_{1i}\lambda_{2i},\ldots, \lambda_{13i},$ such that (\ref{lagrangemultitreer3cond1r4}) occurs.

\end{proof}


We give the definition of the four dimensional Dekster-Wilker Euclidean domain $DW_{\mathbb{R}^{4}}(\ell, s)$ discovered by Dekster-Wilker in \cite{DeksterWilker:87} and \cite{DeksterWilker:91a}. We denote by $\ell=\max_{i,j}a_{ij},$ $s=\min_{i,j}a_{ij}$ of the given tentuple $a_{ij}$ of positive real numbers.

\begin{definition}{The four dimensional Dekster-Wilker Euclidean domain, \cite{DeksterWilker:87}, \cite{DeksterWilker:91a}}
The four dimensional Dekster-Wilker Euclidean domain $DW_{\mathbb{R}^{4}}(\ell, s)$ is a closed domain in $\mathbb{R}^{2}$ between the ray $s=\ell,$
and the graph of a function $\lambda_{4}(\ell )=\ell \sqrt{\frac{7}{12}},$ $\ell \ge 0,$ which is less than $\ell$ for $\ell \ne 0,$
\end{definition}

\begin{proposition}[Lagrange multiplier rule for the Fermat-Steiner Frechet multitree in $\mathbb{R}^{4}$]\label{LagrangerulemultitreeDeksterWilkerr4}
If the admissible point $\tilde{x}_{i}$ yields a minimum multitree for $i=1,2,\ldots, 30.240,$ which correspond to a Frechet $4-$multisimplex derived by the Dekster-Wilker tentuples of edge lengths determining 30.240 incongruent tetrahedra, then there are numbers $\lambda_{0i},\lambda_{1i},\lambda_{2i},\ldots \lambda_{13i},$ such that:

\begin{equation}\label{lagrangemultitreer3cond1blr4}
\frac{\partial \mathcal{L}_{i}(\tilde{x}_{i},\tilde{\lambda}_{i})}{\partial x_{ji}}=0
\end{equation}
for $j=1,2,\ldots,24,$

\[\tilde{x}_{i}= \{a_{(0,1),1},a_{(0,1),2},a_{(0,1),3},\beta_{1},
a_{(0,2),3},a_{(0,2),4},a_{(0,5),3},\beta_{2}, a_{(0,3),3},a_{(0,3),4},a_{(0,3),5},\beta_{3},\]\[ w_{(0,1),1},\ldots, w_{(0,1),4},w_{(0,2),2},\ldots, w_{(0,2),5},w_{(0,3),2},\ldots, w_{(0,3),5}\}  \]

$\tilde{\lambda}_{i}=\{\lambda_{0},\lambda_{1},\ldots,\lambda_{13}\}\}.$

\end{proposition}

\begin{proof}
It is a direct consequence of Theorem~\ref{Lagrangerulemultitreer4} for Dekster-Wilker tentuples $\in DW_{\mathbb{R}^{4}}(\ell, s)$ determining 30.240 incongruent $4-$simplexes in $\mathbb{R}^{4}.$
\end{proof}

\begin{remark}\label{controlledrelationsr4}

Given that:
\[\alpha_{1(0,1)2}=\arccos(\frac{b_{ST}^2 -b_{1}^2-b_{2}^2}{2 b_{1}b_{2}}),\]
\[\alpha_{4(0,3)5}=\arccos(\frac{b_{ST}^2 -b_{4}^2-b_{5}^2}{2 b_{4}b_{5}}),\]
\[\alpha_{(0,1),(0,2),(0,3)}=\arccos(\frac{b_{3}^2 -b_{ST}^2-b_{ST}^2}{2 b_{ST}^2}),\]

\[\alpha_{i(0,k)j}=\arccos(\frac{a_{(0,k),i}^2+a_{(0,k),j}^2-a_{ij}^2}{2 a_{(0,k),i}a_{(0,k),j}}).\]
for $i,j=1,2,3,4,5,$ $k=1,2,3.$
Therefore, we get:
\[w_{(0,1),i}=w_{(0,1)i}(a_{(0,1)1},a_{(0,1),2},a_{(0,1),3},\beta_{1};b_{1});b_{2};b_{ST},\]
for $i=1,2,3,4,5.$
By following a similar process, we get:
\[w_{(0,3),i}=w_{(0,3),i}(a_{(0,3),4},a_{(0,3),5},a_{(0,3),3},\beta_{2};b_{4});b_{5};b_{ST},\]
for $i=1,2,3,4.$

\end{remark}

\begin{theorem}\label{mostnaturalconsecutivetentuplesr4}
The most natural tentuple of numbers from ten consecutive natural numbers $\{a+9,a+8,\ldots,a+1,a,\}$
for $a\ge 30$ is a tentuple of edge lengths having the maximum volume (maximum tentuple) among the 30.240 incongruent $4-$simplexes, which corresponds a Fermat-Steiner tree of minimum total weighted length (global minimum solution), such that the upper bound for the weight $B_{ST}$ is determined by the rest Fermat-Steiner minimal trees having larger or equal weighted minimal total length.
\end{theorem}

\begin{proof}
By applying Proposition~\ref{LagrangerulemultitreeDeksterWilkerr4}  for the Dekster-Wilker tentuple of edge lengths $\{a+9,a+8,\ldots,a+1,a,\}$ for $a\ge 30$ forming 30.240 incongruent $4-$simplexes in $\mathbb{R}^{4},$ we obtain a class of Fermat-Steiner trees for $b_{i}=1,$ for $i=1,2,3,4,5$ and $b_{ST}=1.$ This class of Fermat-Steiner trees forms the Fermat-Steiner-Frechet multitree in $\mathbb{R}^{4}.$
By selecting the proper tentuple of edge lengths, which yields the maximum volume of all the 30.240 incongruent $4-$simplexes in $\mathbb{R}^{4},$ we consider a variable weighted Fermat-Steiner tree having three equally weighted Fermat Steiner points with weight $B_{ST}$ and the same boundary weights $b_{i}=1$ Therefore, by perturbing the length weighted Fermat-Steiner tree structure for the $4-$simplex $(A_{1}A_{2}A_{3}A_{4}A_{5})_{Max}$ having the maximum volume, we can derive an upper bound for the variable weight $B_{ST},$ which is calculated by comparing the perturbed length tree structure with the pther Fermar-Steiner tree structures, that belong to the unweighted Fermat-Steiner-Frechet multitree in $\mathbb{R}^{4}.$ Hence, this particular arrangement of the ten consecutive natural numbers for $a\ge 30$ with the upper bound $(B_{ST})_{s}$ yield the most natural weighted tentuple of natural numbers building a weighted Fermat-Steiner tree with the minimum mass transfer.
\end{proof}


\section{The weighted Fermat-Steiner-Frechet multitree for a given $\frac{N(N+1)}{2}-$tuple of positive real numbers determining the edge lengths of incongruent $N-$simplexes in $\mathbb{R}^{N}$}
In this section, we deal with the solution (multitree) of the weighted Fermat-Steiner-Frechet problem (P(Fermat-Steiner-Frechet)) for a given $\frac{N}(N+1){2}$ of positive real numbers determining incongruent $N-$simplexes in $\mathbb{R}^{N},$ by inserting $N(N-1)$ equality constraints derived by $(N-1)$ independent solutions for $(N-1)$ variable weighted Fermat problems for the Frechet $N-$multisimplex derived by incongruent boundary $N-$simplexes in $\mathbb{R}^{N},$ which correspond to the same $\frac{N(N+1)}{2}-$tuple of positive real numbers (edge lengths) and $N-2$ equality constraints derived by two different expressions of each line segments connecting two consecutive weighted Fermat-Steiner points. By applying a Lagrange program, we can detect the weighted Fermat-Frechet multitree for a given $\frac{N(N+1)}{2}-$tuple of edge lengths determining incongruent $N-$ simplexes (Frechet $N-$multisimplex) in $\mathbb{R}^{N},$ by using the Dekster Wilker function. By seeking unweighted Fermat-Frechet multitrees with $(N-1)$ equally weighted Fermat-Steiner points inside the Frechet $N-$multisimplex, we can detect the most natural of $\frac{N(N+1)}{2}$ consecutive natural numbers (Dekster-Wilker $\frac{N(N+1)}{2}-$tuple) and it is achieved by seeking an upper bound for these equal weights, which yield a global weighted Fermat-Steiner tree of minimum length for a boundary $N-simplex$ having the maximum volume among the derived incongruent $N-$simplexes in $\mathbb{R}^{N}.$


We describe the $N$-dimensional Dekster-Wilker Euclidean domain (see in \cite{DeksterWilker:87}),\cite{DeksterWilker:91a}), which gives all incongruent $N-$simplexes in $\mathbb{R}^{N}$ derived by the same $\frac{N(N+1)}{2}-$tuple of positive real numbers $\{a_{ij}\}.$

Denote by $\ell=\max_{i,j}a_{ij},$ $s=\min_{i,j}a_{ij}$ and
\[ \lambda_{N}(\ell )= \left\{
\begin{array}{ll}
      \ell \sqrt{1-\frac{2(N+1)}{N(N+2)}} & for\ even\ N\ge 2, \\
      \ell \sqrt{1-\frac{2}{(N+1)}} & for\ odd\ N\ge 3

\end{array}
\right.
\]
\begin{definition}{The $N-$dimensional Dekster-Wilker Euclidean domain \cite{DeksterWilker:87},\cite{DeksterWilker:91a} }
The Dekster-Wilker Euclidean domain $DW_{\mathbb{R}^{N}}(\ell, s)$ is a closed domain in $\mathbb{R}^{2}$ between the ray $s=\ell,$
and the graph of a function $\lambda_{N}(\ell ),$ $\ell \ge 0,$ which is less than $\ell$ for $\ell \ne 0,$
\end{definition}

Let $A_{0,1},A_{0,2},\ldots A_{0,N-1}$ be $N-1$ points inside the $N-$simplex $A_{1}A_{2}\ldots A_{N}A_{N+1}$ in $\mathbb{R}^{N}.$
We denote by $a_{(0,i),j}$ the length of the line segment $A_{0,i}A_{j},$ by $\beta_{(0,k),u}$ the Schafli angle formed by the normals of the subpaces spanned by $\{A_{(0,k)}A_{1}A_{2}\ldots A_{u-2}\}$ and $\{A_{1}A_{2}\ldots A_{u-1}\}$ for $i=1,2,\ldots N$ $j=1,2,\ldots N+1,$ $k=1,2,\ldots N-1,$ $u=4,5,\ldots,N.$
The weighted Fermat-Steiner-Frechet problem for a given $\frac{N(N+1)}{2}-$tuple of edge lengths determining incongruent $N-$simplexes in $\mathbb{R}^{N},$ states that:

\begin{problem}[The weighted Fermat-Steiner Frechet in $\mathbb{R}^{N}$]\label{FermatSteinerFrechettetrahedronrn}
Given a $(2N)-$tuple of weights $\{b_{1},b_{2},\ldots ,b_{N+1},b_{ST},\ldots,b_{ST}\},$ and a given $\frac{N(N+1)}{2}-$tuple of positive real numbers (edge lengths) $\{a_{ij}\}\},$ determining a Frechet $N-$multisimplex $F(A_{1}A_{2}\ldots A_{N+1}),$   find the position of $A_{0,1}$ and / or
$A_{0,2}$ and/or $\ldots$ $A_{0,N-1}$  with given
weights $b_{ST}$ in $A_{0,1},$  $b_{ST}$ in $A_{0,2},$ $\ldots$ $b_{ST}$ in $A_{0,N-1},$ such that
\begin{align}\label{equat1L0rn}
f_{0}(T_{S})=\sum_{a_{i_{k},j_{k}}\in T_{S}}b_{k}a_{i_{k},j_{k}}\to min.
\end{align}
\end{problem}
\begin{definition}[A non degenerate weighted Fermat-Steiner tree for $\{A_{1}A_{2}\ldots A_{N+1}\}$]
A non degenerate weighted Fermat-Steiner tree $T_{S}$ is a tree, which consists of some line segments between the $N-1$ weighted Fermat-Steiner points $A_{0,j},$ such that each $A_{0,j}$ has degree (connections) three and of $N+1$ line segments joining each boundary vertex $A_{i}$ with some $A_{0,j}.$
\end{definition}

We consider the following three types of weighted Fermat-Steiner points $A_{0,i}:$ 

$\bullet$ Type~one, if it is connected with two boundary vertices and one weighted Fermat-Steiner point (see Fig.~\ref{figg4}),

$\bullet$ Type~two, if it is connected with one  boundary vertex and two weighted Fermat-Steiner points (see Fig.~\ref{figg5}),

$\bullet$ Type~three, if it is connected with three weighted Fermat-Steiner points (see Fig.~\ref{figg6}).

\begin{figure}\label{figg4}
\centering
\includegraphics[scale=0.80]{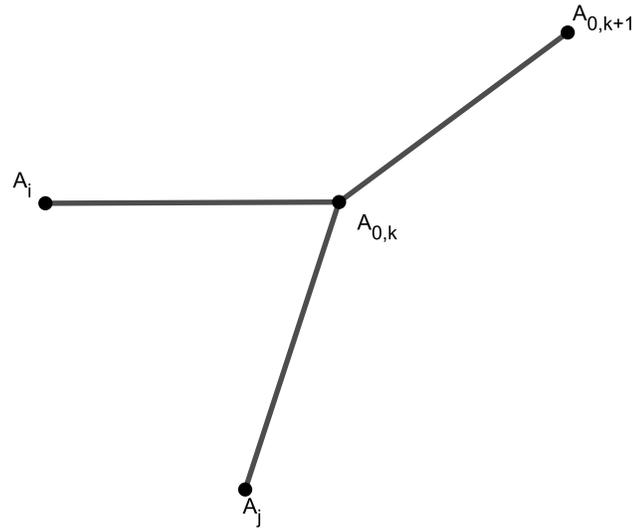}
\caption{A weighted Fermat-Steiner point type~one for $A_{1}A_{2}\ldots A_{N+1}$ in $\mathbb{R}^{N}$} \label{figg4}
\end{figure}

\begin{figure}\label{figg5}
\centering
\includegraphics[scale=0.80]{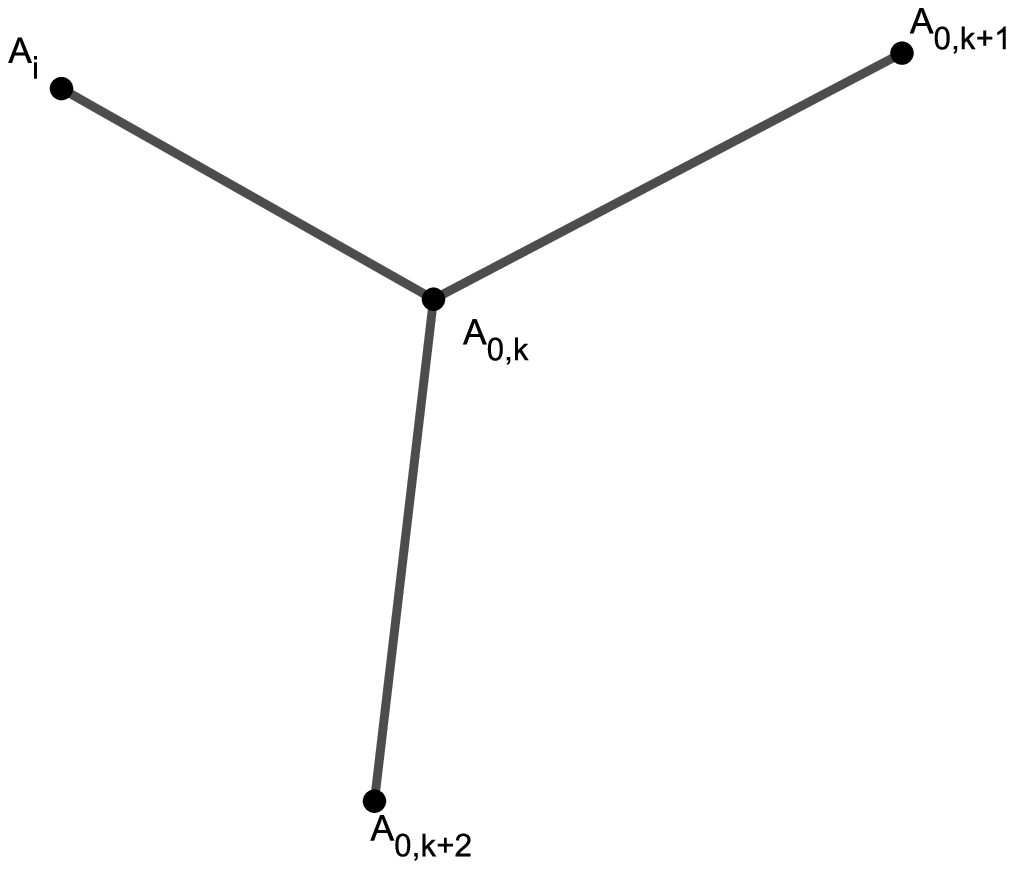}
\caption{A weighted Fermat-Steiner point type~two for $A_{1}A_{2}\ldots A_{N+1}$ in $\mathbb{R}^{N}$} \label{figg5}
\end{figure}

\begin{figure}\label{figg6}
\centering
\includegraphics[scale=0.80]{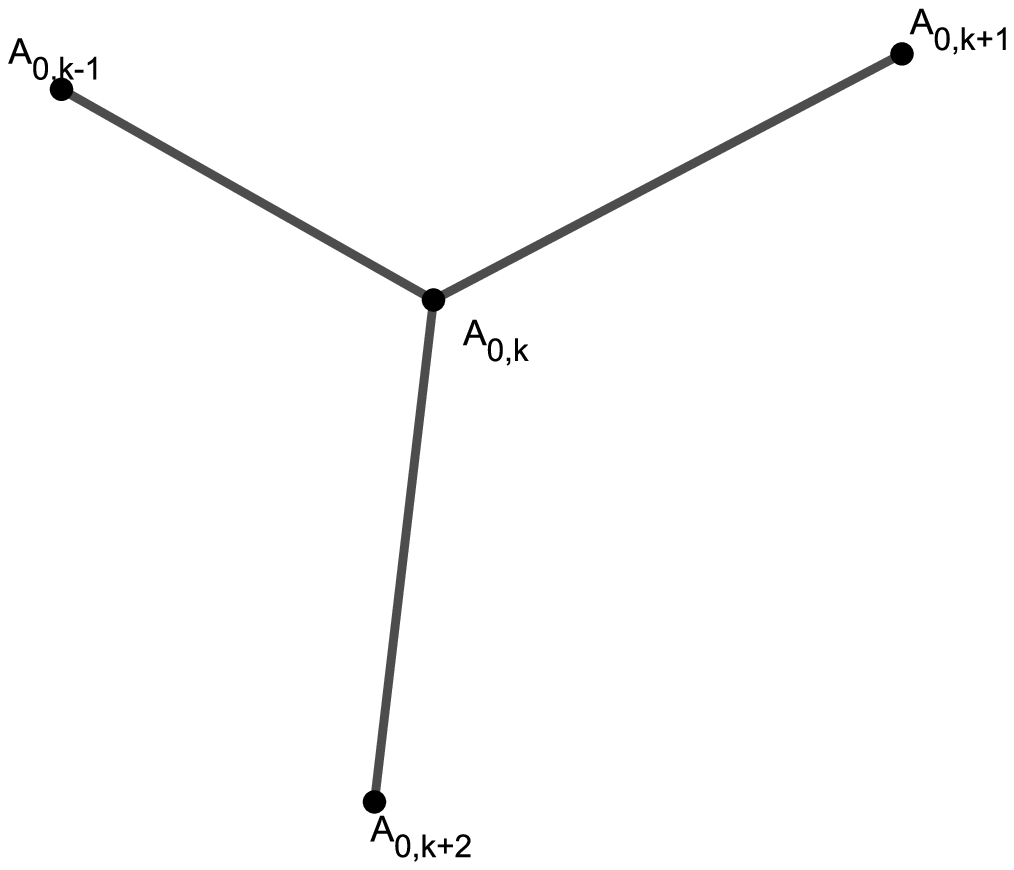}
\caption{A weighted Fermat-Steiner point type~three for $A_{1}A_{2}\ldots A_{N+1}$ in $\mathbb{R}^{N}$} \label{figg6}
\end{figure}


\begin{problem}[The weighted Fermat-Steiner-Frechet (P(Fermat-Steiner-Frechet)) problem in $\mathbb{R}^{N}$ with equality constraints]
\begin{equation*}
\begin{aligned}
& & f_{0}(\tilde{x})=\sum_{a_{i_{k},j_{k}}\in T_{S}}b_{k}a_{i_{k},j_{k}}, \\
& & f_{i,j}(\tilde{x}) = 0, \; i = 1, \ldots,(N-1), j=1,\ldots, N,\\
& & f_{k}(\tilde{x})=0,\; k=1,\ldots,N -2,
\end{aligned}
\end{equation*}
where $N-1$ is: $\#$ of neighboring weighted Fermat-Steiner points $A_{0,i}$ type~one,~two or three,
\begin{align}\label{fundzerorn}
f_{0}(\tilde{x})=f_{0}(T_{S})=\sum_{a_{i_{k},j_{k}}\in T_{S}}b_{k}a_{i_{k},j_{k}},
\end{align}

\begin{align}\label{fundonern}
f_{1,1}(\tilde{x})=\frac{w_{(0,1),1}}{a_{(0,1),1}\operatorname{Vol}(A_{0,1}\ldots A_{N+1})}-\nonumber\\ -\frac{1-\sum_{i=1}^{N}w_{(0,1),i}}{a_{(0,1),N+1}(a_{(0,1),1},a_{(0,1),2},a_{(0,1),3},\beta_{(0,1),4},\beta_{(0,1),5},\ldots, \beta_{(0,1),N})\operatorname{Vol}(A_{0,1}A_{1}\ldots A_{N})},
\end{align}
$\vdots$
\begin{align}\label{fundtworn}
f_{1,N}(\tilde{x})=\nonumber\\ \frac{w_{(0,1),N}}{a_{(0,1),N}(a_{(0,1),1},a_{(0,1),2},a_{(0,1),3},\beta_{(0,1),4},\beta_{(0,1),5},\ldots, \beta_{(0,1),N})\operatorname{Vol}(A_{1}\ldots A_{0,1}A_{N+1})}-\nonumber\\ -\frac{1-\sum_{i=1}^{N}w_{(0,1),i}}{a_{(0,1),N+1}(a_{(0,1),1},a_{(0,1),2},a_{(0,1),3},\beta_{(0,1),4},\beta_{(0,1),5},\ldots, \beta_{(0,1),N})\operatorname{Vol}(A_{0,1}A_{1}\ldots A_{N})},
\end{align}
$\vdots$

\begin{multline}\label{fundn-11rn}
f_{N-1,1}(\tilde{x})= \frac{w_{(0,N-1),1}}{a_{(0,N-1),1}\operatorname{Vol}(A_{0,N-1}\ldots A_{N+1})}-\\ -\frac{1-\sum_{i=1}^{N}w_{(0,N-1),i}}{a_{(0,N-1),N+1})\operatorname{Vol}(A_{0,N-1}A_{1}\ldots A_{N})},
\end{multline}
where \[a_{(0,N-1),N+1}=\]\[=a_{(0,N-1),N+1}(a_{(0,N-1),1},a_{(0,N-1),2},a_{(0,N-1),3},\beta_{(0,N-1),4},\beta_{(0,N-1),5},\ldots, \beta_{(0,N-1),N})\]
$\vdots$
\begin{multline}\label{fundtn-1nn}
f_{N-1,N}(\tilde{x})=\frac{w_{(0,N-1),N}}{a_{(0,N-1),N})\operatorname{Vol}(A_{1}\ldots A_{0,N-1}A_{N+1})}-\\ -\frac{1-\sum_{i=1}^{N}w_{(0,N-1),i}}{a_{(0,N-1),N+1}\operatorname{Vol}(A_{0,N-1}A_{1}\ldots A_{N})},
\end{multline}

where
\[a_{(0,N-1),N}=\]\[=a_{(0,N-1),N}(a_{(0,N-1),1},a_{(0,N-1),2},a_{(0,N-1),3},\beta_{(0,N-1),4},\beta_{(0,N-1),5},\ldots, \beta_{(0,N-1),N})\]

\begin{align}\label{funm1rn}
f_{k}(\tilde{x})=a_{(0,k),(0,k+1)}(\tilde{x};\triangle A_{i}A_{0,k+1}A_{j})-a_{(0,k),(0,k+1)}(\tilde{x};\triangle A_{k+2}A_{0,k}A_{l}),
\end{align}

or

\begin{align}\label{funm2rn}
f_{k}(\tilde{x})=a_{(0,k),(0,k+1)}(\tilde{x};\triangle A_{i}A_{0,k+1}A_{j})-a_{(0,k),(0,k+1)}(\tilde{x};\triangle A_{k+2}A_{0,k}A_{l}),
\end{align}

or

\begin{align}\label{funm3rn}
f_{k}(\tilde{x})=a_{(0,k),(0,k+1)}(\tilde{x};\triangle A_{i}A_{0,k+1}A_{j})-a_{(0,k),(0,k+1)}(\tilde{x};\triangle A_{k+2}A_{0,k}A_{k+l}),
\end{align}

or

\begin{align}\label{funm4rn}
f_{k}(\tilde{x})=a_{(0,k),(0,k+1)}(\tilde{x};\triangle A_{k}A_{0,k+1}A_{j})-a_{(0,k),(0,k+1)}(\tilde{x};\triangle A_{k+2}A_{0,k}A_{k+l}),
\end{align}

for $k=1,2,\ldots N-2.$
\end{problem}

We take into account that:

$\bullet$ (\ref{fundzerorn}) is the objective function of the weighted Fermat-Steiner problem for $A_{1}A_{2}\ldots A_{N+1}$ in $\mathbb{R}^{N}$ having $N-1$ weighted Fermat-Steiner points $A_{0,i},$ such that:

I. Type~one weighted Fermat-Steiner point $A_{0,i}$
\[\alpha_{j(0,i)k}=\arccos(\frac{b_{ST}^2 -b_{i}^2-b_{j}^2}{2 b_{i}b_{j}}),\]
\[\alpha_{j(0,i)(0,i+1)}=\arccos(\frac{b_{k}^2 -b_{j}^2-b_{ST}^2}{2 b_{j}b_{ST}}),\]
\[\alpha_{k(0,i)(0,i+1}=\arccos(\frac{b_{j}^2 -b_{k}^2-b_{ST}^2}{2 b_{k}b_{ST}}),\]

$A_{0,i}$ is the weight Fermat-Steiner point with respect to the boundary triangle $\triangle A_{j}A_{k}A_{0,i+1},$

II. Type~two weighted Fermat-Steiner point $A_{0,i}$

\[\alpha_{(0,i+1)(0,i)j}=\arccos(\frac{-b_{j}}{2 b_{ST}}),\]
\[\alpha_{(0,i+1)(0,i)(0,i+2)}=\arccos(\frac{b_{j}^2 -2b_{ST}^2}{2 b_{ST}^2}),\]
\[\alpha_{(0,i+1)(0,i)j}=\arccos(\frac{-b_{j}}{2 b_{ST}}),\]

$A_{0,2}$ is the weight Fermat-Steiner point with respect to the boundary triangle $\triangle A_{0,i+1}A_{0,i+2}A_{j},$

II. Type~three weighted Fermat-Steiner point $A_{0,i}$

\[\alpha_{(0,i+1)(0,i)(0,i+2)}=\alpha_{(0,i+1)(0,i)(0,i+3)}=120^{\circ},\]

$A_{0,i}$ is the unweighted Fermat-Steiner point with respect to the boundary triangle $\triangle A_{0,i+1}A_{0,i+2}A_{0,i+3}.$

$\bullet$ (\ref{fundonern})-(\ref{fundtworn})deal with the solution of the weighted Fermat problem for $A_{1}\ldots A_{N+1}$ with
weights $w_{(0,1),1},\ldots,w_{(0,1),N}$ and $w_{(0,1),N+1}=1-\sum_{i=1}^{N}w_{(0,1),i},$ which is determined by applying Theorem~\ref{RatioVolumesSimplex}  for the unique weighted Fermat point $A_{0,1}.$

$\vdots$

$\bullet$ (\ref{fundn-11rn})-(\ref{fundtn-1nn}) deal with the solution of the weighted Fermat problem for $A_{1}\ldots A_{N+1}$ with
weights $w_{(0,N-1),1},w_{(0,N-1),2},\ldots,w_{(0,N-1),N},$ and $w_{(0,N-1),N+1}=1-\sum_{i=1}^{N+1}w_{(0,N-1),i},$ which is determined by applying Theorem~\ref{RatioVolumesSimplex} and the unique weighted Fermat point $A_{0,N-1}.$

$\bullet$  (\ref{funm1rn})-(\ref{funm4rn}) is a derivation of possible expressions of $a_{(0,i),(0,i+1)}$ with respect to the boundary $\triangle A_{i}A_{0,k+1}A_{j},$ $\triangle A_{k+2}A_{0,k}A_{l}),$  $\triangle A_{k}A_{0,k+1}A_{j},$ $\triangle A_{k+2}A_{0,k}A_{k+l}$  by applying the generalized cosine law in $\mathbb{R}^{2}$ given in lemma~\ref{cosinelawr2}.

$\bullet$ The Schlafli angle $\beta_{(0,k),u}$ was discovered in \cite{Schlafli:50}.
The distance function $a_{(0,i),j}$ depends
\[a_{(0,i),j}=\]\[=a_{(0,i),j}(a_{(0,i),1},a_{(0,i),2},a_{(0,i),3},\beta_{(0,i),4},\beta_{(0,i),5},\ldots, \beta_{(0,i),j})\]
for $i=1,2,\ldots, N-1,$ $j=N,N+1.$

In $\mathbb{R}^{3},$ we derived that $a_{(0,i),4},$ depend on $a_{(0,i),1},$ $a_{(0,i),2},$ $a_{(0,i),3},$ because $\beta_{(0,i),4},$ can be expressed explicitly as a function with respect to $a_{(0,i),1},$ $a_{(0,i),2},$ $a_{(0,i),3}.$

In $\mathbb{R}^{4},$ we derived that $a_{(0,i),4},$ $a_{(0,i),5},$ depend on $a_{(0,i),1},$ $a_{(0,i),2},$ $a_{(0,i),3},$ $\beta_{(0,i),4},$ because
we derived an implicit function with respect to $a_{(0,i),1},$ $a_{(0,i),2},$ $a_{(0,i),3},$ $\beta_{(0,i),4}$ and $\beta_{(0,i),4}$ cannot be solved explicitly with respect to $a_{(0,i),1},$ $a_{(0,i),2},$ $a_{(0,i),3}.$

Therefore, Schlafli angles $\beta_{(0,i),k}$ are embodied in the computation of the distances $a_{(0,i),N},$ $a_{(0,i),N+1},$ for $N \ge 4$ and $k=4,5,\ldots, N.$

\begin{theorem}[Lagrange multiplier rule for the weighted Fermat-Steiner Frechet multitree in $\mathbb{R}^{N}$]\label{Lagrangerulemultitreern}
If the admissible point $\tilde{x}_{i}$ yields a weighted minimum multitree for $1\le i \le \frac{\frac{1}{2}N(N+1)!}{(N+1)!},$ which correspond to a Frechet $N-$multisimplex derived by a $\frac{N(N+1)}{2}-$tuple of edge lengths determining upto $\frac{\frac{1}{2}N(N+1)!}{(N+1)!}$ incongruent $N-$simplexes constructed by the Dekster-Wilker domain $DW_{\mathbb{R}^{N}}(\ell,s)$, then there are numbers $\lambda_{0i},\lambda_{1i},\lambda_{2i},\ldots \lambda_{(N^2-1)i},$ such that:

\begin{equation}\label{lagrangemultitreer3cond1rn}
\frac{\partial \mathcal{L}_{i}(\tilde{x}_{i},\tilde{\lambda}_{i})}{\partial x_{ji}}=0
\end{equation}
for $j=1,2,\ldots,2N(N-1),$

\[\tilde{x}_{i}= \{a_{(0,1),1},a_{(0,1),2},a_{(0,1),3},\beta_{(0,1),4},\ldots,\beta_{(0,1),N},\ldots,
a_{(0,N-1),1},a_{(0,N-1),2},\]\[a_{(0,N-1),3},\beta_{(0,N-1),4},\ldots,\beta_{(0,N-11),N},\]\[ w_{(0,1),1},\ldots, w_{(0,1),N},w_{(0,2),1},\ldots, w_{(0,2),N},\ldots,\] \[w_{(0,N-1),1},\ldots, w_{(0,N-1),N}\} \]

$\tilde{\lambda}_{i}=\{\lambda_{0},\lambda_{1},\ldots,\lambda_{N^2-1}\}\}.$

\end{theorem}

\begin{proof}
By setting $f_{p,q}\equiv f_{k},$ we consider that $\frac{\partial (f_{k})_{i} }{(x_{ji})}$ are continuous in each parallelepiped $\Pi_{i},$ for $1\le i\le 3\frac{\frac{1}{2}N(N+1)!}{(N+1)!},$ $k=0,1,2,\ldots, N^2-1,$ $j=1,2,\ldots, 2N(N-1)$ and by applying Lagrange multiplier rule,
we obtain the Lagrangian vector $\tilde{\lambda}_{i}=\{\lambda_{0i},\lambda_{1i}\lambda_{2i},\ldots, \lambda_{(N^2-1)i}.$ such that (\ref{lagrangemultitreer3cond1rn}) holds.

\end{proof}


Denote by

\[ a(N)= \left\{
\begin{array}{ll}
      \frac{(\frac{N(N+1)}{2}-1) \sqrt{1-\frac{2(N+1)}{N(N+2)}}}{1-\sqrt{1-\frac{2(N+1)}{N(N+2)}}} & for\ even\ N\ge 2, \\
      \frac{(\frac{N(N+1)}{2}-1) \sqrt{1-\frac{2}{(N+1)}}}{1-\sqrt{1-\frac{2}{(N+1)}}} & for\ odd\ N\ge 3

\end{array}
\right.
\]

The function $a(N)$ is obtained by the function $\ell (N)$ of Dekster-Wilker (see in \cite[p.~352]{DeksterWilker:87}) using the inequality
\[\frac{a(N)}{a(N)+\frac{N(N+1)}{2}-1}\ge \ell(N).\]

\begin{theorem}\label{mostnaturalconsecutivetentuplesrn}
The most natural $\frac{N(N+1)}{2}-$tuple of numbers from $\frac{N(N+1)}{2}$ consecutive natural numbers $\{a+\frac{N(N+1)}{2}-1,\ldots,a+1,a,\}$
for $a\ge a(N)$ is a $\frac{N(N+1)}{2}-$tuple of edge lengths having the maximum volume (maximum $\frac{N(N+1)}{2}-$tuple) among the $\frac{\frac{1}{2}N(N+1)!}{(N+1)!},$ incongruent $N-$simplexes, which corresponds to a Fermat-Steiner tree of minimum total weighted length (global minimum solution), such that the upper bound for the weight $B_{ST}$ is determined by the rest Fermat-Steiner minimal trees having larger or equal weighted minimal total length.
\end{theorem}

\begin{proof}
We will follow the same process that we used for sextuples and tentuples of consecutive natural numbers in $\mathbb{R}^{3}$ and $\mathbb{R}^{4},$ respectively.

By applying Theorem~\ref{Lagrangerulemultitreern} for the Dekster-Wilker $\frac{N(N+1)}{2}-$tuple of edge lengths $\{a+\frac{N(N+1)}{2}-1,\ldots,a+1,a,\}$ for $a\ge a(N)$ forming $\frac{\frac{1}{2}N(N+1)!}{(N+1)!},$ incongruent $N-$simplexes in $\mathbb{R}^{N},$ we obtain a class of Fermat-Steiner trees for $b_{i}=1,$ for $i=1,2,\ldots,N+1$ and $b_{ST}=1,$ which yields a Fermat-Steiner-Frechet multitree in $\mathbb{R}^{N}.$
By selecting the proper $\frac{N(N+1)}{2}-$tuple of edge lengths, which yields the maximum volume from all $\frac{\frac{1}{2}N(N+1)!}{(N+1)!},$ incongruent $N-$simplexes in $\mathbb{R}^{N},$ we consider a variable weighted Fermat-Steiner tree having three equally weighted Fermat Steiner points with weight $B_{ST}$ and the same boundary weights $b_{i}=1.$ By perturbing the length weighted Fermat-Steiner tree structure for the $N-$simplex $(A_{1}A_{2}\ldots A_{N+1})_{Max}$ having the maximum volume, we can derive an upper bound for the variable weight $B_{ST},$ which is calculated by comparing the perturbed length tree structure compared with the pther Fermar-Steiner tree structures, that belong to the unweighted Fermat-Steiner-Frechet multitree in $\mathbb{R}^{N}.$ Hence, this particular arrangement of the $\frac{N(N+1)}{2}$ consecutive natural numbers for $a\ge a(N)$ with the upper bound $(B_{ST})_{s}$ yields the most natural weighted $\frac{N(N+1)}{2}-$tuple of natural numbers referring to weighted Fermat-Steiner trees with the minimum mass transfer.
\end{proof}

\section{Intermediate weighted Fermat-Steiner-Frechet multitrees for a given $\frac{N(N+1)}{2}-$tuple of positive real numbers determining the edge lengths of incongruent $N-$simplexes in $\mathbb{R}^{N}$}
In this section, we deal with the solution (multitree) of the intermediate weighted Fermat-Steiner-Frechet problem (P(I.Fermat-Steiner-Frechet)) for a given $\frac{N}(N+1){2}$ of positive real numbers determining incongruent $N-$simplexes in $\mathbb{R}^{N},$ by inserting $N(l+s)$ equality constraints derived by $l+s<N-1$ independent solutions for $(l+s$ variable weighted Fermat problems for the Frechet $N-$multisimplex derived by incongruent boundary $N-$simplexes in $\mathbb{R}^{N},$ which correspond to the same $\frac{N(N+1)}{2}-$tuple of positive real numbers (edge lengths) and $l+s-1$ equality constraints derived by two different expressions of each line segments connecting two consecutive weighted Fermat-Steiner points. By applying a Lagrange program, we can detect intermediate weighted Fermat-Frechet multitrees for a given $\frac{N(N+1}{2}-$tuple of edge lengths determining incongruent $N-$ simplexes (Frechet $N-$multisimplex) in $\mathbb{R}^{N},$ by using the Dekster Wilker function.

First, we give the definitions of an intermediate weighted Fermat-Steiner tree and an intermediate weighted Fermat-Steiner multitree for a given $\frac{N(N+1)}{2}-$tuple of positive real numbers determining the edge lengths of incongruent $N-$simplexes in $\mathbb{R}^{N}.$

\begin{definition}[A non degenerate intermediate weighted Fermat-Steiner tree for $\{A_{1}A_{2}\ldots A_{N+1}\}$]
A non degenerate weighted Fermat-Steiner tree $T_{IS}$ is a tree, which consists of some line segments between the weighted Fermat-Steiner points $A_{0,j},$ and weighted Fermat points $P_{0,j}$ whose
$\# <N-1,$ such that each $A_{0,j}$ has degree three, each $P_{0,j}$ has degree more than three and of $N+1$ line segments joining each boundary vertex $A_{i}$ with some $A_{0,j},$ or $P_{0,j}.$
\end{definition}

\begin{definition}
An intermediate weighted Fermat-Steiner multitree $MT_{IS}$ is a union of intermediate weighted Fermat-Steiner trees, which correspond to incongruent $N-$simplexes derived by the same Dekster-Wilker $\frac{N(N+1}{2}-$tuple of positive real numbers determining edge lengths.
\end{definition}


We consider an intermediate weighted Fermat-Steiner multitree in $\mathbb{R}^{N}$ having incongruent boundary simplexes, with $s$ a fixed $\#$ of weighted Fermat-Steiner points and $l$ a fixed $\#$ of weighted Fermat points in $\mathbb{R}^{N},$ such that $s+l < N-1.$

\begin{problem}[The intermediate weighted Fermat-Steiner-Frechet (P(I.Fermat-Steiner-Frechet)) problem in $\mathbb{R}^{N}$ with equality constraints]
\begin{equation*}
\begin{aligned}
& & f_{0}(\tilde{x})=\sum_{a_{i_{k},j_{k}}\in T_{S}}b_{k}a_{i_{k},j_{k}}, \\
& & f_{i,j}(\tilde{x}) = 0, \; i = 1, \ldots,l+s), j=1,\ldots, N,\\
& & f_{k}(\tilde{x})=0,\; k=1,\ldots,l+s-1,
\end{aligned}
\end{equation*}
where $s$ is $\#$ of weighted Fermat-Steiner points $A_{0,i}$ type~one,~two or three and $l$ is $\#$ of weighted Fermat-Steiner points $A_{0,i}$
of degree (connections) more than three.
\begin{align}\label{fundzerornis}
f_{0}(\tilde{x})=f_{0}(T_{S})=\sum_{a_{i_{k},j_{k}}\in T_{S}}b_{k}a_{i_{k},j_{k}},
\end{align}

\begin{align}\label{fundonernis}
f_{1,1}(\tilde{x})=\frac{w_{(0,1),1}}{a_{(0,1),1}\operatorname{Vol}(A_{0,1}\ldots A_{N+1})}-\nonumber\\ -\frac{1-\sum_{i=1}^{N}w_{(0,1),i}}{a_{(0,1),N+1}(a_{(0,1),1},a_{(0,1),2},a_{(0,1),3},\beta_{(0,1),4},\beta_{(0,1),5},\ldots, \beta_{(0,1),N})\operatorname{Vol}(A_{0,1}A_{1}\ldots A_{N})},
\end{align}
$\vdots$
\begin{align}\label{fundtwornis}
f_{1,N}(\tilde{x})=\nonumber\\ \frac{w_{(0,1),N}}{a_{(0,1),N}(a_{(0,1),1},a_{(0,1),2},a_{(0,1),3},\beta_{(0,1),4},\beta_{(0,1),5},\ldots, \beta_{(0,1),N})\operatorname{Vol}(A_{1}\ldots A_{0,1}A_{N+1})}-\nonumber\\ -\frac{1-\sum_{i=1}^{N}w_{(0,1),i}}{a_{(0,1),N+1}(a_{(0,1),1},a_{(0,1),2},a_{(0,1),3},\beta_{(0,1),4},\beta_{(0,1),5},\ldots, \beta_{(0,1),N})\operatorname{Vol}(A_{0,1}A_{1}\ldots A_{N})},
\end{align}
$\vdots$

\begin{multline}\label{fundn-11rnis}
f_{l+s,1}(\tilde{x})= \frac{w_{(0,l+s),1}}{a_{(0,l+s),1}\operatorname{Vol}(A_{0,l+s}\ldots A_{N+1})}-\\ -\frac{1-\sum_{i=1}^{N}w_{(0,l+s),i}}{a_{(0,l+s),N+1})\operatorname{Vol}(A_{0,l+s}A_{1}\ldots A_{N})},
\end{multline}
where \[a_{(0,l+s),N+1}=\]\[=a_{(0,l+s),N+1}(a_{(0,l+s),1},a_{(0,l+s),2},a_{(0,l+s),3},\beta_{(0,l+s),4},\beta_{(0,l+s),5},\ldots, \beta_{(0,l+s),N})\]
$\vdots$
\begin{multline}\label{fundtn-1nnis}
f_{l+s,N}(\tilde{x})=\frac{w_{(0,l+s),N}}{a_{(0,l+s),N})\operatorname{Vol}(A_{1}\ldots A_{0,l+s}A_{N+1})}-\\ -\frac{1-\sum_{i=1}^{N}w_{(0,l+s),i}}{a_{(0,l+s),N+1}\operatorname{Vol}(A_{0,l+s}A_{1}\ldots A_{N})},
\end{multline}

where
\[a_{(0,l+s),N}=\]\[=a_{(0,l+s),N}(a_{(0,l+s),1},a_{(0,l+s),2},a_{(0,l+s),3},\beta_{(0,l+s),4},\beta_{(0,l+s),5},\ldots, \beta_{(0,l+s),N})\]

\begin{align}\label{funm1rnis}
f_{k}(\tilde{x})=a_{(0,k),(0,k+1)}(\tilde{x};boundary simplex~1)-a_{(0,k),(0,k+1)}(\tilde{x};boundary simplex~2),
\end{align}

for $k=1,2,\ldots l+s-1.$
\end{problem}

We take into account that:

$\bullet$ (\ref{fundzerornis}) is the objective function of the intermediate weighted Fermat-Steiner problem for $A_{1}A_{2}\ldots A_{N+1}$ in $\mathbb{R}^{N}$ having $N-1$ weighted Fermat-Steiner points $A_{0,i},$ such that:

There are $l$ weighted Fermat-Steiner points $A_{0,i}$ of degree three with respect to a boundary triangle and $s$ weighted Fermat points $A_{0,i}$ of degree more than three with respect to a boundary $r-$simplex in $\mathbb{R}^{N},$ for $r<N,$ which can be derived by the local minimality criterion of Lemma~\ref{ivantuzhimp1}.

$\bullet$ (\ref{fundonernis})-(\ref{fundtwornis})deal with the solution of the weighted Fermat problem for $A_{1}\ldots A_{N+1}$ with
variable weights $w_{(0,1),1},\ldots,w_{(0,1),N}$ and $w_{(0,1),N+1}=1-\sum_{i=1}^{N}w_{(0,1),i},$ which is determined by applying Theorem~\ref{RatioVolumesSimplex}  for the unique weighted Fermat point $A_{0,1}.$

$\vdots$

$\bullet$ (\ref{fundn-11rnis})-(\ref{fundtn-1nnis}) deal with the solution of the weighted Fermat problem for $A_{1}\ldots A_{N+1}$ with
variable weights $w_{(0,l+s),1},w_{(0,l+s),2},\ldots,w_{(0,l+s),N},$ and $w_{(0,l+s),N+1}=1-\sum_{i=1}^{N+1}w_{(0,l+s),i},$ which is determined by applying Theorem~\ref{RatioVolumesSimplex} and the unique weighted Fermat point $A_{0,l+s}.$

$\bullet$  (\ref{funm1rnis}) is a derivation of possible expressions of $a_{(0,i),(0,i+1)}$ with respect to boundary $r-$simplexes for $2 \le r < N$ by applying the generalized cosine law in $\mathbb{R}^{2}$ given in lemma~\ref{cosinelawr2} combined with Theorem~\ref{RatioVolumesSimplex} for $A_{i_{1}}\ldots A_{i_{r+1}},$ with $r+1$ fixed given weights taken from the set $\{b_{1},b_{2},\ldots,b_{N+1},b_{ST}\}.$


The following Lagrangian program detects an intermediate weighted Fermat-Steiner-Frechet multitree in $\mathbb{R}^{N}.$

\begin{theorem}[Lagrange multiplier rule for the weighted Fermat-Steiner Frechet multitree in $\mathbb{R}^{N}$]\label{Lagrangerulemultitreer4is}
If the admissible point $\tilde{x}_{i}$ yields an intermediate weighted minimum multitree for $1\le i \le \frac{\frac{1}{2}N(N+1)!}{(N+1)!},$ which correspond to a Frechet $N-$multisimplex derived by a $\frac{N(N+1)}{2}-$tuple of edge lengths determining upto $\frac{\frac{1}{2}N(N+1)!}{(N+1)!}$ incongruent $N-$simplexes constructed by the Dekster-Wilker domain $DW_{\mathbb{R}^{N}}(\ell,s)$, then there are numbers $\lambda_{0i},\lambda_{1i},\lambda_{2i},\ldots \lambda_{ ((l+s-1)(N+1)+1)i)},$ such that:

\begin{equation}\label{lagrangemultitreer3cond1rnis}
\frac{\partial \mathcal{L}_{i}(\tilde{x}_{i},\tilde{\lambda}_{i})}{\partial x_{ji}}=0
\end{equation}
for $j=1,2,\ldots,2N(l+s-1),$

\[\tilde{x}_{i}= \{a_{(0,1),1},a_{(0,1),2},a_{(0,1),3},\beta_{(0,1),4},\ldots,\beta_{(0,1),N},\ldots,
a_{(0,l+s),1},a_{(0,l+s),2},\]\[a_{(0,l+s),3},\beta_{(0,l+s),4},\ldots,\beta_{(0,l+s),N},\]\[ w_{(0,1),1},\ldots, w_{(0,1),N},w_{(0,2),1},\ldots, w_{(0,2),N},\ldots,\] \[w_{(0,l+s),1},\ldots, w_{(0,l+s),N}\} \]

$\tilde{\lambda}_{i}=\{\lambda_{0i},\lambda_{1i},\ldots,\lambda_{((l+s-1)(N+1)+1i)}\}.$

\end{theorem}

\begin{proof}
By setting $f(p,q)\equiv f_{k}$ we consider that $\frac{\partial (f_{k})_{i} }{(x_{ji})}$  are continuous in each parallelepiped $\Pi_{i},$ for $1\le i\le 3\frac{\frac{1}{2}N(N+1)!}{(N+1)!},$ $k=0,1,2,\ldots, N^2-1,$ $j=1,2,\ldots, 2N(l+s-1)$ and by applying Lagrange multiplier rule,
we obtain the Lagrangian vector $\tilde{\lambda}_{i}=\{\lambda_{0i},\lambda_{1i}\lambda_{2i},\ldots, \lambda_{((l+s-1)(N+1)+1i)},$ such that (\ref{lagrangemultitreer3cond1rnis}) holds.

\end{proof}

\begin{example}\label{multitreer5is}
Let  $\{a_{ij}\}$ be a given $15-$th tuple of positive real numbers determining incongruent $5-$simplexes using Dekster-Wilker conditions in $\mathbb{R}^{5}$ and $A_{1}A_{2}\ldots A_{6}$ be a member of the Frechet $5-$multisimplex in $\mathbb{R}^{5}.$ We consider a fixed tree topology that contains two weighted Fermat-Steiner points $A_{0,1},$ $A_{0,2}$ of degree three and one weighted Fermat point of degree four (see Fig.~\ref{figgr5}).

\begin{figure}\label{figgr5}
\centering
\includegraphics[scale=0.80]{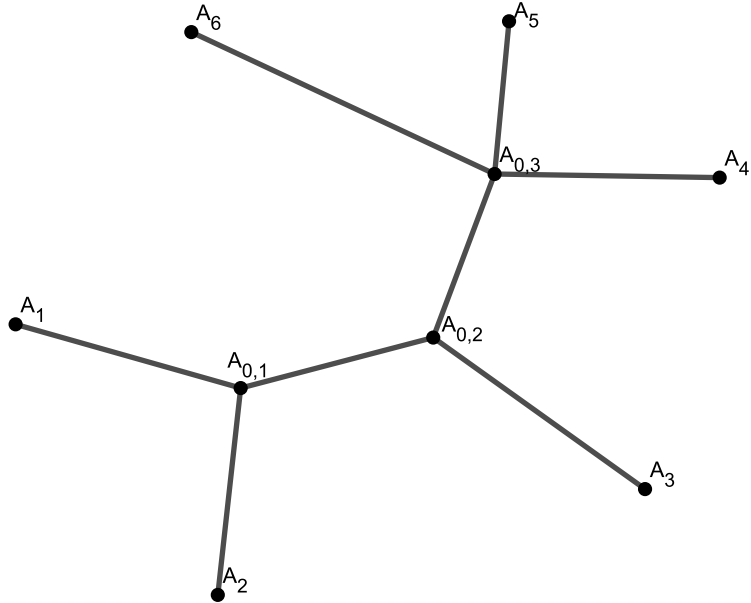}
\caption{An intermediate weighted Fermat-Steiner-Frechet multitree for a boundary $5-$simplex $A_{1}A_{2}\ldots A_{6}$ in $\mathbb{R}^{5}$} \label{figg6}
\end{figure}

The intermediate weighted Fermat-Steiner tree for $A_{1}A_{2}A_{3}A_{4}A_{5}A_{6}$ in $\mathbb{R}^{5}$ consists of the line segments
\[\{A_{0,1}A_{1},A_{0,1}A_{2},A_{0,1}A_{0,2},A_{0,2}A_{3},A_{0,2}A_{0,3},A_{0,3}A_{4},A_{0,3}A_{5},A_{0,3}A_{6}\}.\]

By applying the Lagrangian program of Theorem~\ref{Lagrangerulemultitreer4is}, we can detect intermediate weighted Fermat-Steiner-Frechet multitrees for incongruent $5-$simplexes in $\mathbb{R}^{5}$ derived by a given Dekster-Wilker $15-$tuple of edge lengths.

\end{example}


\section{"Mutation" of an intermediate weighted Fermat Steiner Frechet multitree for $m$ Closed Polytopes in $\mathbb{R}^{N}$}
In this section, we introduce the "mutation" of an intermediate weighted Fermat Steiner Frechet multitree for  $m$ boundary polytopes, by applying the plasticity solutions of the $m-$INVWF  problem for $A_{1}A_{2}\ldots A_{m+1}$ in $\mathbb{R}^{N},$ enriched by a two way mass transportation network.
The dynamic plasticity solutions of the $m-$ INVWF problem in $\mathbb{R}^{N}$ combined with conditions derived by optimal mass transport and storage gives the "mutation of an intermediate weighted Fermat-Frechet multitree for boundary $m$ closed polytopes in $\mathbb{R}^{N}.$

We start by deriving the geometric and dynamic plasticity for $m$ closed polytopes in $\mathbb{R}^{N}$
for $m\ge N+1.$
\begin{definition}\label{defgeomplasticity}
We call \textit{geometric plasticity} of a weighted Fermat tree whose endpoints correspond to a closed polytope in $\mathbb{R}^{N},$ which is formed by $m+1$
variable line segments meeting at the weighted Fermat point $A_{0},$  for
$m+1$ given values of the weights, the set of solutions of the $m+1$
variable lengths of the line segments, which correspond to a family
of weighted networks that preserve the weighted Fermat point $A_{0}.$
\end{definition}

We assume that we select $B_{i}$ that correspond to each vertex $A_{i},$ such that the weighted floating inequalities of Theorem~\ref{theor1} hold:
\[ \|\sum_{j=1,j\ne i}^{m+1}B_{j} \vec {u}(A_j,A_i)\|>B_{i}, \]
for $i,j=1,2,\ldots,m+1.$
Thus, the weighted Fermat point $A_{0}$ of an $m$ closed polytope $A_{1}A_{2}\ldots A_{m}$ in
$\mathbb{R}^{N}$ is the unique intersection point of $m+1$ line segments $A_{0}A_{m}.$

\begin{theorem}[Geometric plasticity of $m$ closed polytopes in $\mathbb{R}^{n}$]\label{geomplasticity}
If we select a point $A_{i}^{\prime}$ on the ray defined by $A_{0}A_{i}$ with corresponding weight $B_{i},$ such that \[ \|\sum_{j=1,j\ne i}^{m+1}B_{j} \vec {u}(A_{j}^{\prime},A_{i}^{\prime})\|>B_{i}, \]
for $i,j=1,2,\ldots,m+1,$ then the corresponding weighted Fermat point $A_{0}^{\prime}\equiv A_{0}.$

\end{theorem}

\begin{proof}
The weighted floating inequalities
\[ \|\sum_{j=1,j\ne i}^{m+1}B_{j} \vec {u}(A_j,A_i)\|>B_{i},\]
\[ \|\sum_{j=1,j\ne i}^{m+1}B_{j} \vec {u}(A_{j}^{\prime},A_{i}^{\prime})\|>B_{i}, \]
yield respectively,
\[\sum_{i=1}^{m+1}B_{i}\vec {u}(A_0,A_i)=\vec{0},\]
\[\sum_{i=1}^{m+1}B_{i}\vec {u}(A_0,A_{i}^{\prime})=\vec{0}.\]
Thus, we derive that $A_{0}^{\prime}\equiv A_{0}.$

\end{proof}

We assume that $m+1$ line segments $A_{0}A_{i}$ intersect at $A_{0}$ in $\mathbb{R}^{N},$ for $m\ge N+1.$
We note that:

(1) for $n=2,$ $m$ angles $\alpha_{i0j}$ determine the dynamic plasticity equations of $A_{1}A_{2}\ldots A_{m+1}$
in $\mathbb{R}^{2}.$

(2) for $n=3,$ $m+(m-1)$ angles $\alpha_{i0j}$ determine the dynamic plasticity equations of $A_{1}A_{2}\ldots A_{m+1}$ in $\mathbb{R}^{3}.$

(3) for $n>3,$ $\sum_{i=1}^{N-1}(m+1-i)$ angles $\alpha_{i0j}$ determine the dynamic plasticity equations of $A_{1}A_{2}\ldots A_{m+1}$ in $\mathbb{R}^{N}.$

We set $\sum_{12\ldots m+1}B\equiv \sum_{i=1}^{m+1} (B_{i})_{12\ldots m+1},$\\
$\sum_{i_{1}i_{2}\ldots i_{N+1}}B=(B_{i_{1}})_{i_{1}i_{2}\ldots i_{N+1}}+\ldots +(B_{i_{N+1}})_{i_{1}i_{2}\ldots i_{N+1}},$\\ for $i_{1},\ldots i_{N+1} \in\{1,2,\ldots m+1\}.$

\begin{lemma}\label{dependenceweightsm}
If $\,\sum_{12\ldots m+1}B=\sum_{i_{1}i_{2}\ldots i_{n+1}}B,$ for every $i_{1},\ldots i_{n+1} \in\{1,2,\ldots m+1\},$ where
$\sum_{12\ldots m+1}B:=(B_{N+1})_{12\ldots m+1}(1+\sum_{i=1, i\neq N+1}^{m}(\frac{B_i}{B_{N+1}})_{1,2,\ldots m+1}),$
then
\begin{eqnarray}\label{B12nplus2}
 (B_{i})_{12\ldots m+1}=\sum_{j=N+2}^{m}a_{i,j} (B_{j})_{12\ldots m+1}+ b_i,\quad i=1,2\ldots, N+1,
\end{eqnarray}
where

\begin{eqnarray}\label{abnplusonej}
\nonumber
&&(a_{N+1,j},\,b_{N+1})=\nonumber\\&&{} (\frac{(\frac{B_{1}}{B_{N+1}})_{12\ldots N+1}(\frac{B_{N+1}}{B_{j}})_{2\ldots j}+\dots +
(\frac{B_{N}}{B_{N+1}})_{12\ldots N+1}(\frac{B_{N+1}}{B_{j}})_{12\ldots (N-1)(N+1)j}  -1}
{\sum_{i=1}^{N+1}(\frac{B_{i}}{B_{N+1}})_{12\ldots N+1}  },\ \nonumber\\&&{}
(B_{N+1})_{12\ldots N+1}),
\nonumber
\end{eqnarray}

\begin{eqnarray}\label{abn}
\nonumber
&&(a_{N,j},\,b_{N})=(a_{N+1,j}(\frac{B_{N}}{B_{N+1}})_{12\ldots N+1}-(\frac{B_{N}}{B_{N+1}})_{12\ldots N+1}(\frac{B_{N+1}}{B_{j}})_{12\ldots (N-1)(N+1)j},\ \nonumber\\&&{}(B_{N})_{12\ldots N+1}),
\nonumber
\end{eqnarray}
$\vdots$
\begin{eqnarray}\label{ab1}
\nonumber
&&(a_{1,j},\,b_{1})=(a_{N+1,j}(\frac{B_{1}}{B_{N+1}})_{12\ldots N+1}-(\frac{B_{1}}{B_{N+1}})_{12\ldots N+1}(\frac{B_{N+1}}{B_{j}})_{23\ldots j},\ \nonumber\\&&{}(B_{1})_{12\ldots N+1}),
\nonumber
\end{eqnarray}
for $j=n+2,\ldots m+1.$

\end{lemma}

By assuming mass flow continuity and by applying the geometric and dynamic plasticity of $m$ polytopes in $\mathbb{R}^{N},$ (Theorem~\ref{geomplasticity}, Lemma~\ref{dependenceweightsm}) the corresponding weighted Fermat tree solutions yield some mass transportation networks, in which the weights correspond to an instantaneous collection of images of masses, which satisfy some specific conditions.

\begin{definition}\label{mutationpolytopes}
We call $(m,k)$ "mutation" of an I.Fermst-Steiner tree with respect to boundary a$m$ polytope the plasticity solutions of the $m-$INVWF  problem for $A_{1}A_{2}\ldots A_{m+1}$ in $\mathbb{R}^{N},$ enriched by a two way mass transportation network, such that $k$ masses (weights) are transferred in $k$ directions from $A_{i}\to A_{0},$ for $i=1,\ldots k$ (inflow) creating a storage at $A_{0}$ and $N-k$ masses are transferred in $m-k$ directions from $A_{0}\to A_{j}$ (outflow), for $j=k+1,\ldots m+1$ and reversely $m-k$ masses (new weights) are transferred back to $A_{0}$ along the same $m-k$ directions creating new storage at $A_{0},$ and $k$ masses (new weights) are transferred back from $A_{0}$ to $A_{i},$ $i=1,2,\ldots k.$
\end{definition}
We denote by $B_{i}$ a mass flow which is transferred from $A_{i}$
to $A_{0}$ for $i=1,2,\ldots k$ by $B_{0}$ a residual weight which
remains at $A_{0}$ and by $B_{k+1},\ldots B_{m+1}$ a mass flow which is transferred
from $A_{0}$ to $A_{k+1},\ldots, A_{m+1}.$

We denote by $\tilde{B_{i}}$ a mass flow which is transferred from
$A_{0}$ to $A_{i}$ for $i=1,2,\ldots k$ by $\tilde{B_{0}}$ a residual
weight which remains at $A_{0}$ and by $\tilde{B_{k+1}},\ldots,B_{m+1}$ a mass flow
which is transferred from $A_{k+1},\ldots A_{m+1}$ to $A_{0}.$

Thus, we derive that:

\begin{equation}\label{weight1outflow}
\sum_{i=1}^{k}B_{i}=\sum_{k+1}^{m+1}{B_{i}}+B_{0}
\end{equation}

and

\begin{equation}\label{weight2inflow}
\sum_{i=1}^{k}\tilde{B_{i}}+\tilde{B_{0}}=\sum_{i=k+1}^{m+1}\tilde{B_{i}}.
\end{equation}

By adding (\ref{weight1outflow}) and (\ref{weight2inflow}) and by
letting $\bar{B_{0}}=B_{0}-\tilde{B_{0}}$ we get:

\begin{equation}\label{weight12inoutflow}
\sum_{i=1}^{k}\bar{B_{i}}=\sum_{i=k+1}^{m+1}\bar{B_{i}}+\bar{B_{0}}
\end{equation}

such that:

\begin{equation}\label{weight12inflowsum}
\sum_{i=1}^{m+1}\bar{B_{i}}=c,
\end{equation}
where $c$ is a positive real number.

Thus, we derive the following theorem as a direct consequence of Lemma~\ref{dependenceweightsm} under the condition for the weights $\bar{B}_{12\ldots m+1}$ taken from (\ref{weight12inoutflow}), (\ref{weight12inflowsum}), which deal with the $(m,k)$ "mutation" of $m$ closed polytopes in $\mathbb{R}^{N}.$
\begin{theorem}{$(m,k)$ "mutation" of an intermediate weighted Fermat-Steiner tree for boundary $m$ closed polytopes in $\mathbb{R}^{N}$}\label{mutationmk}\\
If $\,\sum_{12\ldots m+1}\bar{B}=\sum_{i_{1}i_{2}\ldots i_{n+1}}\bar{B},$ for every $i_{1},\ldots i_{n+1} \in\{1,2,\ldots m+1\},$ where
$\sum_{12\ldots m+1}\bar{B}:=(\bar{B}_{n+1})_{12\ldots m+1}(1+\sum_{i=1, i\neq n+1}^{m+1}(\frac{\bar{B}_i}{\bar{B}_{N+1}})_{1,2,\ldots m+1}),$
then
\begin{eqnarray}\label{B12nplus2}
 (\bar{B}_{i})_{12\ldots m+1}=\sum_{j=n+2}^{m+1}a_{i,j} (\bar{B}_{j})_{12\ldots m+1}+ b_i,\quad i=1,2\ldots, N+1,
\end{eqnarray}
where

\begin{eqnarray}\label{abnplusonejj}
\nonumber
&&(a_{N+1,j},\,b_{N+1})=\nonumber\\&&{} (\frac{(\frac{\bar{B}_{1}}{\bar{B}_{N+1}})_{12\ldots N+1}(\frac{\bar{B}_{N+1}}{\bar{B}_{j}})_{2\ldots j}+\dots +
(\frac{\bar{B}_{N}}{\bar{B}_{N+1}})_{12\ldots N+1}(\frac{\bar{B}_{N+1}}{\bar{B}_{j}})_{12\ldots (N-1)(N+1)j}  -1}
{\sum_{i=1}^{n+1}(\frac{\bar{B}_{i}}{\bar{B}_{N+1}})_{12\ldots N+1}  },\ \nonumber\\&&{}
(\bar{B}_{N+1})_{12\ldots N+1}),
\nonumber
\end{eqnarray}

\begin{eqnarray}\label{abn}
\nonumber
&&(a_{N,j},\,b_{N})=(a_{N+1,j}(\frac{\bar{B}_{N}}{\bar{B}_{N+1}})_{12\ldots N+1}-(\frac{\bar{B}_{N}}{\bar{B}_{N+1}})_{12\ldots N+1}(\frac{\bar{B}_{N+1}}{\bar{B}_{j}})_{12\ldots (N-1)(N+1)j},\ \nonumber\\&&{}(\bar{B}_{N})_{12\ldots N+1}),
\nonumber
\end{eqnarray}
$\vdots$
\begin{eqnarray}\label{ab1}
\nonumber
&&(a_{1,j},\,b_{1})=(a_{N+1,j}(\frac{\bar{B}_{1}}{\bar{B}_{N+1}})_{12\ldots N+1}-(\frac{\bar{B}_{1}}{\bar{B}_{N+1}})_{12\ldots N+1}(\frac{\bar{B}_{N+1}}{\bar{B}_{j}})_{23\ldots j},\ \nonumber\\&&{}(\bar{B}_{1})_{12\ldots N+1}),
\nonumber
\end{eqnarray}
for $j=n+2,\ldots m+1,$
under the conditions for the weights:
\[\sum_{i=1}^{k}(\bar{B_{i}})_{12\ldots m+1}=\sum_{i=k+1}^{m+1}(\bar{B_{i}})_{12\ldots m+1}+(\bar{B_{0}})_{12\ldots m+1}\]

\[\sum_{i=1}^{m+1}(\bar{B_{i}})_{12\ldots m+1}=c.\]

\end{theorem}

Suppose that at time $t=0,$ an I.Fermat-Steiner-Frechet multitree occurs with respect to a given $\frac{N(N+1)}{2}-$tuple of positive real numbers determining the edge lengths of incogruent $N-$simplexes generates by the Dekster-Wilker domain in $\mathbb{R}^{N}.$ Then $N+2,\ldots, m+1$ rays start to grow from the weighted Fermat-point $(A_{0})_{l},$ which creates a two-way mass transport, in order to obtain an $(m,k)$ mutation of an I.Fermat-Steiner-Frechet multitree having one weighted Fermat point of degree $m+1.$

\begin{theorem}[(m,k) "Mutation" of an intermediate weighted Fermat-Steiner Frechet multitree in $\mathbb{R}^{N}$]\label{mutationamultitrees}
The equations of (m,k)"mutation" of an intermediate weighted Fermat-Steiner Frechet multitree in $\mathbb{R}^{N}$ is derived by adding\\ $(A_{0})_{l}(A_{N+2})_{l},(A_{0})_{l}(A_{N+3})_{l},\ldots,(A_{0})_{l} (A_{m+1})_{l}$ rays to a given weighted Fermat-Frechet multitree with respect to a boundary Frechet $N-$multisimplex, such that $k$ masses (weights) are transferred in $k$ directions from $(A_{i})_{l}\to (A_{0})_{l},$ for $i=1,\ldots k$ (inflow) creating a storage at $(A_{0})_{l}$ and $N-k$ masses are transferred in $m-k$ directions from $(A_{0})_{l}\to (A_{j})_{l}$ (outflow), for $j=k+1,\ldots m+1$ and reversely $m-k$ masses (new weights) are transferred back to $(A_{0})_{l}$ along the same $m-k$ directions creating new storage at $(A_{0})_{l},$ and $k$ masses (new weights) are transferred back from $(A_{0})_{l}$ to $(A_{i})_{l},$ $i=1,2,\ldots k,$ $1<l \le \frac{\frac{1}{2}N(N+1)!}{(N+1)!}.$
\end{theorem}

\begin{proof}
By adding $(A_{0})_{l}(A_{N+2})_{l},(A_{0})_{l}(A_{N+3})_{l},\ldots,(A_{0})_{l} (A_{m+1})_{l}$ rays to a given weighted Fermat-Frechet multitree with respect to a boundary Frechet $N-$multisimplex in $\mathbb{R}^{N}$ and by applying Theorem~\ref{mutationmk}, we derive the $(m,k)$ plasticity equations of an I.Fermat-Frechet multitree for $m$ boundary closed polytopes in $\mathbb{R}^{N.}$
\end{proof}

\section{Constructive tree weights for the vertices of a Frechet $N-$multisimplex in $\mathbb{R}^{N}$}
In this section, we obtain an $\epsilon$ approximation of the value of the weight $B_{N+1},$ which corresponds to the vertex $A_{N+1}$ of an $N-$ simplex in $\mathbb{R}^{N}$ circumscribed in a $(N-1)$sphere $S^{N-1}$ of radius $r,$ and center $O,$ by applying the $N-$INWF problem in $\mathbb{R}^{N}.$ For $\epsilon \to 0,$ the limiting floating weighted Fermat tree solution coincides with the absorbing weighted Fermat tree solution of Theorem~\ref{theor1}.
An application of this method is an approximation for the weights of a Frechet $N-$multisimplex via an $\epsilon$ approximation of each corner of incongruent $N-$simplexes with a multiweighted Fermat-Frechet multitree in $\mathbb{R}^{N}.$

Let $A_{1}A_{2}\ldots A_{N+1}$ be an $N-$ simplex circumscribed in $S^{N-1}(0,r).$
We consider a point $A_{0}\in [A_{N+1},O],$ and we denote by $\epsilon=|A_{N+1}A_{0}|.$

\begin{theorem}\label{constructiveweights}
The weights $B_{1}(\epsilon),\ldots,B_{N+1}(\epsilon)$ are uniquely determined by

\begin{eqnarray}\label{inverse111nepsilon}
B_{i}(\epsilon)=\frac{C}{1+\abs{\frac{\sin{\alpha_{i,0k_{1}k_{2}\ldots k_{N-1}}}}{\sin{\alpha_{k_{N},0k_{1}k_{2}\ldots k_{N-1}}}}}+\abs{\frac{\sin{\alpha_{i,0k_{1}k_{2}\ldots k_{N-2}k_{N}}}}{\sin{\alpha_{k_{N-1},0k_{1}k_{2}\ldots k_{N-2}k_{N}}}}}+\ldots+\abs{\frac{\sin{\alpha_{i,0k_{2}\ldots k_{N-1}k_{N}}}}{\sin{\alpha_{k_{1},0k_{2}\ldots k_{N-1}k_{N}}}}}},\nonumber\\
\end{eqnarray}

for $i, k_{1},k_{2},...,k_{N}=1,2,...,N+1$ and $k_{1} \neq k_{2}\neq...\neq k_{N},$ and the weighted Fermat tree $\{A_{1}A_{0},A_{2}A_{0,},\ldots A_{N+1}A_{0}$ is an $\epsilon$ approximation of the absorbing
weighted Fermat tree $\{A_{1}A_{N+1},A_{2}A_{N+1},\ldots A_{N}A_{N+1}\},$ such that: $A_{0}\equiv A_{N+1},$ with a weighted error estimate for $B_{N+1}:$

\begin{eqnarray}\label{errorBnone}
|\sqrt{\sum_{i=1}^{N}B_{i}^{2}(\epsilon)+2\sum_{i,j=1,i<j}^{N}B_{i}(\epsilon)B_{j}(\epsilon)\cos\alpha_{i(N+1)j}}-B_{N+1}(\epsilon)|.
\end{eqnarray}

\end{theorem}

\begin{proof}
By applying the cosine law and the sine law in $\triangle A_{O}A_{i}A_{j},$ and $\triangle OA_{i}A_{0},$ we derive
that $\frac{N(N+1}{2}-1$ angles $\alpha_{i0j}=\alpha_{i0j}(\epsilon).$ By substituting these angles in \ref{inverse111n}, we obtain a unique solution of $B_{i}(\epsilon)$ given by (\ref{inverse111nepsilon}). By replacing $B_{i}(\epsilon)$ for $i=1,2,\ldots ,N$ in the weighted absorbing condition of Theorem~\ref{theor1}, we get an error estimate for $B_{N+1}$ (\ref{errorBnone}).

\end{proof}

\section{Bessel plasticity and "Mutation" of an intermediate weighted Fermat-Steiner-Frechet multitree for $m$ polytopes in $\mathbb{R}^{N}$}

In this section, we will describe the evolutionary structure of an intermediate weighted Fermat-Steiner-Frechet multitree for $m$ boundary polytopes in $\mathbb{R}^{N}$ with random weights following a Bessel motion.

We start by constructing a Bessel motion in the sense of McKean (\cite{McKean:60}, \cite{ItoMcKean:74}).
We consider the $m-$dimensional Brownian motion with sample paths $\mathbf{b}(t),$ ($m\ge 2, t\ge 0$) and generator $\mathcal{G}=\frac{1}{2}(\sum_{i=1}^{m}\frac{\partial ^2 }{\partial \mathbf{b}_{i}^{2}}).$
The radial part $\mathbf{r}(t)=|\mathbf{b}(t)|$ ($t \ge 0$) is the Bessel motion with generator $\mathcal{G}^{+}=\frac{1}{2}(\frac{d^{2}}{d \mathbf{r} ^2})+\frac{N-1}{\mathbf{r}}\frac{d}{d \mathbf{r}}.$
The probability of the event B $P_{\cdot}(B)$ as a function of the starting point $\mathbf{a}=\mathbf{b}(0)$ of the Brownian path if $t_{1}<t_{2}$ is given by:
\[P_{\cdot}([\mathbf{r}(t_{2})\le l | \mathbf{r}(s): s\le t_{1}])=\]\[=\int_{|\mathbf{b}-\mathbf{b}(t_{1})|<l}(2\pi (t_{2}-t_{1}))^{-m/2}e^{-\frac{|\mathbf{b}-\mathbf{b}(t_{1})|^2}{2(t_{2}-t_{1})}}d\mathbf{b}_{1}\ldots d\mathbf{b}_{m}.\]
The Bessel motion $[\mathbf{r}(t), P_{\cdot}]$ is a Markov process, which depends upon $\mathbf{r}(t_{1}).$

\begin{lemma}[Non-negative solutions of Bessel Motion]\cite[p.~318-319]{McKean:60}\label{lemimp1}
The solution of the singular integral equation
\begin{eqnarray}\label{positiverandomweight}
\mathbf{r}(t)=\mathbf{b}(t)+\frac{N-1}{2}\int_{1}^{t}\mathbf{r}^{-1}ds, t\ge 0,
\end{eqnarray}
is a Bessel motion starting at $\mathbf r(0)=\mathbf{b}(0),$ which takes non-negative real values by
neglecting a class of Brownian paths $\mathbf{b}$ of Wiener measure 0.
\end{lemma}

By taking into account the solution of the $N-$INVWF problem for polytopes in $\mathbb{R}^{N},$ we derive the
equations of Bessel plasticity of a weighted Fermat-tree for a boundary polytope in $\mathbb{R}^{N}.$ The Bessel plasticity of a weighted Fermat tree characterizes the combinatorial plasticity (random weights) of non-random polytopes in $\mathbb{R}^{N}.$

\begin{theorem}\label{Besselpasticitypolytopes}

The following equations point out the Bessel plasticity of an intermediate weighted Fermat-Steiner tree with one weighted Fermat point for a boundary $(N+1)$ weighted closed polytope with respect to the non-negative random weights
$(B_{i})_{12\ldots N+2},$ which depend on the Bessel motion $\mathbf{r}_{N+2}(t)\equiv (B_{N+2})_{12\ldots N+2}$ in $\mathbb{R}^{N}:$

\begin{eqnarray}\label{dynamicplasticity2bessel}
(\frac{B_{1}}{B_{N+1}})_{12\ldots N+2}=(\frac{B_{1}}{B_{N+1}})_{12\ldots N+1}(1-(\frac{\mathbf{r}_{N+2}(t)}{B_{N+1}})_{12\ldots (N+2)}(\frac{B_{N+1}}{B_{N+2}})_{2\ldots N+2})\nonumber\\
\end{eqnarray}
\begin{eqnarray}\label{dynamicplasticity3bessel}
(\frac{B_{2}}{B_{N+1}})_{12\ldots N+2}=(\frac{B_{2}}{B_{N+1}})_{12\ldots N+1}(1-(\frac{\mathbf{r}_{N+2}(t)}{B_{N+1}})_{12\ldots n+2}(\frac{B_{N+1}}{B_{N+2}})_{13\ldots N+2})\nonumber\\
\end{eqnarray}
$\vdots$
\begin{eqnarray}\label{dynamicplasticity1bessel}
(\frac{B_{N}}{B_{N+1}})_{12\ldots N+2}=(\frac{B_{N}}{B_{N+1}})_{12\ldots N+1}(1-(\frac{\mathbf{r}_{N+2}(t)}{B_{N+1}})_{12\ldots N+2}(\frac{B_{N+1}}{B_{N+2}})_{12\ldots (N-1)(N+1)(N+2)}),\nonumber\\
\end{eqnarray}

such that:

\[\sum_{i=1}^{n+2}(B_{i})_{12\ldots N+2}=c.\]

\end{theorem}

\begin{proof}
By inserting $\mathbf{r}_{N+2}(t)\equiv (B_{N+2})_{12\ldots N+2}$ into (\ref{dynamicplasticity2},
(\ref{dynamicplasticity3}), (\ref{dynamicplasticity1}), we obtain (\ref{dynamicplasticity2bessel},
(\ref{dynamicplasticity3bessel}), (\ref{dynamicplasticity1bessel}).
\end{proof}

Suppose that at time $t=0,$ an I.Fermat-Steiner-Frechet multitree occurs with respect to a given $\frac{N(N+1)}{2}-$tuple of positive real numbers determining the edge lengths of incogruent $N-$simplexes generates by the Dekster-Wilker domain in $\mathbb{R}^{N}.$ Then an $(N+2)th,$ ray start to grow from the weighted Fermat-point $(A_{0})_{l}.$ If the weight $(B_{N=2})_{l}$ follows a Bessel motion, we derive the Bessel plasticity of a weighted multitree for a boundaty $(N+1)$ closed polytope in $\mathbb{R}^{N.}$

\begin{theorem}[Bessel plasticity of an I.Fermat-Steiner-Frechet multitree for an $N+1$ boundary closed polytope in $\mathbb{R}^{N}$]
The Bessel plasticity equations of an I.Fermat-Steiner-Frechet multitree for an $N+1$ boundary closed polytope in $\mathbb{R}^{N}$ are given by the Bessel plasticity equations of a weighted Fermat tree, such the random weight that corresponds to the $(N+2)$th ray follows a Bessel motion.
\end{theorem}

\begin{proof}
By applying Theorem~\ref{Besselpasticitypolytopes} for an I.Fermat-Steiner-Frechet multitree for an $N+1$ boundary closed polytope in $\mathbb{R}^{N}$
which is derived by adding a ray $(A_{0}A_{N+2})$ with a random weight following a Bessel motion to I.Fermat-Steiner-Frechet multitree for an $N$ boundary Frechet multisimplex in $\mathbb{R}^{N},$ we get the Bessel plasticity equations of the generated I.Fermat-Steiner-Frechet multitree in $\mathbb{R}^{N}.$
\end{proof}

\begin{remark}
We note that the Bessel plasticity of the I.Fermat-Steiner-Frechet multitree for an $N+1$ boundary closed polytope in $\mathbb{R}^{N}$ may csuse a distortion of the length structure of the initial I.Fermat-Steiner-Frechet multitree for an $N$ boundary Frechet multisimplex in $\mathbb{R}^{N}.$
\end{remark}

\section{Open questions}
In this final section, we mention two open questions, which deal with the detection of the most natural of natural numbers.

1. How can we detect the most natural of consecutive $\frac{N(N+1}{2}$ natural numbers, such that an unweighted Fermat-Steiner tree
corresponds to a boundary $N-$simplex having the maximum volume with edge lengths this $\frac{N(N+1}{2}$tuple of natural numbers for $N\ge 3$ by using unweighted Fermat-Steiner points with weight $b_{ST}=1$?
2. How can we detect the $\frac{N(N+1}{2}-$tuples from the topology structure of intermediate unweighted Fermat Steiner Frechet multitrees for consecutive $\frac{N(N+1}{2}-$ natural numbers determining the edge lengths of a Frechet $N-$multisimplex in $\mathbb{R}^{N},$ which correspond to an intermediate unweighted Fermat-Steiner tree having the $ith$ maximal volume?

\end{document}